\documentclass[a4paper,12pt]{article}

\usepackage{amsmath,amsfonts,amssymb}
\usepackage{amsthm}
\usepackage{color,multicol}

\newtheorem{lemma}{Lemma}[section]
\newtheorem{theorem}[lemma]{Theorem}

\newtheorem{remark}[lemma]{Remark}

\begin{document}

\title{On negative eigenvalues of two-dimensional Schr\"odinger operators}

\author{Eugene Shargorodsky\footnote{E-mail: \ eugene.shargorodsky@kcl.ac.uk} \\
Department of Mathematics, King's College London,\\
Strand, London WC2R 2LS, UK}

\date{}

\maketitle

\begin{abstract}
The paper presents estimates for the number of negative eigenvalues of a two-dimensional 
Schr\"odinger operator in terms of $L\log L$ type Orlicz norms of the potential and proves a
conjecture by N.N. Khuri, A. Martin and T.T. Wu.
\end{abstract}

\section{Introduction}

According to the celebrated Cwikel-Lieb-Rozenblum inequality, the number $N_- (V)$
of negative eigenvalues of the Schr\"odinger operator 
$-\Delta - V$, $V \ge 0$ on $L_2(\mathbb{R}^d)$, $d \ge 3$ is estimated above by
$$
\mbox{const}\, \int_{\mathbb{R}^d} V(x)^{d/2} dx .
$$ 
It is well known that this estimate does not hold for $d = 2$. In this case, the Schr\"odinger operator 
has at least one negative eigenvalue for any nonzero $V \ge 0$, and no estimate of the type 
$$
N_- (V) \le \mbox{const} + \int_{\mathbb{R}^2} V(x)  W(x)\, dx
$$
can hold, provided the weight function $W$ is bounded in a neighborhood of at least 
one point (see \cite{Grig}). On the other hand,
$$
N_- (V) \ge  \mbox{const}\, \int_{\mathbb{R}^2} V(x)\, dx 
$$
(see \cite{GNY}). Upper estimates for $N_- (V)$ in the case $d = 2$ can be found in
\cite{BL, BS, CKMW, Grig, KMW, LS0, LS, MV, MV1, Sol, Sol1, Sto} and in the references therein. 
Following the pioneering paper \cite{Sol}, we obtain upper estimates involving 
$L\log L$ type Orlicz norms of the potential (Theorems \ref{clCLRth}, \ref{GrigTalk} anf 
\ref{LaptNetrSol}) 
and prove (Theorem \ref{KMWth})  a 
conjecture by N.N. Khuri, A. Martin and T.T. Wu (\cite{KMW}). 
In fact, we show that the main estimate in \cite{Sol} is actually stronger than the one
conjectured in \cite{KMW}, while none of the later results in \cite{BL, Grig, LN, LS, MV, MV1, Sto} 
implies the Khuri-Martin-Wu inequality (see Section \ref{Compar} below).   
Our approach does not
seem to be sufficient to settle the stronger conjecture by K. Chadan,
N.N. Khuri, A. Martin and T.T. Wu (\cite{CKMW}; see \eqref{CKMWest}, \eqref{eqs}).

We discuss several upper estimates for $N_- (V)$ in the paper and show in Section \ref{Compar}
how they are related to each other (see the diagram close to the end of Section \ref{Compar}). 
All of them involve terms of two types: integrals of $V$ with a logarithmic weight and
$L\log L$ type (or $L_p$, $p > 1$) norms of $V$. It turns out that the finiteness of the term of the first
type suggested in \cite{Sol} is necessary for 
$N_- (\alpha V) = O\left(\alpha\right)$  as  $\alpha \to +\infty$ to hold (see Theorem \ref{lognecesnew}
below). Unfortunately this is not true
for the terms of the second type.
None of the estimates discussed in the paper is sharp in the sense that $N_- (V)$ has to be infinite if the
right-hand side is infinite. On the other hand, both the
logarithmic weight and the $L\log L$ norm are in a sense optimal if one tries to estimate
$N_- (V)$ in terms of weighted integrals of $V$ and of its Orlicz norms 
(see \cite[Section 4]{Sol} and Section \ref{Concl} below). 
It is probably difficult to obtain an estimate for $N_- (V)$ that is sharp in the above
sense. Indeed, there
are potentials $V \ge 0$ such that $N_- (\alpha V) < \infty$ for $\alpha < 1$ and
$N_- (\alpha V) = \infty$ for $\alpha > 1$ (this is true in the multidimensional case $d \ge 3$
as well). For such potentials, $N_- (V)$
may be finite or infinite, and in the latter case, 
 $N_- (\alpha V)$ may grow arbitrarily fast or arbitrarily slow as $\alpha \to 1 - 0$ (see
 Theorem \ref{alpha1} and Appendix B).

\section{Notation and auxiliary results}\label{App} 

We need some notation from the theory of Orlicz  spaces (see \cite{KR, RR}).
Let $(\Omega, \Sigma, \mu)$ be a measure space, let $\Phi$ and
$\Psi$ be mutually complementary $N$-functions, and let $L_\Phi(\Omega)$, 
$L_\Psi(\Omega)$ be the corresponding Orlicz spaces. (These spaces are denoted by 
$L^*_\Phi(\Omega)$,  $L^*_\Psi(\Omega)$ in \cite{KR}, where $\Omega$ is assumed to be a closed
bounded subset of $\mathbb{R}^d$
equipped with the standard Lebesgue measure.) We will use the following
norms on $L_\Psi(\Omega)$ 
\begin{equation}\label{Orlicz}
\|f\|_{\Psi} = \|f\|_{\Psi, \Omega} = \sup\left\{\left|\int_\Omega f g d\mu\right| : \ 
\int_\Omega \Phi(g) d\mu \le 1\right\}
\end{equation}
and
\begin{equation}\label{Luxemburg}
\|f\|_{(\Psi)} = \|f\|_{(\Psi, \Omega)} = \inf\left\{\kappa > 0 : \ 
\int_\Omega \Psi\left(\frac{f}{\kappa}\right) d\mu \le 1\right\} .
\end{equation}
These two norms are equivalent
\begin{equation}\label{Luxemburgequiv}
\|f\|_{(\Psi)} \le \|f\|_{\Psi} \le 2 \|f\|_{(\Psi)}\, , \ \ \ \forall f \in L_\Psi(\Omega) .
\end{equation}
Note that
\begin{equation}\label{LuxNormImpl}
\int_\Omega \Psi\left(\frac{f}{\kappa_0}\right) d\mu \le C_0, \ \ C_0 \ge 1  \ \ \Longrightarrow \ \
\|f\|_{(\Psi)} \le C_0 \kappa_0 . 
\end{equation}
Indeed, since $\Psi$ is even, convex and increasing on 
$[0, +\infty)$, and $\Psi(0) = 0$, we get for any $\kappa \ge C_0 \kappa_0$,
\begin{equation}\label{LuxProof}
\int_{\Omega} \Psi\left(\frac{f}{\kappa}\right) d\mu \le
\int_{\Omega} \Psi\left(\frac{f}{C_0 \kappa_0}\right) d\mu \le
\frac{1}{C_0} \int_{\Omega} \Psi\left(\frac{f}{\kappa_0}\right) d\mu \le 1 .
\end{equation}
It follows from \eqref{LuxNormImpl} with $\kappa_0 = 1$ that 
\begin{equation}\label{LuxNormPre}
\|f\|_{(\Psi)} \le \max\left\{1, \int_{\Omega} \Psi(f) d\mu\right\} .  
\end{equation}

We will need the following
equivalent norm on $L_\Psi(\Omega)$ with $\mu(\Omega) < \infty$ which was introduced in
\cite{Sol}:
\begin{equation}\label{OrlAverage}
\|f\|^{\rm (av)}_{\Psi} = \|f\|^{\rm (av)}_{\Psi, \Omega} = \sup\left\{\left|\int_\Omega f g d\mu\right| : \ 
\int_\Omega \Phi(g) d\mu \le \mu(\Omega)\right\} .
\end{equation}

\begin{lemma}\label{avequiv}
$$
\min\{1, \mu(\Omega)\}\, \|f\|_{\Psi, \Omega} \le \|f\|^{\rm (av)}_{\Psi, \Omega}
\le \max\{1, \mu(\Omega)\}\, \|f\|_{\Psi, \Omega}
$$
\end{lemma}
\begin{proof}
Let
$$
\mathbb{B}_1 := \left\{g : \ \int_\Omega \Phi(g) d\mu \le 1\right\} , \ \ \ 
\mathbb{B}_\Omega := \left\{g : \ \int_\Omega \Phi(g) d\mu \le \mu(\Omega)\right\} .
$$
Suppose $\mu(\Omega) \ge 1$. Then, clearly, 
$\|f\|_{\Psi, \Omega} \le \|f\|^{\rm (av)}_{\Psi, \Omega}$.  It is easy to see that
$$
g \in \mathbb{B}_\Omega \ \Longrightarrow \ \frac1{\mu(\Omega)}\, g \in \mathbb{B}_1 
$$
(cf. \eqref{LuxProof}). Hence 
$$
\|f\|^{\rm (av)}_{\Psi, \Omega} = \sup_{g \in \mathbb{B}_\Omega}
\left|\int_\Omega f g d\mu\right| \le \sup_{h \in \mathbb{B}_1}
\left|\int_\Omega f \cdot\big(\mu(\Omega) h\big) d\mu\right| = \mu(\Omega)
\|f\|_{\Psi, \Omega} .
$$

Suppose now $\mu(\Omega) < 1$. Then 
$\|f\|^{\rm (av)}_{\Psi, \Omega} \le \|f\|_{\Psi, \Omega}$ and 
$$
g \in \mathbb{B}_1 \ \Longrightarrow \ \mu(\Omega)\, g \in \mathbb{B}_\Omega .
$$
Hence,
$$
\mu(\Omega) \|f\|_{\Psi, \Omega} = \mu(\Omega) \sup_{g \in \mathbb{B}_1}
\left|\int_\Omega f g d\mu\right| \le \sup_{h \in \mathbb{B}_\Omega}
\left|\int_\Omega f  h d\mu\right| = \|f\|^{\rm (av)}_{\Psi, \Omega} .
$$
\end{proof}

We will need the following pair of mutually complementary $N$-fuctions
\begin{equation}\label{thepair}
\mathcal{A}(s) = e^{|s|} - 1 - |s| , \ \ \ \mathcal{B}(s) = (1 + |s|) \ln(1 + |s|) - |s| , \ \ \ s \in \mathbb{R} .
\end{equation}
We will use the following standard notation
\begin{equation}\label{a+}
a_+ := \max\{0, a\}, \ \ \ a \in \mathbb{R} .
\end{equation}

\begin{lemma}\label{elem}
$\frac12\, s\ln_+ s \le \mathcal{B}(s) \le s + 2s\ln_+ s$, \ $\forall s \ge 0$.
\end{lemma}
\begin{proof}
If $s \ge 1$, then integrating the inequality 
$$
1 + \ln t = \ln (e t) < \ln(1 + 3t) \le 2\ln(1 + t) ,  \ \ \ t \ge 1, 
$$
over the interval $[1, s]$,
one gets $s \ln s \le 2(\mathcal{B}(s) - 2\ln 2 + 1) < 2\mathcal{B}(s)$. 
Hence $s \ln_+ s \le 2\mathcal{B}(s)$, $\forall s \ge 0$.

If $s \ge 1$, then 
\begin{eqnarray*}
\mathcal{B}(s) = (1 + s) \ln(1 + s) - s \le 2s \ln(2s) - s = (2\ln 2 - 1)s + 2s\ln s \\
< s + 2s\ln s .
\end{eqnarray*}
If $s \in [0, 1)$, then integrating the inequality $\ln(1 + t) \le t$ over the interval $[0, s]$, one gets
$$
\mathcal{B}(s) = (1 + s) \ln(1 + s) - s \le \frac12\, s^2 \le \frac12\, s \le s .
$$
\end{proof}

\begin{lemma}\label{ABelem}
$\mathcal{B}(s) \le \frac12\, s^2$, \ $\frac12\, s^2 \le \mathcal{A}(s) \le \frac{e}2\, s^2$, \ 
$\forall s \in [0, 1]$.
\end{lemma}
\begin{proof}
The first inequality was proved at the end of the proof of Lemma \ref{elem}, the second one
is obtained by integrating the inequality $1 \le e^s \le e$, \ $s \in [0, 1]$ twice.
\end{proof}

\begin{lemma}\label{Aelem}
$e^s \le 2\mathcal{A}(s) + \frac32$, \ $\forall s \ge 0$.
\end{lemma}
\begin{proof}
$$
s \le \frac12 + \frac{s^2}2 \le \frac12 + \mathcal{A}(s) .
$$
Hence 
$$
e^s = \mathcal{A}(s) + 1 + s \le 2\mathcal{A}(s) + \frac32, \ \  \forall s \ge 0 .
$$
\end{proof}

\begin{lemma}\label{avequivB}
Let $\mu(\Omega) > 1$. Then
$$
\|f\|^{\rm (av)}_{\mathcal{B}, \Omega}
\le \|f\|_{\mathcal{B}, \Omega} + \ln\left(\frac72\, \mu(\Omega)\right)\, \|f\|_{L_1(\Omega, \mu)} .
$$
\end{lemma}
\begin{proof}
Since $\|f\|^{\rm (av)}_{\mathcal{B}, \Omega} = \||f|\|^{\rm (av)}_{\mathcal{B}, \Omega}$ and
$\|f\|_{\mathcal{B}, \Omega} = \||f|\|_{\mathcal{B}, \Omega}$, we can assume without loss of generality 
that $f \ge 0$. In this case, 
\begin{eqnarray*}
&& \|f\|_{\mathcal{B}, \Omega} = \sup\left\{\int_\Omega f g\, d\mu : \ g \ge 0 , \
\int_\Omega \mathcal{A}(g) d\mu \le 1\right\} , \\
&& \|f\|^{\rm (av)}_{\mathcal{B}, \Omega} = \sup\left\{\int_\Omega f g\, d\mu : \ g \ge 0 , \
\int_\Omega \mathcal{A}(g) d\mu \le \mu(\Omega)\right\} .
\end{eqnarray*}
Take an arbitrary $g \ge 0$ with $\int_\Omega \mathcal{A}(g) d\mu \le \mu(\Omega)$ and set
\begin{eqnarray*}
& \chi(x) := \left\{\begin{array}{l}
  1 \ \mbox{ if } g(x) \le  \ln\left(\frac72\, \mu(\Omega)\right),   \\ \\
   0 \  \mbox{ if } g(x) >  \ln\left(\frac72\, \mu(\Omega)\right),  
\end{array}\right. \\
& g_1 := g\chi + \ln\left(\frac72\, \mu(\Omega)\right)\, (1 - \chi) , \ \ g_2 := g - g_1 
= \big(g - \ln\left(\frac72\, \mu(\Omega)\right)\big) (1 - \chi) . 
\end{eqnarray*}
Then $g = g_1 + g_2$, $0 \le g_1 \le \ln\left(\frac72\, \mu(\Omega)\right)$, 
and it follows from Lemma \ref{Aelem} that
\begin{eqnarray*}
&& \int_\Omega \mathcal{A}(g_2) d\mu = 
\int_{g(x) >  \ln\left(\frac72\, \mu(\Omega)\right)} 
\mathcal{A}\left(g(x) - \ln\left(\frac72\, \mu(\Omega)\right)\right) d\mu(x)\\
&& \le \int_{g(x) >  \ln\left(\frac72\, \mu(\Omega)\right)} 
e^{g(x) - \ln\left(\frac72\, \mu(\Omega)\right)} d\mu(x) \le
\frac{2}{7 \mu(\Omega)} \int_\Omega e^g d\mu \\
&& \le \frac{2}{7 \mu(\Omega)} \int_\Omega \left(2\mathcal{A}(g) + \frac32\right) d\mu \le
\frac{2}{7 \mu(\Omega)} \left(2\mu(\Omega) + \frac32\, \mu(\Omega)\right) = 1 .
\end{eqnarray*}
Hence
\begin{eqnarray*}
\int_\Omega f g\, d\mu = \int_\Omega f g_1\, d\mu + \int_\Omega f g_2\, d\mu 
\le \ln\left(\frac72\, \mu(\Omega)\right) \int_\Omega f\, d\mu +  \|f\|_{\mathcal{B}, \Omega} \\
= \|f\|_{\mathcal{B}, \Omega} + \ln\left(\frac72\, \mu(\Omega)\right)\, \|f\|_{L_1(\Omega, \mu)} .
\end{eqnarray*}
\end{proof}

\section{A Solomyak type estimate}

Let $\mathcal{H}$ be a Hilbert space and let $\mathbf{q}$ be a Hermitian form with a domain
$\mbox{Dom}\, (\mathbf{q}) \subseteq \mathcal{H}$. Set
\begin{equation}\label{N-}
N_- (\mathbf{q}) := \sup\left\{\dim \mathcal{L}\, | \  \mathbf{q}[u] < 0, \ 
\forall u \in \mathcal{L}\setminus\{0\}\right\} ,
\end{equation}
where $\mathcal{L}$ denotes a linear subspace of $\mbox{Dom}\, (\mathbf{q})$. If $\mathbf{q}$
is the quadratic form of a self-adjoint operator $A$ with no essential spectrum in $(-\infty, 0)$, then
by the variational principle,
$N_- (\mathbf{q})$ is the number of negative eigenvalues of $A$ repeated according to their
multiplicity (see, e.g., \cite[S1.3]{BerShu} or \cite[Theorem 10.2.3]{BirSol}).

Let $V \ge 0$ be locally integrable on $\mathbb{R}^2$. Consider
\begin{eqnarray*}
\mathcal{E}_V[w] := \int_{\mathbb{R}^2} |\nabla w(x)|^2 dx - 
\int_{\mathbb{R}^2} V(x) |w(x)|^2 dx ,  \\  
\mbox{Dom}\, (\mathcal{E}_V) :=
W^1_2\left(\mathbb{R}^2\right)\cap L_2\left(\mathbb{R}^2, V(x)dx\right) .  
\end{eqnarray*}

\begin{theorem}\label{clCLRth}
There exists a constant $C > 0$ such that 
\begin{equation}\label{clCLRest}
N_- (\mathcal{E}_V) \le C  \left(\|V\|_{\mathcal{B}, \mathbb{R}^2} 
+ \int_{\mathbb{R}^2} V(x) \ln(1 + |x|)\, dx\right)  + 1 , \ \ \ 
\forall V\ge 0 .
\end{equation}
\end{theorem}
\begin{proof}
Let
\begin{eqnarray*}
&& \mathcal{E}_{V, m}[w] : = \int_{\Omega_m} |\nabla w(x)|^2 dx - 
\int_{\Omega_m} V(x) |w(x)|^2 dx ,  \\  
&& \mbox{Dom}\, (\mathcal{E}_{V, m}) =
W^1_2\left(\Omega_m\right)\cap L_2\left(\Omega_m, V(x)dx\right) , \ \ m = 1, 2 , 
\end{eqnarray*}
where
$$
\Omega_1 = B(0, 1) = \{x \in \mathbb{R}^2 : \ |x| < 1\} , \ \ 
\Omega_2 = \mathbb{R}^2\setminus \overline{B(0, 1)} = \{x \in \mathbb{R}^2 : \ |x| > 1\} .
$$
Then by the variational principle,
\begin{equation}\label{addit}
N_- (\mathcal{E}_V) \le N_- (\mathcal{E}_{V, 1}) + N_- (\mathcal{E}_{V, 2})
\end{equation}
(see, e.g., \cite[Lemma 3.5]{Grig}).

There exists an independent of $V$ constant $C_1 > 0$ such that
\begin{equation}\label{U1}
N_- (\mathcal{E}_{V, 1}) \le C_1 \|V\|_{\mathcal{B}, \Omega_1} + 1 
\end{equation}
(see \cite[Theorems 4 and $4'$, and Proposition 3]{Sol}).

Below, we use the complex notation $z  = x_1 + i x_2$ alongside the real one $x = (x_1, x_2)$.
Let
$$
\hat{V}(z) := \frac{1}{|z|^4} V\left(\frac{1}{z}\right) , \ \ \ \
\tilde{w}(z) := w(1/z) , \ \ w \in W^1_2\left(\Omega_2\right) ,    
\ \ \ \ |z| < 1 
$$
(cf. the proof of Proposition 5.3 in \cite{GN}). An easy calculation gives
\begin{eqnarray*}
&& \int_{\Omega_2} |\nabla w(y)|^2 dy = \int_{\Omega_1} |\nabla \tilde{w}(x)|^2 dx , \\
&& \int_{\Omega_2} V(y) |w(y)|^2 dy = \int_{\Omega_1} \hat{V}(x) |\tilde{w}(x)|^2 dx ,
\end{eqnarray*}
and it follows from \eqref{U1} that
\begin{equation}\label{U2}
N_- (\mathcal{E}_{V, 2}) \le C_1 \left\|\hat{V}\right\|_{\mathcal{B}, \Omega_1} + 1 . 
\end{equation}
Let us estimate the norm in the right-hand side. It is more convenient to work with the
Luxemburg (gauge) norm \eqref{Luxemburg}. Using the notation $\zeta = y_1 + iy_2 = 
1/z$, we get for any $\kappa > 0$
\begin{eqnarray*}
\int_{\Omega_1} \mathcal{B}\left(\frac{\hat{V}(z)}{\kappa}\right) dx =
\int_{\Omega_1} \mathcal{B}\left(\frac{V(1/z)}{\kappa |z|^4}\right) dx =
\int_{\Omega_2} \mathcal{B}\left(\frac{|\zeta|^4 V(\zeta)}{\kappa}\right) 
\frac{1}{|\zeta|^4}\, dy \\
= \int_{\Omega_2} \left(\Big(1 + \frac{1}{\kappa}\, |y|^4 V(y)\Big)
\ln\Big(1 + \frac{1}{\kappa}\, |y|^4 V(y)\Big) - \frac{1}{\kappa}\, |y|^4 V(y)\right) 
\frac{1}{|y|^4}\, dy \\
\le \int_{\Omega_2} \left(\Big(1 + \frac{1}{\kappa}\, |y|^4 V(y)\Big)
\ln\Big(1 + \frac{1}{\kappa}\, V(y)\Big) - \frac{1}{\kappa}\, |y|^4 V(y)\right) 
\frac{1}{|y|^4}\, dy \\
+ \int_{\Omega_2} 
\frac{1}{|y|^4}\, \Big(1 + \frac{1}{\kappa}\, |y|^4 V(y)\Big)
\ln\big(1 +  |y|^4\big)dy \\
\le \int_{\Omega_2} \left(\Big(1 + \frac{1}{\kappa}\, V(y)\Big)
\ln\Big(1 + \frac{1}{\kappa}\, V(y)\Big) - \frac{1}{\kappa}\,  V(y)\right)\, dy \\
+ \frac{1}{\kappa} \int_{\Omega_2} 
V(y) \ln\big(1 +  |y|^4\big)dy + \int_{\Omega_2} 
\frac{1}{|y|^4}\, \ln\big(1 +  |y|^4\big)dy \\
< \int_{\Omega_2} \mathcal{B}\left(\frac{V(y)}{\kappa}\right) dy
+ \frac{4}{\kappa} \int_{\Omega_2} 
V(y) \ln\big(1 +  |y|\big)dy \\
+ 4 \int_{\Omega_2} 
\frac{1}{|y|^4}\, \ln\big(1 +  |y|\big)dy . 
\end{eqnarray*}
Let
$$
\kappa_0 := \max\left\{\|V\|_{(\mathcal{B}, \Omega_2)},
\int_{\Omega_2} 
V(y) \ln\big(1 +  |y|\big)dy\right\} .
$$
Then
\begin{eqnarray*}
\int_{\Omega_1} \mathcal{B}\left(\frac{\hat{V}(z)}{\kappa_0}\right) dx <
1 + 4 + 4 \int_{\Omega_2} |y|^{-3}\, dy = 5 + 8\pi =: C_2 , 
\end{eqnarray*} 
and
$$
\left\|\hat{V}\right\|_{(\mathcal{B}, \Omega_1)} \le C_2 \max\left\{\|V\|_{(\mathcal{B}, \Omega_2)},
\int_{\Omega_2}  V(y) \ln\big(1 +  |y|\big)dy\right\} 
$$
(see \eqref{LuxNormImpl}). 

Now it follows from \eqref{addit}--\eqref{U2} that
\begin{equation}\label{with2}
N_- (\mathcal{E}_V) \le C_3  \left(\|V\|_{\mathcal{B}, \mathbb{R}^2} 
+ \int_{\mathbb{R}^2} V(x) \ln(1 + |x|)\, dx\right)  + 2 ,
\end{equation}
where one can take $C_3 = C_1(1 + 2C_2)$. The last inequality gives the estimate
$N_- (\mathcal{E}_V) \le 2$ for small $V$ and it is left to show that one actually has
$N_- (\mathcal{E}_V) = 1$ in this case.

According to Proposition 4.4 in \cite{GN}, $N_- (\mathcal{E}_V) = 1$ provided
\begin{equation}\label{intKV}
\sup_{x' \in \mathbb{R}^2} \int_{\mathbb{R}^2} K(x, x') V(x)\, dx
\end{equation}
is sufficiently small, where
\begin{eqnarray*}
&& K(x, x') := \ln(2 + |x|) + K_0(x, x') , \\
&& K_0(x, x') := \ln_+\frac{1}{|x - x'|} , \ \ \ 
x, x' \in \mathbb{R}^2 , \ x \not= x'
\end{eqnarray*}
(see \eqref{a+}).
It follows from \eqref{normlog} (see below) that
$$
\sup_{x' \in \mathbb{R}^2} \|K_0(\cdot, x')\|_{(\mathcal{A}, \mathbb{R}^2)} =
\|K_0(\cdot, 0)\|_{(\mathcal{A}, \mathbb{R}^2)} \le 2\pi .
$$
Using the
H\"older inequality for Orlicz  spaces (see \cite[\S 3.3, (17)]{RR}), one gets
\begin{eqnarray*}
\int_{\mathbb{R}^2} K(x, x') V(x)\, dx \le \int_{\mathbb{R}^2} V(x) \ln(2 + |x|)\, dx +
\|K_0(\cdot, x')\|_{(\mathcal{A}, \mathbb{R}^2)} \|V\|_{\mathcal{B}, \mathbb{R}^2} \\
\le \mbox{const } \left(\int_{\mathbb{R}^2} V(x) \ln(1 + |x|)\, dx 
+ \|V\|_{\mathcal{B}, \mathbb{R}^2}\right) 
\end{eqnarray*}
with a constant independent of $x'$ and $V$.
Hence \eqref{intKV} can be made arbitrarily small by making $\|V\|_{\mathcal{B}, \mathbb{R}^2} +
\|V\|_{L_1(\mathbb{R}^2,\, \ln(1 + |x|)\, dx)}$ sufficiently small. In this case,
$N_- (\mathcal{E}_V) = 1$. Combining this 
with \eqref{with2} one gets the existence of a constant $C > 0$ for which 
\eqref{clCLRest} holds.
\end{proof}

\begin{remark}\label{MVSol}
{\rm Estimate \eqref{clCLRest} looks similar to the following one obtained in \cite{MV1}:
\begin{equation}\label{MVest}
N_- (\mathcal{E}_V) \le C  \left(\int_{V(x) \ge 1} V(x) \ln V(x)\, dx
+ \int_{\mathbb{R}^2} V(x) \ln(2 + |x|)\, dx\right)  + 1 .
\end{equation}
An advantage of \eqref{clCLRest} is that its right-hand side agrees with the semi-classical 
asymtotics $N_- (\mathcal{E}_{\alpha V}) = O(\alpha)$ as $\alpha \to +\infty$. In fact,
Theorems \ref{clCLRth} and \ref{GrigTalk} are direct descendants of 
\eqref{SolMainEst} (see below) which was obtained in \cite{Sol}. 
It turns out that the right-hand side of \eqref{clCLRest} 
dominates that of \eqref{SolMainEst}, i.e. that \eqref{SolMainEst}
is actually a stronger estimate than \eqref{clCLRest} (see Section \ref{Compar}).}
\end{remark}

\section{The Khuri-Martin-Wu conjecture}

Let $V_* : \mathbb{R}_+ \to [0, +\infty]$ be the non-increasing spherical rearrangement
of $V$, i.e. let $V_*$ be non-increasing, continuous from the right and such that
$$
\left|\{x \in \mathbb{R}^2 : V_*(|x|) > s\}\right| = \left|\{x \in \mathbb{R}^2 : V(x) > s\}\right| , 
\ \ \ \forall s > 0 ,
$$
where $|E|$ denotes the two dimensional Lebesgue measure of $E \subset \mathbb{R}^2$.
(Note that $V_*(r) = V^*(\pi r^2)$, where $V^*$ is the standard non-increasing rearrangement 
used in the theory of Lorentz spaces; see, e.g., \cite[1.8]{Zie}.) Then
\begin{equation}\label{equal}
\int_{\mathbb{R}^2} F(V_*(|x|)) dx = \int_{\mathbb{R}^2} F(V(x)) dx
\end{equation}
for any measurable $F \ge 0$ (see, e.g., \cite[Ch. X, (1.10)]{RR}).

\begin{lemma}\label{Orliczlem}{\rm (Cf. \cite{Ben1})} 
There exists a constant $C_4 > 0$ such that 
\begin{equation}\label{Orliczest}
\|V\|_{\mathcal{B}, \mathbb{R}^2} \le C_4 \left(\|V\|_{L_1(\mathbb{R}^2)}
+ \int_{\mathbb{R}^2} V_*(|x|) \ln_+ \frac{1}{|x|}\, dx\right)   , \ \ \ 
\forall V\ge 0 .
\end{equation}
\end{lemma}
\begin{proof}
It follows from the devinition of $V_*$ that
\begin{eqnarray*}
\left|\{x \in \mathbb{R}^2 : V(x) \ge V_*(r)\}\right| = 
\left|\{x \in \mathbb{R}^2 : V_*(|x|) \ge V_*(r)\}\right| \\
\ge \left|\{x \in \mathbb{R}^2 : |x| \le r\}\right| = \pi r^2 , \ \ \ \forall r > 0 .
\end{eqnarray*}
On the other hand, Chebyshev's inequality implies
$$
\left|\{x \in \mathbb{R}^2 : V(x) \ge V_*(r)\}\right| \le 
\frac{\|V\|_{L_1(\mathbb{R}^2)}}{V_*(r)}\, .
$$
Hence
\begin{equation}\label{star}
V_*(r) \le \frac{\|V\|_{L_1(\mathbb{R}^2)}}{\pi r^2}\, .
\end{equation}

Similarly to the proof of Theorem \ref{clCLRth}, it will be convenient for us to work with the
norm \eqref{Luxemburg}. Let
$$
\kappa_1 := \max\left\{\frac{\|V\|_{L_1(\mathbb{R}^2)}}{\pi}\, , \
\int_{\mathbb{R}^2} V_*(|x|) \ln_+ \frac{1}{|x|}\, dx\right\} .
$$
Then it follows from Lemma \ref{elem} and \eqref{equal}, \eqref{star} that
\begin{eqnarray*}
&& \int_{\mathbb{R}^2} \mathcal{B}\left(\frac{V(x)}{\kappa_1}\right) dx \le
\frac1{\kappa_1} \int_{\mathbb{R}^2} V(x) dx + 
\frac2{\kappa_1} \int_{\mathbb{R}^2} V(x) \ln_+\frac{V(x)}{\kappa_1}\, dx \\
&& \le \pi + \frac2{\kappa_1} \int_{\mathbb{R}^2} V_*(|x|) \ln_+\frac{V_*(|x|)}{\kappa_1}\, dx
\le \pi + \frac2{\kappa_1} \int_{\mathbb{R}^2} 
V_*(|x|) \ln_+\frac{\|V\|_{L_1(\mathbb{R}^2)}}{\kappa_1\pi |x|^2}\, dx \\
&& \le \pi + \frac2{\kappa_1} \int_{\mathbb{R}^2} 
V_*(|x|) \ln_+\frac{1}{|x|^2}\, dx \le \pi + 4 =: C_5 .
\end{eqnarray*}
Hence
$$
\|V\|_{(\mathcal{B}, \mathbb{R}^2)} \le C_5 \max\left\{\frac{\|V\|_{L_1(\mathbb{R}^2)}}{\pi}\, , \
\int_{\mathbb{R}^2} V_*(|x|) \ln_+ \frac{1}{|x|}\, dx\right\} 
$$
(see \eqref{LuxNormImpl}), 
which implies \eqref{Orliczest} with $C_4 = 2C_5$ (see \eqref{Luxemburgequiv}).
\end{proof}

\begin{lemma}\label{Orliczlem2}{\rm (Cf. \cite{Ben1})} 
$$
\int_{\mathbb{R}^2} V_*(|x|) \ln_+ \frac{1}{|x|}\, dx \le 
4\pi \|V\|_{\mathcal{B}, \mathbb{R}^2}  , \ \ \ \forall V\ge 0 .
$$
\end{lemma}
\begin{proof}
Let $\mathbb{D} := \{x \in \mathbb{R}^2 : \ |x| \le 1\}$.
The H\"older inequality (see \cite[\S 3.3, (4)]{RR}) and 
\eqref{equal} imply
\begin{eqnarray*}
\int_{\mathbb{R}^2} V_*(|x|) \ln_+ \frac{1}{|x|}\, dx = 
\int_{\mathbb{D}} V_*(|x|) \ln \frac{1}{|x|}\, dx 
\le 2 \|V_*(|\cdot|)\|_{(\mathcal{B}, \mathbb{D})} 
\left\|\ln \frac{1}{|\cdot|}\right\|_{(\mathcal{A}, \mathbb{D})} \\
\le 2 \|V_*(|\cdot|)\|_{(\mathcal{B}, \mathbb{R}^2)} 
\left\|\ln \frac{1}{|\cdot|}\right\|_{(\mathcal{A}, \mathbb{D})} =
2 \|V\|_{(\mathcal{B}, \mathbb{R}^2)} 
\left\|\ln \frac{1}{|\cdot|}\right\|_{(\mathcal{A}, \mathbb{D})} .
\end{eqnarray*}
$$
\int_{\mathbb{D}} \mathcal{A}\left(\ln \frac{1}{|x|}\right)\, dx \le
\int_{\mathbb{D}} e^{\ln \frac{1}{|x|}}\, dx = \int_{\mathbb{D}} \frac{1}{|x|}\, dx =
\int_{-\pi}^\pi \int_0^1 1\, dr d\vartheta = 2\pi .
$$
Hence
\begin{equation}\label{normlog}
\left\|\ln \frac{1}{|\cdot|}\right\|_{(\mathcal{A}, \mathbb{D})} \le 2\pi
\end{equation}
(see \eqref{LuxNormImpl}).
\end{proof}

Lemmas \ref{Orliczlem} and \ref{Orliczlem2} imply that Theorem \ref{clCLRth} is
equivalent to the following result. 

\begin{theorem}\label{KMWth}
There exists a constant $C_6 > 0$ such that 
\begin{eqnarray}\label{myKMWest}
N_- (\mathcal{E}_V) \le C_6  \left(\int_{\mathbb{R}^2} V(x) \ln(2 + |x|)\, dx + 
\int_{\mathbb{R}^2} V_*(|x|) \ln_+ \frac{1}{|x|}\, dx\right)  + 1 , \\ 
\forall V\ge 0 . \nonumber
\end{eqnarray}
\end{theorem}

Theorem \ref{KMWth} is in turn equivalent to the following estimate which was 
conjectured in \cite{KMW}
\begin{eqnarray}\label{KMWest}
N_- (\mathcal{E}_V) &\le&  c_1 \int_{\mathbb{R}^2} V_*(|x|) \ln_+ \frac{1}{|x|}\, dx +
c_2 \int_{\mathbb{R}^2} V(x) \ln_+ |x|\, dx \nonumber \\
&+& c_3 \int_{\mathbb{R}^2} V(x)\, dx + 1  ,
\end{eqnarray} 
provided one does not care about the exact values of the constants.
It is natural however to ask what the best constants in \eqref{KMWest} are. 
It was conjectured in \cite{CKMW} that a similar estimate
\begin{eqnarray}\label{CKMWest}
N_- (\mathcal{E}_V) &\le&  d_1 \int_{\mathbb{R}^2} V_*(|x|) \ln_+ \frac{1}{|x|}\, dx +
d_2 \int_{\mathbb{R}^2} V(x) \ln |x|\, dx \nonumber \\
&+& d_3 \int_{\mathbb{R}^2} V(x)\, dx + 1  
\end{eqnarray} 
holds with 
\begin{equation}\label{eqs}
2\pi d_1 = 2 , \ \ 2\pi d_2 = 1, \ \ 2\pi d_3 = \frac2{\sqrt{3}}\, .
\end{equation}
The difference between \eqref{KMWest} and \eqref{CKMWest} is that the second integral in the
latter can be negative. If $V(x) = F(|x|)$ is a decreasing radial potential, then 
\eqref{KMWest} and \eqref{CKMWest} become
\begin{eqnarray*}
N_- (\mathcal{E}_V) &\le&  2\pi c_1 \int_0^1r F(r)  |\ln r|\, dr +
2\pi c_2 \int_1^\infty r F(r)  \ln r\, dr \\
&+& 2\pi c_3 \int_0^\infty r F(r)\, dr + 1
\end{eqnarray*}
and
\begin{eqnarray*}
N_- (\mathcal{E}_V) &\le&  2\pi (d_1 - d_2) \int_0^1r F(r) |\ln r|\, dr +
2\pi d_2 \int_1^\infty r F(r) \ln r\, dr  \\
&+& 2\pi d_3 \int_0^\infty r F(r)\, dr + 1
\end{eqnarray*}
respectively. This allows one to obtain lower estimates for $c_1$, $c_2$, $c_3$ and
$d_1$, $d_2$, $d_3$ from the known results on radial potentials and explains the 
values of $d_1$ and $d_2$ in \eqref{eqs}. Theorem 5.2 in \cite{LS0} suggests that one could
perhaps try to prove  \eqref{CKMWest} with $2\pi d_3 = 1$.

\begin{remark}\label{MVKMW}
{\rm The proof of Theorem \ref{KMWth} relies on an idea from  \cite[Section 5]{MV1} where 
\eqref{MVest} was used in place of \eqref{clCLRest} and were the Khuri-Martin-Wu conjecture
was proved when $\|V\|_{L_1} \le 1$ or $V$ has a compact support of a fixed size.}
\end{remark}

\section{The Birman-Solomyak method}\label{BSsect}

Our description of the Birman-Solomyak method of estimating $N_- (\mathcal{E}_V)$ follows
\cite{BL, Sol, Sol2}, although some details are different. 

We denote the polar coordinates in $\mathbb{R}^2$ by $(r, \vartheta)$, \ $r \in \mathbb{R}_+$, \
$\vartheta \in = (-\pi, \pi]$. Let
$$
f_{\mathcal{R}}(r) := \frac{1}{2\pi}\, \int_{-\pi}^\pi f(r, \vartheta)\, d\vartheta , \ \ \ 
f_{\mathcal{N}}(r, \vartheta) := f(r, \vartheta) - f_{\mathcal{R}}(r) , \ \ \ 
f \in C\left(\mathbb{R}^2\setminus\{0\}\right) .
$$
Then
\begin{equation}\label{fN0}
\int_{-\pi}^\pi f_{\mathcal{N}}(r, \vartheta)\, d\vartheta = 0 , \ \ \ \forall r > 0 ,
\end{equation}
and it is easy to see that 
$$
\int_{\mathbb{R}^2} f_{\mathcal{R}} g_{\mathcal{N}}\, dx = 0 , \ \ \ \forall 
f, g \in C\left(\mathbb{R}^2\setminus\{0\}\right)\cap L_2\left(\mathbb{R}^2\right) .
$$
Hence
$f \mapsto Pf := f_{\mathcal{R}}$
extends to an orthogonal projection $P : L_2\left(\mathbb{R}^2\right) \to 
L_2\left(\mathbb{R}^2\right)$. 

Using the representation of the gradient in polar coordinates one gets
\begin{eqnarray*}
&& \int_{\mathbb{R}^2} \nabla f_{\mathcal{R}} \nabla g_{\mathcal{N}}\, dx =
\int_{\mathbb{R}^2} \left(\frac{\partial f_{\mathcal{R}}}{\partial r}
\frac{\partial g_{\mathcal{N}}}{\partial r} + \frac1{r^2}
\frac{\partial f_{\mathcal{R}}}{\partial \vartheta}
\frac{\partial g_{\mathcal{N}}}{\partial \vartheta}\right)\, dx \\
&& = \int_{\mathbb{R}^2} \frac{\partial f_{\mathcal{R}}}{\partial r}
\frac{\partial g_{\mathcal{N}}}{\partial r}\, dx =
\int_{\mathbb{R}^2} \left(\frac{\partial f}{\partial r}\right)_{\mathcal{R}}
\left(\frac{\partial g}{\partial r}\right)_{\mathcal{N}}\, dx = 0 , \ \ \ 
\forall f, g \in C^\infty_0\left(\mathbb{R}^2\right) .
\end{eqnarray*}
Hence $P : W^1_2\left(\mathbb{R}^2\right) \to W^1_2\left(\mathbb{R}^2\right)$ is also 
an orthogonal projection.

Since
\begin{eqnarray*}
&& \int_{\mathbb{R}^2} |\nabla w|^2\, dx = \int_{\mathbb{R}^2} |\nabla w_{\mathcal{R}}|^2\, dx + 
\int_{\mathbb{R}^2} |\nabla w_{\mathcal{N}}|^2\, dx , \\
&& \int_{\mathbb{R}^2} V |w|^2\, dx \le 2 \int_{\mathbb{R}^2} V |w_{\mathcal{R}}|^2\, dx +
2\int_{\mathbb{R}^2} V |w_{\mathcal{N}}|^2\, dx ,
\end{eqnarray*}
one has
\begin{equation}\label{BSest}
N_- (\mathcal{E}_V) \le N_- (\mathcal{E}_{\mathcal{R}, 2V}) + N_- (\mathcal{E}_{\mathcal{N}, 2V}) ,
\end{equation}
where $\mathcal{E}_{\mathcal{R}, 2V}$ and $\mathcal{E}_{\mathcal{N}, 2V}$ denote the restrictions
of the form $\mathcal{E}_{2V}$ to $P W^1_2\left(\mathbb{R}^2\right)$ and
$(I - P) W^1_2\left(\mathbb{R}^2\right)$ respectively.

\begin{remark}\label{RadialBS}
{\rm If the potential $V$ is radial, i.e. $V(x) = F(r)$, then it is easy to see that 
$$
P(V w) = V Pw , \ \ \ \forall w \in L_2\left(\mathbb{R}^2\right)\cap L_2\left(\mathbb{R}^2, V(x)dx\right) ,
$$
and one gets a sharper version of \eqref{BSest}:
\begin{equation}\label{RadialBSest}
N_- (\mathcal{E}_V) = N_- (\mathcal{E}_{\mathcal{R}, V}) + N_- (\mathcal{E}_{\mathcal{N}, V}) 
\end{equation}
(cf. \cite{LS0}).}
\end{remark}
 
Let us estimate the right-hand side of \eqref{BSest}. We start with the first term, i.e
with the case of $P W^1_2\left(\mathbb{R}^2\right) = 
\left\{w \in W^1_2\left(\mathbb{R}^2\right) : \ w(x) = w_{\mathcal{R}}(r)\right\}$. Using the notation
$r = e^t$, $w(x) = w_{\mathcal{R}}(r) = u(t)$, we get
\begin{eqnarray*}
&& \int_{\mathbb{R}^2} |\nabla w(x)|^2\, dx = 2\pi \int_{\mathbb{R}} |u'(t)|^2\, dt , \ \ \ 
\int_{\mathbb{R}^2} |w(x)|^2\, dx = 2\pi \int_{\mathbb{R}} |u(t)|^2 e^{2t}\, dt , \\
&& \int_{\mathbb{R}^2} V(x) |w(x)|^2\, dx = 2\pi \int_{\mathbb{R}} G(t) |u(t)|^2\, dt , 
\end{eqnarray*}
where
\begin{equation}\label{V1d}
G(t) := \frac{e^{2t}}{2\pi}\, \int_{-\pi}^\pi V(e^t, \vartheta)\, d\vartheta .
\end{equation}
The above change of variable defines a unitary operator from $P W^1_2\left(\mathbb{R}^2\right)$
onto 
$$
X := \left\{u \in W^1_{2, {\rm loc}}(\mathbb{R}) : \ 
\|u\|_X := \sqrt{2\pi}\, \left(\int_{\mathbb{R}} |u'|^2\, dt + 
\int_{\mathbb{R}} |u|^2 e^{2t}\, dt\right)^{1/2} < \infty\right\} .
$$
Let $X_0 := \{u \in X : \ u(0) = 0\}$ and 
\begin{equation}\label{X0H0}
\mathcal{H}_0 := \left\{u \in W^1_{2, {\rm loc}}(\mathbb{R}) : \ 
u(0) = 0 , \ \int_{\mathbb{R}} |u'|^2\, dt < \infty\right\} .
\end{equation}
Let $\mathcal{E}_{X, G}$, $\mathcal{E}_{X_0, G}$ and $\mathcal{E}_{\mathcal{H}_0, G}$
denote the forms defined by
\begin{equation}\label{Gform}
\int_{\mathbb{R}} |u'(t)|^2\, dt - \int_{\mathbb{R}} G(t) |u(t)|^2\, dt
\end{equation}
on the domains
$$
X\cap L_2\left(\mathbb{R}, G(t)dt\right) ,  \ \ \ X_0\cap L_2\left(\mathbb{R}, G(t)dt\right) \ \
\mbox{ and } \ \ \mathcal{H}_0\cap L_2\left(\mathbb{R}, G(t)dt\right)
$$
respectively. Since $\dim(X/X_0) = 1$ and $X_0 \subset \mathcal{H}_0$, one has
$$
N_- (\mathcal{E}_{\mathcal{R}, 2V}) = N_- (\mathcal{E}_{X, 2G}) \le 
N_- (\mathcal{E}_{X_0, 2G}) + 1 \le N_- (\mathcal{E}_{\mathcal{H}_0, 2G}) + 1 .
$$
It follows from Hardy's inequality (see, e.g., \cite[Theorem 327]{HLP}) that
$$
\int_{\mathbb{R}} |u'|^2\, dt + \kappa\, \int_{\mathbb{R}} \frac{|u|^2}{|t|^2}\, dt \le 
(4\kappa + 1) \int_{\mathbb{R}} |u'|^2\, dt ,  \ \ \ \forall u \in \mathcal{H}_0 , \ \ 
\forall \kappa \ge 0 .
$$
Hence
$$
N_- (\mathcal{E}_{\mathcal{H}_0, 2G}) \le N_- (\mathcal{E}_{\kappa, G}) ,
$$
where
\begin{eqnarray}\label{kappaG}
\mathcal{E}_{\kappa, G}[u] := \int_{\mathbb{R}} |u'(t)|^2\, dt + 
\kappa\, \int_{\mathbb{R}} \frac{|u(t)|^2}{|t|^2}\, dt - 
2(4\kappa + 1)\, \int_{\mathbb{R}} G(t) |u(t)|^2\, dt , \\
\mbox{Dom}\, (\mathcal{E}_{\kappa, G}) = \mathcal{H}_0\cap L_2\left(\mathbb{R}, G(t)dt\right) .
\nonumber
\end{eqnarray}
It follows from the above that
\begin{equation}\label{VkappaG}
N_- (\mathcal{E}_{\mathcal{R}, 2V}) \le N_- (\mathcal{E}_{\kappa, G}) + 1 .
\end{equation}
We estimate $N_- (\mathcal{E}_{\kappa, G})$ by partitioning $\mathbb{R}$ into the intervals
\begin{eqnarray*}
I_n := [2^{n - 1}, 2^n], \ n > 0 , \ \ \ I_0 := [-1, 1] , \ \ \ 
I_n := [-2^{|n|}, -2^{|n| - 1}], \ n < 0 ,
\end{eqnarray*}
and by using the variational principle to obtain 
\begin{equation}\label{VkappaGn}
N_- (\mathcal{E}_{\kappa, G}) \le \sum_{n \in \mathbb{Z}} N_- (\mathcal{E}_{\kappa, G, n}) ,
\end{equation}
where
\begin{eqnarray*}
&& \mathcal{E}_{\kappa, G, n}[u] := \int_{I_n} |u'|^2\, dt + 
\kappa\, \int_{I_n} \frac{|u|^2}{|t|^2}\, dt - 
2(4\kappa + 1)\, \int_{I_n} G |u|^2\, dt , \\
&& \mbox{Dom}\, (\mathcal{E}_{\kappa, G, n}) = 
W^1_2(I_n)\cap L_2\left(I_n, G(t)dt\right) , \ n \in \mathbb{Z}\setminus\{0\} , \\
&& \mbox{Dom}\, (\mathcal{E}_{\kappa, G, 0}) = 
\{u \in W^1_2(I_0) : \ u(0) = 0\}\cap L_2\left(I_0, G(t)dt\right) .
\end{eqnarray*}

Let $n > 0$. For any $N \in \mathbb{N}$, there exists a subspace $\mathcal{F}_N \in 
\mbox{Dom}\, (\mathcal{E}_{\kappa, G, n})$ of co-dimension $N$ such that
$$
\int_{I_n} G |u|^2\, dt \le \left(\frac{|I_n|}{N^2}\, \int_{I_n} G\, dt\right) \int_{I_n} |u'|^2\, dt , \ \ \
\forall u \in \mathcal{F}_N
$$
(see \cite[the proof of Proposition 4.2 in Appendix]{Sol2}). If
$$
2(4\kappa + 1)\, \frac{|I_n|}{N^2}\, \int_{I_n} G\, dt \le 1 ,
$$
then $ \mathcal{E}_{\kappa, G, n}[u] \ge 0$, $\forall u \in \mathcal{F}_N$, and  
$N_- (\mathcal{E}_{\kappa, G, n}) \le N$. Let
\begin{equation}\label{calAn}
\mathcal{A}_n := \int_{I_n} |t| G(t)\, dt , \ n \not= 0 , \ \ \ \mathcal{A}_0 := \int_{I_0} G(t)\, dt .
\end{equation}
Since $|I_n|\, \int_{I_n} G\, dt \le \mathcal{A}_n$, $n \not= 0$, it follows from the above that
$$
2(4\kappa + 1) \mathcal{A}_n \le N^2 \ \Longrightarrow \ N_- (\mathcal{E}_{\kappa, G, n}) \le N .
$$
Hence 
\begin{equation}\label{ceil1}
N_- (\mathcal{E}_{\kappa, G, n}) \le \left\lceil\sqrt{2(4\kappa + 1) \mathcal{A}_n}\right\rceil , 
\end{equation}
where $\lceil\cdot\rceil$ denotes the ceiling function, i.e. $\lceil a\rceil$ is the smallest integer 
not less than $a$. The right-hand side of this estimate is at least 1, so one cannot feed it straight 
into \eqref{VkappaGn}. One needs to find conditions under which 
$N_- (\mathcal{E}_{\kappa, G, n}) = 0$. 

It follows from \eqref{1dHSob}, \eqref{1dHOpt} that
$$
\int_{I_n} G |u|^2\, dt \le \mathcal{A}_n C(\kappa) \left(\int_{I_n} |u'|^2\, dt + 
\kappa\, \int_{I_n} \frac{|u|^2}{|t|^2}\, dt\right) ,
$$
where
$$
C(\kappa) = \frac{1}{2\kappa}\,\left(1 + \sqrt{1 + 4\kappa}\,
\frac{2^{\sqrt{1 + 4\kappa}} + 1}{2^{\sqrt{1 + 4\kappa}} - 1}\right)\, .
$$
Hence $N_- (\mathcal{E}_{\kappa, G, n}) = 0$, i.e. $\mathcal{E}_{\kappa, G, n} \ge 0$, provided
$\mathcal{A}_n \le \Phi(\kappa)$, where
\begin{equation}\label{Psikappa}
\Phi(\kappa) := \frac{\kappa}{4\kappa + 1}\,\left(1 + \sqrt{4\kappa + 1}\,
\frac{2^{\sqrt{4\kappa + 1}} + 1}{2^{\sqrt{4\kappa + 1}} - 1}\right)^{-1}\, .
\end{equation}
The above estimates for $N_- (\mathcal{E}_{\kappa, G, n})$ clearly hold for $n < 0$ as well, but
the case $n = 0$ requires some changes. Since $u(0) = 0$ for any 
$u \in \mbox{Dom}\, (\mathcal{E}_{\kappa, G, 0})$, one can use the same argument as the one
leading to \eqref{ceil1}, but with two differences: a) $\mathcal{F}_N$ can be chosen to be of 
co-dimension $N - 1$, and b) $|I_0|\, \int_{I_0} G\, dt = 2\mathcal{A}_0$. This gives the following 
analogue of \eqref{ceil1}
$$
N_- (\mathcal{E}_{\kappa, G, 0}) \le \left\lceil2\,\sqrt{(4\kappa + 1) \mathcal{A}_0}\right\rceil - 1 <
2\,\sqrt{(4\kappa + 1) \mathcal{A}_0}\, .
$$
In particular, $N_- (\mathcal{E}_{\kappa, G, 0}) = 0$ if 
$\mathcal{A}_0 \le 1/(4(4\kappa + 1))$. Using Remark \ref{HSob01}, one can easily see that the
implication $\mathcal{A}_n \le \Phi(\kappa) \  \Longrightarrow \ 
N_- (\mathcal{E}_{\kappa, G, n}) = 0$ remains true for $n = 0$.
Now it follows from \eqref{VkappaG}, \eqref{VkappaGn} that
\begin{equation}\label{RadGenEst}
N_- (\mathcal{E}_{\mathcal{R}, 2V}) \le 1 + 
\sum_{\{n \in \mathbb{Z}\setminus\{0\} : \ \mathcal{A}_n > \Phi(\kappa)\}} 
\left\lceil\sqrt{2(4\kappa + 1) \mathcal{A}_n}\right\rceil + 
2\,\sqrt{(4\kappa + 1) \mathcal{A}_0}\, ,
\end{equation}
and one can drop the last term in $\mathcal{A}_0 \le \Phi(\kappa)$. The presence of the 
parameter $\kappa$ in this estimate allows a degree of flexibility. In order to decrease the number
of terms in the sum in the right-hand side, one should choose $\kappa$ in such a way that
$\Phi(\kappa)$ is close to its maximum. A Mathematica calculation shows that the maximum is
approximately $0.046$ and is achieved at $\kappa \approx 1.559$. For values of $\kappa$
close to 1.559, one has
$$
\mathcal{A}_n > \Phi(\kappa)  \  \Longrightarrow \ 
\sqrt{2(4\kappa + 1) \mathcal{A}_n}  > \sqrt{2(4\kappa + 1) \Phi(\kappa)} \approx 0.816 .
$$
Since $\lceil a\rceil \le 2a$ for $a \ge 1/2$, \eqref{RadGenEst} implies
$$
N_- (\mathcal{E}_{\mathcal{R}, 2V}) \le 1 + 2\sqrt{2(4\kappa + 1)}
\sum_{\mathcal{A}_n > \Phi(\kappa)}  \sqrt{\mathcal{A}_n}
$$
with $\kappa \approx 1.559$. Hence
$$
N_- (\mathcal{E}_{\mathcal{R}, 2V}) \le 1 + 7.61
\sum_{\mathcal{A}_n > 0.046}  \sqrt{\mathcal{A}_n}\, .
$$
Let us rewrite this estimate in terms of the original potential $V$. Set
\begin{eqnarray}\label{ringsR}
&& U_n :=  \{x \in \mathbb{R}^2 : \ e^{2^{n - 1}} < |x| < e^{2^n}\} , \ n > 0 ,  \nonumber \\ 
&& U_0 := \{x \in \mathbb{R}^2 : \ e^{-1} < |x| < e\}, \\
&& U_n :=  \{x \in \mathbb{R}^2 : \ e^{-2^{|n|}} < |x| < e^{-2^{|n| - 1}}\} , \ n < 0 , \nonumber 
\end{eqnarray}
and
\begin{equation}\label{AnBnR}
A_0 := \int_{U_0} V(x)\, dx , \ \
A_n := \int_{U_n} V(x) |\ln|x||\, dx , \ n \not= 0  .
\end{equation}
Then it follows from \eqref{V1d}, \eqref{calAn} that $A_n = 2\pi \mathcal{A}_n$, and we have
$$
N_- (\mathcal{E}_{\mathcal{R}, 2V}) \le 1 + 3.04
\sum_{A_n > 0.29}  \sqrt{A_n}\, .
$$
Below, we will use the following less precise but nicer looking estimate
\begin{equation}\label{RadMainEst}
N_- (\mathcal{E}_{\mathcal{R}, 2V}) \le 1 + 4
\sum_{A_n > 1/4}  \sqrt{A_n}\, .
\end{equation}

\begin{remark}\label{SqrtHist}
{\rm To the best of my knowledge, estimates of this type (without explicit constants) were first 
obtained by M. Birman and M. Solomyak for Schr\"odinger-type operators of order $2\ell$ in 
$\mathbb{R}^d$ with $2\ell > d$ (see \cite[\S 6]{BS}). A. Grigor'yan and N. Nadirashvili 
(\cite{Grig}) obtained an estimate of this type for  two-dimensional Schr\"odinger operators
(see \eqref{GrigTalkEst0} below). The Grigor'yan-Nadirashvili estimate is discussed in the
next section.}
\end{remark}

Returning to \eqref{BSest}, we need to estimate $N_- (\mathcal{E}_{\mathcal{N}, 2V})$.
To this end, we split $\mathbb{R}^2$ into the following unnuli
\begin{equation}\label{ringsN}
\Omega_n :=  \{x \in \mathbb{R}^2 : \ e^{n} < |x| < e^{n + 1}\} , \ \ \ n \in \mathbb{Z} . 
\end{equation}
It follows from \eqref{fN0} that
$$
\int_{\Omega_n} w(x)\, dx = 0 , \ \ \  \forall w \in (I - P) W^1_2\left(\mathbb{R}^2\right) , \ \ 
\forall n \in \mathbb{Z} . 
$$
Hence the variational principle implies
\begin{equation}\label{VNn}
N_- (\mathcal{E}_{\mathcal{N}, 2V}) \le \sum_{n \in \mathbb{Z}} 
N_- (\mathcal{E}_{\mathcal{N}, 2V, n}) ,
\end{equation}
where
\begin{eqnarray}\label{OmeganForm}
&& \mathcal{E}_{\mathcal{N}, 2V, n}[w] : = \int_{\Omega_n} |\nabla w(x)|^2 dx - 
2\int_{\Omega_n} V(x) |w(x)|^2 dx ,  \\  
&&  \mbox{Dom}\, (\mathcal{E}_{\mathcal{N}, 2V, n}) =
\left\{w \in W^1_2\left(\Omega_n\right)\cap L_2\left(\Omega_n, V(x)dx\right) : \
\int_{\Omega_n} w\, dx = 0\right\}.  \nonumber
\end{eqnarray}
It is clear that $x \mapsto e^n x$ maps $U_0$ onto $U_n$. So, any estimate for 
$N_- (\mathcal{E}_{\mathcal{N}, 2V, 0})$ that has the right scaling leads to an estimate for
$N_- (\mathcal{E}_{\mathcal{N}, 2V, n})$, and then an estimate for
$N_- (\mathcal{E}_V)$ follows from \eqref{BSest}, \eqref{RadMainEst} and
\eqref{VNn}.

\section{A Grigor'yan-Nadirashvili type estimate}\label{GNsect}

Let
\begin{equation}\label{AnBnN}
\mathcal{B}_n := 
\|V\|^{\rm (av)}_{\mathcal{B}, \Omega_n} , \ \ \
B_n := \left(\int_{\Omega_n} V^p(x) |x|^{2(p -1)}\, dx \right)^{1/p} , \ \ \ p > 1 , \ \ 
n \in \mathbb{Z} 
\end{equation}
(see \eqref{ringsN}).

\begin{theorem}\label{GrigTalk}
There exist  constants $C_7 > 0$ and $c > 0$ such that 
\begin{equation}\label{GrigTalkEst}
N_- (\mathcal{E}_V) \le 1 + 4 \sum_{A_n > 1/4}  \sqrt{A_n}
+ C_7 \sum_{\mathcal{B}_n > c} \mathcal{B}_n  , \ \ \
\forall V\ge 0 
\end{equation}
(see \eqref{ringsR}, \eqref{AnBnR}).
\end{theorem}
\begin{proof}
According to Theorem $4'$ in \cite{Sol} (see also Proposition 3 there), there exists 
$C_7 > 0$ such that $N_- (\mathcal{E}_{\mathcal{N}, 2V, n}) \le C_7 \mathcal{B}_n$,
$\forall n \in \mathbb{Z}$ (see \eqref{OmeganForm}). In particular, 
$N_- (\mathcal{E}_{\mathcal{N}, 2V, n}) = 0$ if $\mathcal{B}_n < 1/C_7$, and one can 
drop this term from the sum in \eqref{VNn}. Now it
follows from \eqref{BSest}, \eqref{RadMainEst} and \eqref{VNn} that 
\eqref{GrigTalkEst} holds for any $c < 1/C_7$.
\end{proof}

\begin{remark}\label{GNSol}
{\rm The above result is closely related to the well known estimate from \cite{Sol}:
\begin{equation}\label{SolMainEst}
N_- (\mathcal{E}_V) \le 1 + C_8 \left(\left\|(\mathbf{A}_n)_{n \ge 0}\right\|_{1, \infty}
+ \sum_{n = 0}^\infty \mathbf{B}_n\right)  , \ \ \
\forall V\ge 0 , 
\end{equation}
where $\mathbf{A}_n = A_n$, $\mathbf{B}_n = \mathcal{B}_n$ for $n \in \mathbb{N}$,
\begin{equation}\label{bald0}
\mathbf{A}_0 := \int_{\mathbf{\Omega}_0} V(x) |\ln|x||\, dx , \ \ \ \mathbf{B}_0 = 
\|V\|^{\rm (av)}_{\mathcal{B}, \mathbf{\Omega}_0} ,
\ \ \  \mathbf{\Omega}_0 =  \{x \in \mathbb{R}^2 : \ |x| \le e\} ,
\end{equation}
and 
$$
\left\|(a_n)_{n \ge 0}\right\|_{1, \infty} := \sup_{s > 0} \Big(s\, \mbox{card}\, \{n \ge 0 : \ 
|a_n| > s\}\Big) .
$$
Note that 
$$
\sum_{\{n \le 0 : \ \mathcal{B}_n > c\}} \mathcal{B}_n \le
\sum_{n \le 0} \mathcal{B}_n \le \mathbf{B}_0
$$
(see \cite[Lemma 3]{Sol}), and that
\begin{eqnarray}\label{sqrtweak}
&& \sum_{|a_n| > c} \sqrt{|a_n|} = \sum_{|a_n| > c} \int_0^{|a_n|} \frac12\, s^{-1/2} ds \nonumber \\
&& = \frac12 \int_0^\infty s^{-1/2}\, \mbox{card}\, \{n \ge 0 : \ 
|a_n| > s \ \& \ |a_n| > c\}\, ds \nonumber \\
&& \le \frac12 \int_0^\infty s^{-1/2}\, 
\frac{\left\|(a_n)_{n \ge 0}\right\|_{1, \infty}}{\max\{s, c\}}\, ds \\
&& = \frac{\left\|(a_n)_{n \ge 0}\right\|_{1, \infty}}{2} \left(\int_0^c \frac{s^{-1/2}}{c}\, ds +
\int_c^\infty s^{-3/2} ds\right) = \frac{2}{\sqrt{c}}\, \left\|(a_n)_{n \ge 0}\right\|_{1, \infty} . \nonumber
\end{eqnarray}
Hence
\begin{eqnarray*}
&& \sum_{\{n \in \mathbb{Z}\setminus\{0\} : \ A_n > c\}}  \sqrt{A_n} = 
\sum_{\{n < 0 : \ A_n > c\}}  \sqrt{A_n} +
\sum_{\{n > 0 : \ A_n > c\}}  \sqrt{A_n} \\
&& \le \frac{1}{\sqrt{c}} \sum_{n < 0} A_n + 
\frac{2}{\sqrt{c}}\, \left\|(A_n)_{n > 0}\right\|_{1, \infty}  \le \frac{1}{\sqrt{c}}\, \mathbf{A}_0 + 
\frac{2}{\sqrt{c}}\, \left\|(A_n)_{n > 0}\right\|_{1, \infty}  \\
&& \le \frac{3}{\sqrt{c}}\, \left\|(\mathbf{A}_n)_{n \ge 0}\right\|_{1, \infty} .
\end{eqnarray*}
It follows from the H\"older inequality (see \cite[Theorem 9.3]{KR}) that if $A_0 > c$, then
\begin{eqnarray*}
\sqrt{A_0} < \frac{1}{\sqrt{c}}\, A_0 =   \frac{1}{\sqrt{c}} \int_{U_0} V(x)\, dx 
\le \frac{1}{\sqrt{c}} \int_{\mathbf{\Omega}_0} V(x)\, dx \\
\le \frac{1}{\sqrt{c}} \|V\|_{\mathcal{B}, \mathbf{\Omega}_0} \|1\|_{\mathcal{A}, \mathbf{\Omega}_0} 
\le \mbox{const}\, \|V\|^{\rm (av)}_{\mathcal{B}, \mathbf{\Omega}_0}  = 
\mbox{const}\, \mathbf{B}_0 .
\end{eqnarray*}
So, \eqref{GrigTalkEst} implies \eqref{SolMainEst}.}
\end{remark}

\begin{remark}\label{GNimplied}
{\rm Theorem \ref{GrigTalk} is a slight improvement of Theorem 1.1 in \cite{Grig} where
an analogue of \eqref{GrigTalkEst} was obtained with $B_n$ in place of $\mathcal{B}_n$:
\begin{equation}\label{GrigTalkEst0}
N_- (\mathcal{E}_V) \le 1 + C_7 \sum_{\{n \in \mathbb{Z} : \ A_n > c\}}  \sqrt{A_n}
+ C_7 \sum_{\{n \in \mathbb{Z} : \ B_n > c\}} B_n  , \ \ \
\forall V\ge 0 . 
\end{equation}
Indeed, let 
\begin{eqnarray*}
&& \Delta := \left\{x \in \mathbb{R}^2 : \ \frac{1}{\sqrt{\pi (e^2 - 1)}} < |x| < 
\frac{e}{\sqrt{\pi (e^2 - 1)}}\right\} ,\\
&& R_n := e^n\, \sqrt{\pi (e^2 - 1)}\, , \\
&& \xi_n : \mathbb{R}^2 \to \mathbb{R}^2 , \ \ \xi_n(x) = R_n x , \ \ \ n \in \mathbb{Z} .
\end{eqnarray*}
Then $|\Delta| = 1$, $\xi_n(\Delta) = \Omega_n$ and it follows from \cite[(7)]{Sol} and
\cite[\S 13]{KR} that
\begin{eqnarray*}
&& \mathcal{B}_n = \|V\|^{\rm (av)}_{\mathcal{B}, \Omega_n} = |\Omega_n| 
\|V\circ \xi_n\|_{\mathcal{B}, \Delta} \le C(p) |\Omega_n| \|V\circ \xi_n\|_{L_p(\Delta)} \\
&& = C(p) R_n^2 R_n^{-2/p} \|V\|_{L_p(\Omega_n)} = C(p) R_n^{2 -2/p}
\left(\int_{\Omega_n} V^p(x) dx \right)^{1/p} \\
&& \le C(p) \left(\pi (e^2 - 1)\right)^{1 - 1/p} 
\left(\int_{\{e^{n} < |x| < e^{n + 1}\}} V^p(x) |x|^{2(p -1)}\, dx \right)^{1/p} \\
&& =  C(p) \left(\pi (e^2 - 1)\right)^{1 - 1/p} B_n , \ \ \ p > 1.
\end{eqnarray*}
Note that the optimal constant $M(p)$ in the inequality 
$\|f\|_{(\mathcal{B}, \Delta)} \le M(p) \|f\|_{L_p(\Delta)}$ has the following asymptotics:
\begin{equation}\label{M(p)}
M(p) = \frac1{e(p - 1)} - \frac1e\, \ln\frac{1}{p - 1}  + O(1) \ \ \mbox{ as  } \ \ p \to 1, \ p > 1 .
\end{equation}
Indeed,  an elementary analysis shows that the fraction
$$
\frac{\mathcal{B}(t)}{t^p} = \frac{(1 + t)\ln(1 + t) - t}{t^p}\, , \ \ \ t > 0 , \ \ 1 < p \le 2,
$$
achieves its maximum at a point
$$
t_p = \exp\left(\frac{p}{p - 1} + o(1)\right) \ \ \mbox{ as  } \ \ p \to 1 ,
$$
where its value is 
$$
m(p) = \frac1{e(p - 1)} + O(1) \ \ \mbox{ as  } \ \ p \to 1 .
$$
Hence 
$$
\int_{\Delta} \mathcal{B}\left(\frac{|f(x)|}{m(p)^{1/p}\|f\|_{L_p(\Delta)}}\right)\, dx \le m(p) 
\int_{\Delta}\left|\frac{f(x)}{m(p)^{1/p}\|f\|_{L_p(\Delta)}}\right|^p\, dx = 1 ,
$$
and $\|f\|_{(\mathcal{B}, \Delta)} \le m(p)^{1/p} \|f\|_{L_p(\Delta)}$.
On the other hand, take any subset $\Delta_p \subset \Delta$  
such that $|\Delta_p| = (m(p) t_p^p)^{-1}$, and define the function $f_p$ by $f_p(x) = t_p$
if $x \in \Delta_p$ and $f_p(x) = 0$ if $x \in \Delta\setminus\Delta_p$. Then 
$\|f_p\|_{L_p(\Delta)} = m(p)^{-1/p}$ and 
\begin{eqnarray*}
\int_{\Delta} \mathcal{B}(f_p(x))\, dx = m(p) \int_{\Delta} |f_p(x)|^p\, dx = 1 \ \ \Longrightarrow \\ 
\|f_p\|_{(\mathcal{B}, \Delta)} = 1 =  m(p)^{1/p} \|f_p\|_{L_p(\Delta)} .
\end{eqnarray*}
Hence the optimal constant is
\begin{eqnarray*}
&& M(p) = m(p)^{1/p} = \left(\frac1{e(p - 1)} + O(1)\right)^{1/p} \\ 
&&= \left(\frac1{e(p - 1)} + O(1)\right)
\exp\left(\frac{1 - p}{p}\, \ln\left(\frac1{e(p - 1)} + O(1)\right)\right) \\
&& = \left(\frac1{e(p - 1)} + O(1)\right) \left(1 + (1 - p) \ln\frac{1}{p - 1} + O(p - 1)\right) \\
&& = \frac1{e(p - 1)} - \frac1e\, \ln\frac{1}{p - 1}  + O(1)  \ \ \mbox{ as  } \ \ p \to 1 .
\end{eqnarray*}}
\end{remark}

\section{A Laptev-Netrusov-Solomyak type estimate}

We denote the polar coordinates in $\mathbb{R}^2$ by $(r, \vartheta)$, \ $r \in \mathbb{R}_+$, \
$\vartheta \in \mathbb{S} := (-\pi, \pi]$.
Let  $I \subseteq \mathbb{R}_+$ be a nonempty open interval and let 
$$
\Omega_I := \{x \in \mathbb{R}^2 : \ |x| \in I\} .
$$ 
We denote by $\mathcal{L}_1\left(I, L_{\mathcal{B}}(\mathbb{S})\right)$ the space of measurable
functions $f : \Omega_I \to \mathbb{C}$ such that
\begin{equation}\label{curlyLnorm}
\|f\|_{\mathcal{L}_1\left(I, L_{\mathcal{B}}(\mathbb{S})\right)} := \int_I
\|f(r, \cdot)\|_{\mathcal{B}, \mathbb{S}}\, r dr < +\infty .
\end{equation}
Let
\begin{equation}\label{InDn}
\mathcal{I}_n := (e^n, e^{n + 1}) , \ \ \ \mathcal{D}_n := 
\|V\|_{\mathcal{L}_1\left(\mathcal{I}_n, L_{\mathcal{B}}(\mathbb{S})\right)} , \ \ \ n \in \mathbb{Z} .
\end{equation}

\begin{theorem}\label{LaptNetrSol}
There exist  constants $C_9 > 0$ and $c > 0$ such that 
\begin{equation}\label{LaptNetrSolEst}
N_- (\mathcal{E}_V) \le 1 + 4 \sum_{A_n > 1/4}  \sqrt{A_n}
+ C_9 \sum_{\mathcal{D}_n > c} \mathcal{D}_n\,  , \ \ \
\forall V\ge 0 
\end{equation}
(see \eqref{ringsR}, \eqref{AnBnR}).
\end{theorem}

\begin{remark}\label{LaptNetrSolRem}
{\rm It is clear that
$$
\sum_{\{n \in \mathbb{Z} : \ \mathcal{D}_n > c\}} \mathcal{D}_n \le 
\sum_{n \in \mathbb{Z}} \mathcal{D}_n = 
\|V\|_{\mathcal{L}_1\left(\mathbb{R}_+, L_{\mathcal{B}}(\mathbb{S})\right)} .
$$
It follows from \eqref{sqrtweak} that
$$
\sum_{\{n \in \mathbb{Z} : \ A_n > c\}}  \sqrt{A_n} \le 
\frac{2}{\sqrt{c}}\, \left\|(A_n)_{n \in \mathbb{Z}}\right\|_{1, \infty} .
$$
Hence \eqref{LaptNetrSolEst} implies
\begin{equation}\label{LaptNetrSolEst2}
N_- (\mathcal{E}_V) \le 1 + C_{10} \left(\left\|(A_n)_{n \in \mathbb{Z}}\right\|_{1, \infty} +
\|V\|_{\mathcal{L}_1\left(\mathbb{R}_+, L_{\mathcal{B}}(\mathbb{S})\right)}\right)  , \ \ \
\forall V\ge 0 ,
\end{equation}
which, in turn, implies the following estimate obtained in \cite{LS} (see also \cite{LN}):
\begin{equation}\label{LaptNetrSolEst3}
N_- (\mathcal{E}_V) \le 1 + C_{11} \left(\left\|(A_n)_{n \in \mathbb{Z}}\right\|_{1, \infty} +
\int_{\mathbb{R}_+} \left(\int_{\mathbb{S}}|V(r, \vartheta)|^p d\vartheta\right)^{1/p} r dr \right) ,
\end{equation}
where $p > 1$ and $C_{11} = O(1/(p - 1))$ as $p \to 1$ (see \eqref{M(p)}).
In fact, the estimate obtained in \cite{LS} has a slightly different form:
\begin{equation}\label{LaptNetrSolEst4}
N_- (\mathcal{E}_V) \le 1 + C_{12} \left(\left\|(A_n)_{n \in \mathbb{Z}}\right\|_{1, \infty} +
\int_{\mathbb{R}_+} 
\left(\int_{\mathbb{S}}|V_{\mathcal{N}}(r, \vartheta)|^p d\vartheta\right)^{1/p} r dr\right) ,
\end{equation}
where
$$
V_{\mathcal{N}}(r, \vartheta) := V(r, \vartheta) - V_{\mathcal{R}}(r) , \ \ \ 
V_{\mathcal{R}}(r) := \frac1{2\pi} \int_{\mathbb{S}} V(r, \vartheta) d\vartheta .
$$
Since 
\begin{eqnarray*}
\left|\int_{\mathbb{R}_+} \left(\int_{\mathbb{S}}|V(r, \vartheta)|^p d\vartheta\right)^{1/p} r dr
- \int_{\mathbb{R}_+} 
\left(\int_{\mathbb{S}}|V_{\mathcal{N}}(r, \vartheta)|^p d\vartheta\right)^{1/p} r dr\right| \\
\le \int_{\mathbb{R}_+} 
\left(\int_{\mathbb{S}}|V_{\mathcal{R}}(r)|^p d\vartheta\right)^{1/p} r dr = 
(2\pi)^{1/p} \int_{\mathbb{R}_+} V_{\mathcal{R}}(r)\, r dr \\
= (2\pi)^{1/p - 1} \int_{\mathbb{R}^2} V(x)\, dx = 
(2\pi)^{1/p - 1} \sum_{n \in \mathbb{Z}} \int_{U_n} V(x)\, dx \\
\le \mbox{const}\, \sum_{n \in \mathbb{Z}} 2^{-|n|} A_n
\le \mbox{const}\, \sup_{n \in \mathbb{Z}}  A_n 
\le \mbox{const}\, \left\|(A_n)_{n \in \mathbb{Z}}\right\|_{1, \infty} ,
\end{eqnarray*}
estimates \eqref{LaptNetrSolEst3} and \eqref{LaptNetrSolEst4} are equivalent to each other. 
An advantage of the latter is that it separates the contribution of the radial part $V_{\mathcal{R}}$
of $V$ from that of the non-radial part $V_{\mathcal{N}}$ (see \cite{LS}). Similarly to the above,
one can easily show that  
\begin{eqnarray}\label{NVequiv}
\left|\|V\|_{\mathcal{L}_1\left(\mathbb{R}_+, L_{\mathcal{B}}(\mathbb{S})\right)} -
\|V_{\mathcal{N}}\|_{\mathcal{L}_1\left(\mathbb{R}_+, L_{\mathcal{B}}(\mathbb{S})\right)}\right|
\le \mbox{const}\, \sup_{n \in \mathbb{Z}}  A_n \\
\le \mbox{const}\, \left\|(A_n)_{n \in \mathbb{Z}}\right\|_{1, \infty} , \nonumber
\end{eqnarray}
and hence \eqref{LaptNetrSolEst2} is equivalent to
\begin{equation}\label{LaptNetrSolEst5}
N_- (\mathcal{E}_V) \le 1 + C_{13} \left(\left\|(A_n)_{n \in \mathbb{Z}}\right\|_{1, \infty} +
\|V_{\mathcal{N}}\|_{\mathcal{L}_1\left(\mathbb{R}_+, L_{\mathcal{B}}(\mathbb{S})\right)}\right)  , \ \ \
\forall V\ge 0 .
\end{equation}
}
\end{remark}

Let $I_1, I_2 \subseteq \mathbb{R}$ be nonempty open intervals. 
We denote by $L_1\left(I_1, L_{\mathcal{B}}(I_2)\right)$ the space of measurable functions
$f : I_1\times I_2 \to \mathbb{C}$ such that
\begin{equation}\label{L1LlogLnorm}
\|f\|_{L_1\left(I_1, L_{\mathcal{B}}(I_2)\right)} := \int_{I_1} 
\|f(x_1, \cdot)\|^{\rm (av)}_{\mathcal{B}, I_2}\, dx_1 < +\infty
\end{equation}
(see \eqref{OrlAverage}).

\begin{lemma}\label{affine}{\rm (Cf. \cite[Lemma 1]{Sol})} Consider an affine transformation
$$
\xi : \mathbb{R}^2 \to \mathbb{R}^2, \ \xi(x) := Ax + x^0, \ \ \ A = \begin{pmatrix}
    R_1  &  0  \\
  0    &  R_2
\end{pmatrix} , \ R_1, R_2 > 0 , \ \ \ x^0 \in \mathbb{R}^2 .
$$
Let $I_1\times I_2 = \xi(J_1\times J_2)$. Then 
\begin{eqnarray*}
\frac{1}{|I_1\times I_2|}\, \|f\|_{L_1\left(I_1, L_{\mathcal{B}}(I_2)\right)} =
\frac{1}{|J_1\times J_2|}\, \|f\circ\xi\|_{L_1\left(J_1, L_{\mathcal{B}}(J_2)\right)}, \\ 
\forall f \in L_1\left(I_1, L_{\mathcal{B}}(I_2)\right) . 
\end{eqnarray*}
\end{lemma}
\begin{proof}
\begin{eqnarray*}
\frac{1}{|I_1\times I_2|}\, \|f\|_{L_1\left(I_1, L_{\mathcal{B}}(I_2)\right)} =
\frac{1}{|I_1\times I_2|}\, \int_{I_1} 
\|f(y_1, \cdot)\|^{\rm (av)}_{\mathcal{B}, I_2}\, dy_1 \\
= \frac{R_1}{|I_1\times I_2|}\,  \int_{J_1} 
\|f(R_1x_1 + x^0_1, \cdot)\|^{\rm (av)}_{\mathcal{B}, I_2}\, dx_1  \\
= \frac{R_1 R_2}{|I_1\times I_2|}\,  \int_{J_1} 
\|f(R_1x_1 + x^0_1, R_2\cdot + x^0_2)\|^{\rm (av)}_{\mathcal{B}, J_2}\, dx_1\\
 = \frac{1}{|J_1\times J_2|}\, \|f\circ\xi\|_{L_1\left(J_1, L_{\mathcal{B}}(J_2)\right)}
\end{eqnarray*}
(see \cite[Lemma 1]{Sol}).
\end{proof}

\begin{lemma}\label{subadd}{\rm (Cf. \cite[Lemma 3]{Sol})} 
Let rectangles $I_{1, k}\times I_{2, k}$, $k = 1, \dots, n$ be pairwise disjoint subsets of 
$I_1\times I_2$. Then
\begin{equation}\label{subaddest}
\sum_{k = 1}^n \|f\|_{L_1\left(I_{1, k}, L_{\mathcal{B}}(I_{2, k})\right)} \le 
\|f\|_{L_1\left(I_1, L_{\mathcal{B}}(I_2)\right)} , \ \ \ \forall f \in L_1\left(I_1, L_{\mathcal{B}}(I_2)\right) .
\end{equation}
\end{lemma}
\begin{proof}
Let us label the endpoints of the intervals $I_1$ and $I_{1, k}$, $k = 1, \dots, n$ in the
increasing order $a_1 < a_2 < \cdots < a_N$. Here $N \le 2(n + 1)$ and $N$ may be strictly less
than $2(n + 1)$ as some of these endpoints may coincide. Then
\begin{eqnarray*}
&& \sum_{k = 1}^n \|f\|_{L_1\left(I_{1, k}, L_{\mathcal{B}}(I_{2, k})\right)}  =
\sum_{k = 1}^n \int_{I_{1, k}} 
\|f(x_1, \cdot)\|^{\rm (av)}_{\mathcal{B}, I_{2, k}}\, dx_1 \\
&& =
\sum_{k = 1}^n \sum_{(a_j, a_{j + 1}) \subseteq I_{1, k}} 
\int_{a_j}^{a_{j + 1}} \|f(x_1, \cdot)\|^{\rm (av)}_{\mathcal{B}, I_{2, k}}\, dx_1 \\
&& = \sum_{j = 1}^{N - 1} \int_{a_j}^{a_{j + 1}} \sum_{\{k : \ I_{1, k} \supseteq (a_j, a_{j + 1})\}}
\|f(x_1, \cdot)\|^{\rm (av)}_{\mathcal{B}, I_{2, k}}\, dx_1 \\
&& \le \sum_{j = 1}^{N - 1} \int_{a_j}^{a_{j + 1}} \|f(x_1, \cdot)\|^{\rm (av)}_{\mathcal{B}, I_2}\, dx_1
= \int_{I_1} \|f(x_1, \cdot)\|^{\rm (av)}_{\mathcal{B}, I_2}\, dx_1
= \|f\|_{L_1\left(I_1, L_{\mathcal{B}}(I_2)\right)} . 
\end{eqnarray*}
The inequality above follows from Lemma 3 in \cite{Sol} and from the implication
$$
I_{1, k} \supseteq (a_j, a_{j + 1}) , \ I_{1, m} \supseteq (a_j, a_{j + 1}) , \ k \not= m \ \ 
\Longrightarrow \ \ I_{2, k}\cap I_{2, m} = \emptyset ,
$$
which holds because the rectangles $I_{1, k}\times I_{2, k}$, $k = 1, \dots, n$ are pairwise disjoint.
\end{proof}

Let $Q := (0, 1)^2$ and $\mathbb{I} := (0, 1)$. 
We will also use the following notation:
$$
w_S := \frac{1}{|S|} \int_S w(x)\, dx ,
$$
where $S \subseteq \mathbb{R}^2$ is a set of a finite positive two dimensional 
Lebesgue measure $|S|$. 

\begin{lemma}\label{CLRl1}{\rm (Cf. \cite[Lemma 2]{Sol})}
There exists $C_{14} > 0$ such that for any nonempty open intervals
$I_1, I_2 \subseteq \mathbb{R}$ of lengths $R_1$ and $R_2$ respectively, any 
$w \in W^1_2(I_1\times I_2)\cap C\left(\overline{I_1\times I_2}\right)$ with 
$w_{I_1\times I_2} = 0$, and 
any $V \in L_1\left(I_1, L_{\mathcal{B}}(I_2)\right)$, $V \ge 0$
the following inequality holds:
\begin{eqnarray}\label{CLRl1eq}
&& \int_{I_1\times I_2}  V(x) |w(x)|^2\, dx \nonumber \\
&& \le C_{14} \max\left\{\frac{R_1}{R_2}, \frac{R_2}{R_1}\right\}
\|V\|_{L_1\left(I_1, L_{\mathcal{B}}(I_2)\right)} \int_{I_1\times I_2} |\nabla w(x)|^2\, dx .
\end{eqnarray}
\end{lemma}
\begin{proof}
Let us start with the case $I_1 = I_2 = \mathbb{I}$, $R_1= R_2 = 1$. 
There exists $C_{15} > 0$ such that
$$
\sup_{x_1 \in \mathbb{I}} \left\|w^2(x_1, \cdot)\right\|_{\mathcal{A}, \mathbb{I}} \le C_{15} \|w\|^2_{W^1_2(Q)}, \ \ \ \forall w \in 
W^1_2(Q)\cap C\left(\overline{Q}\right)
$$
(see \eqref{thepair}). This can be proved by applying the trace theorem (see, e.g., 
\cite[Theorems 4.32 and 7.53]{Ad}) and then using the Yudovich--Pohozhaev--Trudinger
embedding theorem for $H^{1/2}_ 2$ (see \cite{Gri}, \cite{Peet} and \cite[Lemma 1.2.4]{MS}, 
\cite[8.25]{Ad}) or, in one go, by applying a trace inequality of the 
Yudovich--Pohozhaev--Trudinger type (see \cite[Corollary 11.8/2]{Maz}; a sharp result can be found
in \cite{Cia}).

Next, we use the Poincar\'e inequality (see, e.g., \cite[1.11.1]{Maz}): there exists $C_{16} > 0$
such that
\begin{eqnarray*}
\int_Q (w(x) - w_Q)^2\, dx =
\inf_{a \in \mathbb{R}} \int_Q (w(x) - a)^2\, dx \le C_{16} \int_Q |\nabla w(x)|^2\, dx, \\ 
\forall w \in W^1_2(Q) .
\end{eqnarray*}
Hence
$$
\sup_{x_1 \in \mathbb{I}} \left\|w^2(x_1, \cdot)\right\|_{\mathcal{A}, \mathbb{I}}  
\le C_{14} \int_Q |\nabla w(x)|^2\, dx, \ \ \ \forall w \in 
W^1_2(Q)\cap C\left(\overline{Q}\right), \ w_Q = 0 ,
$$
where $C_{14} = C_{15}(C_{16} + 1)$. Now, the H\"older inequality (see \cite[Theorem 9.3]{KR})
implies
\begin{eqnarray*}
\int_Q  V(x) |w(x)|^2\, dx \le 
\|V\|_{L_1\left(\mathbb{I}, L_{\mathcal{B}}(\mathbb{I})\right)} 
\sup_{x_1 \in \mathbb{I}} \left\|w^2(x_1, \cdot)\right\|_{\mathcal{A}, \mathbb{I}} \\
\le C_{14} \|V\|_{L_1\left(\mathbb{I}, L_{\mathcal{B}}(\mathbb{I})\right)}   
\int_Q |\nabla w(x)|^2\, dx ,
\end{eqnarray*}
i.e. \eqref{CLRl1eq} holds for $I_1 = I_2 = \mathbb{I}$. 

Consider now arbitrary intervals $I_1, I_2 \subseteq \mathbb{R}$. Then
$I_1\times I_2 = \xi(Q)$, where $\xi$ is the affine mapping from
Lemma \ref{affine}, and \eqref{CLRl1eq} can be derived from the above inequality with 
the help of Lemma \ref{affine} and the inequality
$$
\int_Q |\nabla (w\circ\xi)(x)|^2\, dx \le \max\left\{\frac{R_1}{R_2}, \frac{R_2}{R_1}\right\}
\int_{I_1\times I_2} |\nabla w(y)|^2\, dy
$$
(see the the proof of Lemma 2 in \cite{Sol}).
\end{proof}

\begin{lemma}\label{CLRl2}{\rm (Cf. \cite[Theorem 1]{Sol})}
For any $V \in L_1\left(\mathbb{I}, L_{\mathcal{B}}(\mathbb{I})\right)$, $V \ge 0$ and any 
$n \in \mathbb{N}$  there exists a finite cover of $Q$ by
rectangles $S_k = I_{1, k}\times I_{2, k}$, $k = 1, \dots, n_0$  such that $n_0 \le n$ and
\begin{equation}\label{CLRl2eq}
\int_Q  V(x) |w(x)|^2\, dx \le  
C_{17} n^{-1}\|V\|_{L_1\left(\mathbb{I}, L_{\mathcal{B}}(\mathbb{I})\right)} 
\int_{Q} |\nabla w(x)|^2\, dx 
\end{equation}
for all
$w \in W^1_2(Q)\cap C\left(\overline{Q}\right)$ with $w_{S_k} = 0$, $k = 1, \dots, n_0$, where 
the constant $C_{17}$ does not depend on $V$.
\end{lemma}
\begin{proof}
According to the Besicovitch covering lemma (see, e.g., \cite[Ch. I. Theorem 1.1]{Guz}),
there exists a constant $\nu \in \mathbb{N}$ such that any cover $\{\Delta_x\}_{x \in \overline{Q}}$
of $\overline{Q}$ by closed squares $\Delta_x$ centered at $x$ has a countable or a finite subcover
that can be split into $\nu$ families in such a way that any two squares belonging to the same
family are disjoint. 

We can assume that $n > \nu$, as otherwise one could take $n_0 = 1$, $S_k = Q$ and
get \eqref{CLRl2eq} with $C_{17} \ge \nu C_{14}$ directly from \eqref{CLRl1eq}.
For any $x \in \overline{Q}$ there exists a closed square $\Delta_x = I^0_{1, x}\times I^0_{2, x}$ 
centered at $x$ such that 
$$
\|V\|_{L_1\left(I_{1, x}, L_{\mathcal{B}}(I_{2, x})\right)}  = 
\nu n^{-1}\|V\|_{L_1\left(\mathbb{I}, L_{\mathcal{B}}(\mathbb{I})\right)} ,
$$
where 
$$
I_{i, x} := I^0_{i, x}\cap [0, 1] , \ i = 1, 2 , \ \ \mbox{ i.e. } \ \
I_{1, x}\times I_{2, x}  = \Delta_x \cap\overline{Q} 
$$
(see \cite[Lemma 4]{Sol}). Let $\{\Delta_{x_k}\}$ be the subcover from the Besicovitch covering 
lemma and let $S_k = I_{1, k}\times I_{2, k} := I_{1, x_k}\times I_{2, x_k}  = 
\Delta_{x_k} \cap\overline{Q}$. Then $\Xi := \{S_k\}$ is also a cover of $\overline{Q}$ and, 
like $\{\Delta_{x_k}\}$, it can be split into $\nu$ families $\Xi_l$, $l = 1, \dots, \nu$, 
consisting of pairwise disjoint elements. Lemma \ref{subadd} implies
$$
\nu n^{-1}\|V\|_{L_1\left(\mathbb{I}, L_{\mathcal{B}}(\mathbb{I})\right)}\, \mbox{card}\, \Xi_l 
= \sum_{S_k \in \Xi_l} \|V\|_{L_1\left(I_{1, k}, L_{\mathcal{B}}(I_{2, k})\right)} \le
\|V\|_{L_1\left(\mathbb{I}, L_{\mathcal{B}}(\mathbb{I})\right)} .
$$
Hence $\mbox{card}\, \Xi_l \le n/\nu$ and
$$
n_0 := \mbox{card}\, \Xi =  \sum_{l = 1}^\nu \mbox{card}\, \Xi_l \le \nu n/\nu = n .
$$

Take any $w \in W^1_2(Q)\cap C\left(\overline{Q}\right)$ with $w_{S_k} = 0$, $k = 1, \dots, n_0$. 
Since the centre of the square $\Delta_{x_k}$ belongs to $\overline{Q}$, the ratio of the side
lengths of the rectangle $S_k = \Delta_{x_k} \cap\overline{Q}$ is between $1/2$ and $2$.
Hence it follows from Lemma \ref{CLRl1} that
\begin{eqnarray*}
&& \int_Q  V(x) |w(x)|^2\, dx \le \sum_{k = 1}^{n_0} \int_{S_k}  V(x) |w(x)|^2\, dx  \\
&& \le 2 C_{14} \sum_{k = 1}^{n_0} \|V\|_{L_1\left(I_{1, k}, L_{\mathcal{B}}(I_{2, k})\right)}
\int_{S_k} |\nabla w(x)|^2\, dx  \\
&& =  2 C_{14} \nu n^{-1}\|V\|_{L_1\left(\mathbb{I}, L_{\mathcal{B}}(\mathbb{I})\right)}
\sum_{k = 1}^{n_0} 
\int_{S_k} |\nabla w(x)|^2\, dx  \\
&& =  2 C_{14} \nu n^{-1}\|V\|_{L_1\left(\mathbb{I}, L_{\mathcal{B}}(\mathbb{I})\right)}
\sum_{l = 1}^\nu \sum_{S_k \in \Xi_l} 
\int_{S_k} |\nabla w(x)|^2\, dx\\
&& =  2 C_{14} \nu n^{-1}\|V\|_{L_1\left(\mathbb{I}, L_{\mathcal{B}}(\mathbb{I})\right)}
\sum_{l = 1}^\nu \int_Q |\nabla w(x)|^2\, dx \\
&& = C_{17} n^{-1}\|V\|_{L_1\left(\mathbb{I}, L_{\mathcal{B}}(\mathbb{I})\right)}
\int_Q |\nabla w(x)|^2\, dx ,
\end{eqnarray*}
where $C_{17} := 2 C_{14} \nu^2$.
\end{proof}

Let
\begin{eqnarray*}
& \mathcal{E}_{V, 0}[w] : = \int_{Q} |\nabla w(x)|^2 dx - 
\int_{Q} V(x) |w(x)|^2 dx , & \\  
& \mbox{Dom}\, (\mathcal{E}_{V, 0}) =
W^1_2\left(Q\right)\cap L_2\left(Q, V(x)dx\right) . & 
\end{eqnarray*}

\begin{lemma}\label{CLRl3}{\rm (Cf. \cite[Theorem 4]{Sol})}
$$
N_- (\mathcal{E}_{V, 0}) \le C_{17} 
\|V\|_{L_1\left(\mathbb{I}, L_{\mathcal{B}}(\mathbb{I})\right)} + 1 , \ \ \ \forall V \ge 0 ,
$$
where $C_{17}$ is the same as in Lemma \ref{CLRl2}.
\end{lemma}
\begin{proof}
Let $n = \left[C_{17} \|V\|_{L_1\left(\mathbb{I}, L_{\mathcal{B}}(\mathbb{I})\right)}\right] + 1$
in Lemma \ref{CLRl2}. Take any
linear subspace $\mathcal{L} \subset \mbox{Dom}\, (\mathcal{E}_{V, 0})$
such that 
$$
\dim \mathcal{L} > \left[C_{17} \|V\|_{L_1\left(\mathbb{I}, L_{\mathcal{B}}(\mathbb{I})\right)}\right] + 1 .
$$
Since $n_0 \le n$, there exists $w \in \mathcal{L}\setminus\{0\}$ such that 
$w_{S_k} = 0$, $k = 1, \dots, n_0$. Then
\begin{eqnarray*}
\mathcal{E}_{V, 0}[w]  &=& \int_{Q} |\nabla w(x)|^2 dx - 
\int_{Q} V(x) |w(x)|^2 dx \\
&\ge& \int_{Q} |\nabla w(x)|^2 dx - 
\frac{C_{17} \|V\|_{L_1\left(\mathbb{I}, L_{\mathcal{B}}(\mathbb{I})\right)}}{\left[C_{17} 
\|V\|_{L_1\left(\mathbb{I}, L_{\mathcal{B}}(\mathbb{I})\right)}\right] + 1}\,
\int_{Q} |\nabla w(x)|^2 dx \\
&\ge& \int_{Q} |\nabla w(x)|^2 dx - \int_{Q} |\nabla w(x)|^2 dx = 0 .
\end{eqnarray*}
Hence
$$
N_- (\mathcal{E}_V) \le \left[C_{17} 
\|V\|_{L_1\left(\mathbb{I}, L_{\mathcal{B}}(\mathbb{I})\right)}\right] + 1 \le
C_{17} \|V\|_{L_1\left(\mathbb{I}, L_{\mathcal{B}}(\mathbb{I})\right)} + 1.
$$
\end{proof}

Let $R > 0$, \ $I(R) := (R, eR)$, \
$\Omega(R) := \{x \in \mathbb{R}^2 : \ R < |x| < e R\}$ and
\begin{eqnarray*}
& \mathcal{E}^R_V[w] : = \int_{\Omega(R)} |\nabla w(x)|^2 dx - 
\int_{\Omega(R)} V(x) |w(x)|^2 dx , & \\  
& \mbox{Dom}\, (\mathcal{E}^R_V) =
\left\{w \in W^1_2\left(\Omega(R)\right)\cap L_2\left(\Omega(R), V(x)dx\right) : \
w_{\Omega(R)} = 0\right\}. & 
\end{eqnarray*}

\begin{lemma}\label{CLRl4}{\rm (Cf. \cite[Theorem $4'$]{Sol})}
There exists $C_{18} > 0$ such that
\begin{equation}\label{CLRl4est}
N_- (\mathcal{E}^R_V) \le C_{18} 
\|V\|_{\mathcal{L}_1\left(I(R), L_{\mathcal{B}}(\mathbb{S})\right)} , \ \ \ \forall V \ge 0,
\ \ \ \forall R > 0 .
\end{equation}
\end{lemma}
\begin{proof}
We start with the case $R = 1$. Let $\mathbb{S}_+ := (0, \pi)$, $\mathbb{S}_- := (-\pi, 0)$,
\begin{eqnarray*}
& \Omega_\pm := \{(r \cos\vartheta, r \sin\vartheta) \in \mathbb{R}^2 : \
1 < r < e, \ \vartheta \in \mathbb{S}_\pm\} & \\
& \mathcal{E}^\pm_{V}[w] : = \int_{\Omega_\pm} |\nabla w(x)|^2 dx - 
\int_{\Omega_\pm} V(x) |w(x)|^2 dx , & \\  
& \mbox{Dom}\, (\mathcal{E}^\pm_V) =
W^1_2\left(\Omega_\pm\right)\cap L_2\left(\Omega_\pm, V(x)dx\right) . & 
\end{eqnarray*}
Then by the variational principle,
\begin{equation}\label{addit2}
N_- (\mathcal{E}^1_V) \le N_- (\mathcal{E}^+_V) + N_- (\mathcal{E}^-_V) .
\end{equation}
(see, e.g., \cite[Lemma 3.5]{Grig}). Consider the diffeomorphisms 
$\varphi_\pm : Q \to \Omega_\pm$,
$$
\varphi_\pm(y_1, y_2) = \left(\big(1 + (e - 1)y_1\big) \cos(\pm\pi y_2) , 
\big(1 + (e - 1)y_1\big) \sin(\pm\pi y_2)\right) .
$$
The polar coordinates of $\varphi_\pm(y_1, y_2)$ are
$r = 1 + (e - 1)y_1$, \ $\vartheta = \pm\pi y_2$.

Let
$$
\hat{V}_\pm(y) := V(\varphi_\pm(y)) , \ \ \ \tilde{w}(y) := w(\varphi_\pm(y)) , \ 
w \in \mbox{Dom}\, (\mathcal{E}^\pm_V) .
$$
There exist absolute constants $C_{19} > 0$ and $C_{20} > 0$ such that
\begin{eqnarray*}
&& \int_{\Omega_\pm} |\nabla w(x)|^2 dx \ge 
\frac1{C_{19}} \int_Q |\nabla \tilde{w}(y)|^2 dy , \\
&& \int_{\Omega_\pm} V(x) |w(x)|^2 dx \le C_{20} \int_Q \hat{V}_\pm(y) |\tilde{w}(y)|^2 dy .
\end{eqnarray*}
Then Lemma \ref{CLRl3} implies 
$$
N_- (\mathcal{E}^\pm_V) \le N_- (\mathcal{E}_{C_{19} C_{20} \hat{V}_\pm, 0}) \le
C_{17} C_{19} C_{20}  
\left\|\hat{V}_\pm\right\|_{L_1\left(\mathbb{I}, L_{\mathcal{B}}(\mathbb{I})\right)} + 1 .
$$
Using \cite[Lemma 1]{Sol} one gets
$$
\left\|\hat{V}_\pm\right\|_{L_1\left(\mathbb{I}, L_{\mathcal{B}}(\mathbb{I})\right)} =
\int_{\mathbb{I}} 
\left\|\hat{V}_\pm(y_1, \cdot)\right\|^{\rm (av)}_{\mathcal{B}, \mathbb{I}}\, dy_1
= \frac{1}{\pi(e - 1)} \int_1^e \|V(r, \cdot)\|^{\rm (av)}_{\mathcal{B}, \mathbb{S}_\pm}\, dr .
$$
Now it follows from \eqref{addit2}, Lemma \ref{avequiv} and \cite[Lemma 3]{Sol} that
\begin{eqnarray}\label{plus2}
N_- (\mathcal{E}^1_V) \le C_{21} \int_1^e 
\left(\|V(r, \cdot)\|^{\rm (av)}_{\mathcal{B}, \mathbb{S}_+}
+ \|V(r, \cdot)\|^{\rm (av)}_{\mathcal{B}, \mathbb{S}_-}\right) dr + 2 \nonumber \\
\le C_{21} \int_1^e 
\|V(r, \cdot)\|^{\rm (av)}_{\mathcal{B}, \mathbb{S}}\, dr + 2 \le 
C_{21} \int_1^e 
\|V(r, \cdot)\|^{\rm (av)}_{\mathcal{B}, \mathbb{S}}\, r dr + 2 \\
\le 2\pi C_{21} \|V\|_{\mathcal{L}_1\left(I(1), L_{\mathcal{B}}(\mathbb{S})\right)} + 2 , \nonumber
\end{eqnarray}
where $C_{21} = \frac{C_{17} C_{19} C_{20}}{\pi(e - 1)}$.

Using the Yudovich--Pohozhaev--Trudinger
embedding theorem, the Poincar\'e inequality and the H\"older inequality for Orlicz spaces
as in the proof of Lemma \ref{CLRl1} one can prove the existence of a constant
$C_{22} > 0$ such that
\begin{eqnarray*}
\int_{\Omega(1)} V(x) |w(x)|^2 dx \le C_{22} 
\|V\|_{\mathcal{L}_1\left(I(1), L_{\mathcal{B}}(\mathbb{S})\right)}
\int_{\Omega(1)} |\nabla w(x)|^2 dx , \\ 
\forall w \in \mbox{Dom}\, (\mathcal{E}^1_V) .
\end{eqnarray*}
If $\|V\|_{\mathcal{L}_1\left(I(1), L_{\mathcal{B}}(\mathbb{S})\right)} \le 1/C_{22}$, then
$$
\mathcal{E}^1_V[w] \ge 0 , \ \ \  \forall w \in \mbox{Dom}\, (\mathcal{E}^1_V) ,
$$
i.e. $N_- (\mathcal{E}^1_V) = 0$. Combining this 
with \eqref{plus2} one gets the existence of a constant $C_{18} > 0$ for which 
\eqref{CLRl4est} holds in the case $R = 1$. The case of a general $R$ is reduced to $R = 1$
in the standard way with the help of the scaling $x \mapsto Rx$ (cf. the proof of Theorem 
$4'$ in \cite{Sol}).
\end{proof}

\begin{proof}[Proof of Theorem \ref{LaptNetrSol}] 
The theorem is proved in the same way as Theorem \ref{GrigTalk} (see \eqref{BSest}, 
\eqref{RadMainEst} and \eqref{VNn}). One only needs to use
Lemma \ref{CLRl4} instead of Theorem $4'$ of \cite{Sol}.
\end{proof}

\begin{remark}\label{Laptev}
{\rm It has been observed by A. Laptev that interchanging the r\^oles of the variables $y_1$ and
$y_2$ in the proof of Lemma \ref{CLRl4} one can get an analogue of \eqref{CLRl4est} where the
$L_{\mathcal{B}}$ norm is taken with respect to the radial variable, while the $L_1$ norm is taken
with respect to the angular one. Then using the Birman-Solomyak method exactly as in the proofs of
Theorems \ref{GrigTalk} and  \ref{LaptNetrSol}, one shows the existence of  constants
$C_{23} > 0$ and $c > 0$ such that 
\begin{equation}\label{LaptevEst}
N_- (\mathcal{E}_V) \le 1 + 4 \sum_{A_n > 1/4}  \sqrt{A_n}
+ C_{23} \sum_{\mathcal{G}_n > c} \mathcal{G}_n\,  , \ \ \
\forall V\ge 0 ,
\end{equation}
where 
\begin{equation}\label{Gn}
\mathcal{G}_n := e^n 
\int_{\mathbb{S}} \|V(\cdot, \vartheta)\|^{\rm (av)}_{\mathcal{B}, \mathcal{I}_n}\, d\vartheta\, ,
\ \ \ n \in \mathbb{Z} 
\end{equation}
(cf. \eqref{curlyLnorm}, \eqref{InDn}), and the factor $e^n$ in front of the integral ensures the
correct scaling under the mapping $x \mapsto e^n x$.}
\end{remark}

\section{Comparison of various estimates}\label{Compar}

Let us start with comparing \eqref{clCLRest} and \eqref{SolMainEst}. Note first of all that 
\eqref{SolMainEst} is equivalent to
\begin{equation}\label{SolMainEst'}
N_- (\mathcal{E}_V) \le 1 + C_{24} \left(\left\|(\mathbf{A}_n)_{n \ge 1}\right\|_{1, \infty}
+ \sum_{n = 0}^\infty \mathbf{B}_n\right)  , \ \ \
\forall V\ge 0 , 
\end{equation}
where we have dropped the term $\mathbf{A}_0$. Indeed, the H\"older inequality 
(see \cite[\S 3.3, (17)]{RR}) implies
\begin{eqnarray*}
&&\mathbf{A}_0 = \int_{|x| \le e} V(x) |\ln|x||\, dx \le  \|V\|_{\mathcal{B}, \mathbf{\Omega}_0} 
\||\ln|\cdot||\|_{(\mathcal{A}, \mathbf{\Omega}_0)} , \\
&& \int_{\mathbf{\Omega}_0} \mathcal{A}\left( |\ln|x||\right)\, dx \le
\int_{|x| \le 1} e^{\ln \frac{1}{|x|}}\, dx + \int_{1 < |x| \le e} e^{\ln |x|}\, dx\\
&& = \int_{|x| \le 1} \frac{1}{|x|}\, dx + \int_{1 < |x| \le e}  |x|\, dx \\
&& = \int_{-\pi}^\pi \int_0^1 1\, dr d\vartheta + \int_{-\pi}^\pi \int_1^e r^2\, dr d\vartheta
= \frac{2\pi}3\, (e^3 + 2) =: C_{25} .
\end{eqnarray*}
Hence
$$
\||\ln|\cdot||\|_{(\mathcal{A}, \mathbf{\Omega}_0)} \le C_{25}
$$
(see \eqref{LuxNormImpl}) and
\begin{equation}\label{A0B0}
\mathbf{A}_0 \le C_{25} \|V\|_{\mathcal{B}, \mathbf{\Omega}_0} \le
C_{25} \|V\|^{\rm (av)}_{\mathcal{B}, \mathbf{\Omega}_0} = C_{25}
\mathbf{B}_0 .
\end{equation}

It is clear that
\begin{eqnarray}\label{logcomp}
\left\|(\mathbf{A}_n)_{n \ge 1}\right\|_{1, \infty} \le \left\|(\mathbf{A}_n)_{n \ge 1}\right\|_1 =
\sum_{n = 1}^\infty \int_{U_n} V(x) \ln|x|\, dx \nonumber \\
=  \int_{|x| > e} V(x) \ln|x|\, dx \le
\int_{\mathbb{R}^2} V(x) \ln(1 + |x|)\, dx .
\end{eqnarray}
It is also easy to see that if 
\begin{equation}\label{log3}
V(x) = \left\{\begin{array}{cc}
 \frac{1}{2\pi |x|^2 (\ln|x|)^2 \ln\ln|x|}\, ,   &  |x| \ge e^2 ,  \\ \\
   0 ,  &   x| < e^2 ,
\end{array}
\right.
\end{equation}
then $\mathbf{A}_n = \ln\left(1 + \frac1{n -1}\right) = 
\frac1n + O\left(n^{-2}\right)$, $n \ge 2$, and the left-hand side of \eqref{logcomp}
is finite, while the right-hand side is infinite.

On the other hand,
\begin{eqnarray}\label{Orlcomp}
\sum_{n = 0}^\infty \mathbf{B}_n \ge \|V\|_{\mathcal{B}, \mathbf{\Omega}_0} + 
\sum_{n = 1}^\infty \|V\|_{\mathcal{B}, \Omega_n} \ge 
\|V\|_{\mathcal{B}, \mathbb{R}^2} ,
\end{eqnarray}
where the last estimate follows from the triangle inequality applied to 
$f = f\chi_{\mathbf{\Omega}_0} + \sum_{n = 1}^\infty f\chi_{\Omega_n}$ with
$\chi_{\Omega}(x) := 1$ for $x \in \Omega$ and 
$\chi_{\Omega}(x) := 0$ for $x \not\in \Omega$. It is not difficult to show that there exist a $V$ for
which the left-hand side of \eqref{Orlcomp} is infinite while the middle term is finite, and a $V$ for
which   the middle term is infinite while the right-hand side is finite. Indeed, let 
$V(x) = \frac{1}{|x|^2}$ for $|x| \ge e$ and $V(x) = 0$ for $|x| < e$. Then using \cite[(7)]{Sol}
and the notation from Remark \ref{GNimplied} one gets
\begin{eqnarray*}
\sum_{n = 1}^\infty \mathbf{B}_n = \sum_{n = 1}^\infty \mathcal{B}_n \ge 
\sum_{n = 1}^\infty e^{-2(n + 1)} \|1\|^{\rm (av)}_{\mathcal{B}, \Omega_n}
= \sum_{n = 1}^\infty e^{-2(n + 1)} |\Omega_n| \|1\|_{\mathcal{B}, \Delta} \\
= \|1\|_{\mathcal{B}, \Delta} \sum_{n = 1}^\infty \pi \left(1 - \frac1{e^2}\right) = +\infty .
\end{eqnarray*}
On the other hand, \cite[(9.11)]{KR} and Lemma \ref{ABelem} imply 
\begin{eqnarray*}
\sum_{n = 1}^\infty \|V\|_{\mathcal{B}, \Omega_n} \le \sum_{n = 1}^\infty e^{-2n}
\|1\|_{\mathcal{B}, \Omega_n} = \sum_{n = 1}^\infty e^{-2n} |\Omega_n| 
\mathcal{A}^{-1}\left(\frac1{|\Omega_n| }\right) \\
\le \sum_{n = 1}^\infty e^{-2n} \sqrt{2|\Omega_n|} =  \sum_{n = 1}^\infty e^{-n} 
\sqrt{2\pi (e^2 - 1)} < +\infty .
\end{eqnarray*}

Take now $V(x) = \frac{1}{|x| \ln|x|}$ for $|x| \ge e$ and $V(x) = 0$ for $|x| < e$. Then
\begin{eqnarray*}
\sum_{n = 1}^\infty \|V\|_{\mathcal{B}, \Omega_n} \ge \sum_{n = 1}^\infty \frac{e^{-(n + 1)}}{n + 1}\,
\|1\|_{\mathcal{B}, \Omega_n} = \sum_{n = 1}^\infty \frac{e^{-(n + 1)}}{n + 1}\, |\Omega_n| 
\mathcal{A}^{-1}\left(\frac1{|\Omega_n| }\right) \\
\ge \sum_{n = 1}^\infty \frac{e^{-(n + 1)}}{n + 1}\,  \sqrt{\frac2e\, |\Omega_n|} =  
 \sum_{n = 1}^\infty \frac1{n + 1}\, 
\sqrt{\frac{2\pi}e\,  \left(1 - \frac1{e^2}\right)} = +\infty ,
\end{eqnarray*}
while
\begin{eqnarray*}
 \int_{\mathbb{R}^2} \mathcal{B}\left(V(x)\right)\, dx = 
 \int_{|x| \ge e} \mathcal{B}\left(\frac{1}{|x| \ln|x|}\right)\, dx 
= 2\pi \int_{r \ge e} \mathcal{B}\left(\frac{1}{r \ln r}\right)\, r dr \\
\le \pi \int_{r \ge e} \frac{1}{r \ln^2 r}\,  dr = \pi 
\end{eqnarray*}
(see Lemma \ref{ABelem}) and hence $\|V\|_{\mathcal{B}, \mathbb{R}^2} < \infty$.

The above discussion of \eqref{logcomp} and \eqref{Orlcomp} may give one an impression that
neither of \eqref{clCLRest} and \eqref{SolMainEst} is stronger than the other one. It turns out that
this is not the case, and it follows from \eqref{A0B0}, \eqref{logcomp} and 
Lemmas \ref{avequiv}, \ref{avequivB} and \ref{reversetr} (see below) 
that \eqref{SolMainEst} is actually stronger than \eqref{clCLRest}. Indeed,
\begin{eqnarray}\label{Solimpl}
&& \left\|(\mathbf{A}_n)_{n \ge 0}\right\|_{1, \infty}
+ \sum_{n = 0}^\infty \mathbf{B}_n \le \mbox{const}\, 
\left(\left\|(\mathbf{A}_n)_{n \ge 1}\right\|_{1, \infty}
+ \sum_{n = 0}^\infty \mathbf{B}_n\right) \nonumber \\
&& \le \mbox{const}\, \Big(\int_{\mathbb{R}^2} V(x) \ln(1 + |x|)\, dx
+ \pi e^2 \|V\|_{\mathcal{B}, \mathbf{ \Omega}_0} \nonumber \\
&& + \sum_{n = 1}^\infty \|V\|_{\mathcal{B}, \Omega_n} +  
\sum_{n = 1}^\infty \ln\left(\frac72\, |\Omega_n|\right)\, 
\int_{\Omega_n} V(x)\, dx\Big) \\
&& \le \mbox{const}\, \left(\int_{\mathbb{R}^2} V(x) \ln(1 + |x|)\, dx
+ \|V\|_{\mathcal{B}, \mathbb{R}^2} 
+  \int_{|x| > e} V(x) \ln\ln|x|\, dx\right) \nonumber \\
&& \le \mbox{const}\, \left(\|V\|_{\mathcal{B}, \mathbb{R}^2} 
+ \int_{\mathbb{R}^2} V(x) \ln(1 + |x|)\, dx\right) , \ \ \ \forall V \ge 0 . \nonumber
\end{eqnarray}
Note that for $V$ given by \eqref{log3} one has
\begin{eqnarray*}
\sum_{n = 0}^\infty \mathbf{B}_n = \sum_{n = 2}^\infty \mathcal{B}_n \le 
\sum_{n = 2}^\infty \frac{1}{2\pi e^{2n} n^2 \ln n}
 \|1\|^{\rm (av)}_{\mathcal{B}, \Omega_n}
= \|1\|_{\mathcal{B}, \Delta} \sum_{n = 2}^\infty \frac{|\Omega_n|}{2\pi e^{2n} n^2 \ln n} \\
= \|1\|_{\mathcal{B}, \Delta} \sum_{n = 2}^\infty \frac{e^2 - 1}{2 n^2 \ln n} < +\infty .
\end{eqnarray*}
Hence the right-hand side of \eqref{SolMainEst} is finite for this $V$ while the right-hand side of
\eqref{clCLRest} is infinite. So, \eqref{SolMainEst} is strictly stronger than \eqref{clCLRest}.

\begin{lemma}\label{reversetr}
There exists $C_{26} > 0$ such that
$$
\sum_{n = 1}^\infty \|V\|_{\mathcal{B}, \Omega_n}  \le C_{26}
\left(\|V\|_{\mathcal{B}, \mathbb{R}^2\setminus\mathbf{\Omega}_0} +
\int_{|x| > e} V(x) \ln\ln|x|\, dx\right) , \ \ \ \forall V \ge 0
$$
(see \eqref{ringsN}, \eqref{bald0}).
\end{lemma}
\begin{proof}
Suppose first that $\|V\|_{(\mathcal{B}, \mathbb{R}^2\setminus\mathbf{\Omega}_0)} = 1$ and let
$$
\alpha_n := \int_{\Omega_n} \mathcal{B}(V(x))\, dx , \ \ \ 
\kappa_n := \|V\|_{(\mathcal{B}, \Omega_n)} , \ \ \ n \in \mathbb{N} .
$$
Then
$$
\sum_{n = 1}^\infty \alpha_n = \sum_{n = 1}^\infty \int_{\Omega_n} \mathcal{B}(V(x))\, dx =
\int_{\mathbb{R}^2\setminus\mathbf{\Omega}_0} \mathcal{B}(V(x))\, dx = 1 , 
$$
$$
\alpha_n \le 1 \ \Longrightarrow \ \kappa_n \le 1 ,
$$
and
\begin{eqnarray*}
1 =  \int_{\Omega_n} \mathcal{B}\left(\frac{V(x)}{\kappa_n}\right) dx \le 
\int_{\Omega_n} \left(\frac{V(x)}{\kappa_n} + 
2\frac{V(x)}{\kappa_n} \ln_+\frac{V(x)}{\kappa_n}\right) dx \\
\le \frac1{\kappa_n} \int_{\Omega_n} \left(V(x) + 2 V(x) \ln_+V(x)\right) dx +
\frac2{\kappa_n}\, \ln\frac1{\kappa_n}\, \|V\|_{L_1(\Omega_n)} \\
\le \frac4{\kappa_n}\, \alpha_n + \frac1{\kappa_n}\left(1 + 2 \ln\frac1{\kappa_n}\right)
\|V\|_{L_1(\Omega_n)} 
\end{eqnarray*}
(see Lemma \ref{elem}).
Hence
$$
\kappa_n \le 4\alpha_n + \left(1 + 2 \ln\frac1{\kappa_n}\right)
\|V\|_{L_1(\Omega_n)}
$$
and
\begin{eqnarray*}
\sum_{n = 1}^\infty \|V\|_{\mathcal{B}, \Omega_n} \le 2 \sum_{n = 1}^\infty \kappa_n =
2 \sum_{\kappa_n \le 1/n^2} \kappa_n + 2 \sum_{\kappa_n > 1/n^2} \kappa_n \\
\le 2 \sum_{n = 1}^\infty \frac1{n^2} + 8 \sum_{n = 1}^\infty \alpha_n +
2 \sum_{n = 1}^\infty (1 + 4 \ln n)\|V\|_{L_1(\Omega_n)} \\
= \frac{\pi^2}{3} + 8 + 2 \sum_{n = 1}^\infty (1 + 4 \ln n)
\int_{e^n < |x| < e^{n + 1}} V(x)\, dx \\
\le \frac{\pi^2}{3} + 8 + 2 \int_{|x| > e} V(x)(1 + 4 \ln\ln|x|)\, dx \\
\le C_{26}
\left(\|V\|_{\mathcal{B}, \mathbb{R}^2\setminus\mathbf{\Omega}_0} +
\int_{|x| > e} V(x) \ln\ln|x|\, dx\right) 
\end{eqnarray*}
(see \eqref{Luxemburgequiv}).
The case of a general $V$ is reduced to 
$\|V\|_{(\mathcal{B}, \mathbb{R}^2\setminus\mathbf{\Omega}_0)} = 1$ by the scaling
$V \mapsto t V$, $t > 0$. 
\end{proof}

Let us now show that each of \eqref{LaptNetrSolEst} and \eqref{LaptevEst}
implies \eqref{GrigTalkEst} (with a different $c$). 
Suppose $\|V\|_{(\mathcal{B}, \Omega_0)} = 1$.
It follows from \eqref{Luxemburgequiv}, \eqref{LuxNormPre} and Lemma \ref{avequiv} that
\begin{eqnarray*}
&& \mathcal{D}_0 = \int_1^e \|V(r, \cdot)\|_{\mathcal{B}, \mathbb{S}}\, r dr \le
2 \int_1^e \left(1 + \int_\mathbb{S} \mathcal{B}(V(r, \vartheta))\, d\vartheta\right) r dr \\
&& \le (e^2 - 1) + 2 \int_{\Omega_0} \mathcal{B}(V(x))\, dx = e^2 + 1 =
(e^2 + 1) \|V\|_{(\mathcal{B}, \Omega_0)} \\
&& \le  (e^2 + 1) \|V\|_{\mathcal{B}, \Omega_0} \le  
(e^2 + 1) \|V\|^{\rm (av)}_{\mathcal{B}, \Omega_0} = (e^2 + 1) \mathcal{B}_0 , \\ \\
&& \mathcal{G}_0 =  
\int_{\mathbb{S}} \|V(\cdot, \vartheta)\|^{\rm (av)}_{\mathcal{B}, \mathcal{I}_0}\, d\vartheta
\le 2(e - 1) \int_{\mathbb{S}} \left(1 + \int_1^e \mathcal{B}(V(r, \vartheta))\, dr\right)\, d\vartheta \\
&& \le 2(e - 1)\left(2\pi + \int_{\Omega_0} \mathcal{B}(V(x))\, dx\right) = 2(e - 1)(2\pi + 1) \\
&& = 2(e - 1)(2\pi + 1) \|V\|_{(\mathcal{B}, \Omega_0)} \le 2(e - 1)(2\pi + 1) \mathcal{B}_0
=: C_{27} \mathcal{B}_0 .
\end{eqnarray*}
The scaling $V \mapsto t V$, $t > 0$ allows one to extend the inequalities  
$\mathcal{D}_0 \le (e^2 + 1) \mathcal{B}_0$ and $\mathcal{G}_0 \le C_{27} \mathcal{B}_0$
to arbitrary $V \ge 0$. Using the
scaling $x \mapsto e^n x$ of the independent variable and \cite[Lemma 1]{Sol}, one gets 
$\mathcal{D}_n \le (e^2 + 1) \mathcal{B}_n$ and $\mathcal{G}_n \le C_{27} \mathcal{B}_n$,
$\forall n \in \mathbb{Z}$. Hence
\begin{eqnarray*}
&& \sum_{\{n \in \mathbb{Z} : \ \mathcal{D}_n > c\}} \mathcal{D}_n \le
(e^2 + 1) \sum_{\{n \in \mathbb{Z} : \ \mathcal{D}_n > c\}} \mathcal{B}_n \le
(e^2 + 1) \sum_{\{n \in \mathbb{Z} : \ \mathcal{B}_n > c/(e^2 + 1)\}} \mathcal{B}_n\, , \\ \\
&& \sum_{\{n \in \mathbb{Z} : \ \mathcal{G}_n > c\}} \mathcal{G}_n \le
C_{27} \sum_{\{n \in \mathbb{Z} : \ \mathcal{G}_n > c\}} \mathcal{B}_n \le
C_{27}  \sum_{\{n \in \mathbb{Z} : \ \mathcal{B}_n > c/C_{27}\}} \mathcal{B}_n\, .
\end{eqnarray*}
Similarly,
\begin{eqnarray*}
\|V\|_{\mathcal{L}_1\left(\mathbb{R}_+, L_{\mathcal{B}}(\mathbb{S})\right)} =
\sum_{n \in \mathbb{Z}} \mathcal{D}_n \le
(e^2 + 1) \sum_{n \in \mathbb{Z}} \mathcal{B}_n ,
\end{eqnarray*}
and \eqref{LaptNetrSolEst2} implies \eqref{SolMainEst}.

Putting together what we have obtained so far, we get the diagram below. 
For the convenience of the reader, we precede the diagram with a list of the estimates 
discussed above:
{\footnotesize
$$
N_- (\mathcal{E}_V) \le C  \left(\|V\|_{\mathcal{B}, \mathbb{R}^2} 
+ \int_{\mathbb{R}^2} V(x) \ln(1 + |x|)\, dx\right)  + 1, \eqno{\eqref{clCLRest}}
$$
$$
N_- (\mathcal{E}_V) \le C_6  \left(\int_{\mathbb{R}^2} V(x) \ln(2 + |x|)\, dx + 
\int_{\mathbb{R}^2} V_*(|x|) \ln_+ \frac{1}{|x|}\, dx\right)  + 1 ,  \eqno{\eqref{myKMWest}}
$$
$$
N_- (\mathcal{E}_V) \le  c_1 \int_{\mathbb{R}^2} V_*(|x|) \ln_+ \frac{1}{|x|}\, dx +
c_2 \int_{\mathbb{R}^2} V(x) \ln_+ |x|\, dx 
+ c_3 \int_{\mathbb{R}^2} V(x)\, dx + 1  , \eqno{ \eqref{KMWest}}
$$
$$
N_- (\mathcal{E}_V) \le 1 + 4 \sum_{A_n > 1/4}  \sqrt{A_n}
+ C_7 \sum_{\mathcal{B}_n > c} \mathcal{B}_n  ,  \eqno{\eqref{GrigTalkEst}}
$$
$$
N_- (\mathcal{E}_V) \le 1 + C_8 \left(\left\|(\mathbf{A}_n)_{n \ge 0}\right\|_{1, \infty}
+ \sum_{n = 0}^\infty \mathbf{B}_n\right)  , \eqno{\eqref{SolMainEst}}
$$
$$
N_- (\mathcal{E}_V) \le 1 + C_7 \sum_{\{n \in \mathbb{Z} : \ A_n > c\}}  \sqrt{A_n}
+ C_7 \sum_{\{n \in \mathbb{Z} : \ B_n > c\}} B_n  , \eqno{\eqref{GrigTalkEst0}}
$$
$$
N_- (\mathcal{E}_V) \le 1 + 4 \sum_{A_n > 1/4}  \sqrt{A_n}
+ C_9 \sum_{\mathcal{D}_n > c} \mathcal{D}_n  ,
\eqno{\eqref{LaptNetrSolEst}}
$$
$$
N_- (\mathcal{E}_V) \le 1 + C_{10} \left(\left\|(A_n)_{n \in \mathbb{Z}}\right\|_{1, \infty} +
\|V\|_{\mathcal{L}_1\left(\mathbb{R}_+, L_{\mathcal{B}}(\mathbb{S})\right)}\right)  , 
\eqno{\eqref{LaptNetrSolEst2}}
$$
$$
N_- (\mathcal{E}_V) \le 1 + C_{11} \left(\left\|(A_n)_{n \in \mathbb{Z}}\right\|_{1, \infty} +
\int_{\mathbb{R}_+} \left(\int_{\mathbb{S}}|V(r, \vartheta)|^p d\vartheta\right)^{1/p} r dr \right) , 
\eqno{\eqref{LaptNetrSolEst3}}
$$
$$
N_- (\mathcal{E}_V) \le 1 + C_{12} \left(\left\|(A_n)_{n \in \mathbb{Z}}\right\|_{1, \infty} +
\int_{\mathbb{R}_+} 
\left(\int_{\mathbb{S}}|V_{\mathcal{N}}(r, \vartheta)|^p d\vartheta\right)^{1/p} r dr\right) , 
\eqno{\eqref{LaptNetrSolEst4}}
$$
$$
N_- (\mathcal{E}_V) \le 1 + C_{13} \left(\left\|(A_n)_{n \in \mathbb{Z}}\right\|_{1, \infty} +
\|V_{\mathcal{N}}\|_{\mathcal{L}_1\left(\mathbb{R}_+, L_{\mathcal{B}}(\mathbb{S})\right)}\right)  ,
\eqno{\eqref{LaptNetrSolEst5}}
$$
$$
N_- (\mathcal{E}_V) \le 1 + 4 \sum_{A_n > 1/4}  \sqrt{A_n}
+ C_{23} \sum_{\mathcal{G}_n > c} \mathcal{G}_n\,  ,
\eqno{\eqref{LaptevEst}}
$$
where
\begin{eqnarray*}
&& A_0 := \int_{U_0} V(x)\, dx , \ \
A_n := \int_{U_n} V(x) |\ln|x||\, dx , \ n \not= 0 ,  \\ 
&& \mathcal{B}_n := 
\|V\|^{\rm (av)}_{\mathcal{B}, \Omega_n} , \ \ \
B_n := \left(\int_{\Omega_n} V^p(x) |x|^{2(p -1)}\, dx \right)^{1/p} , \\
&& \mathbf{A}_n = A_n, \ \mathbf{B}_n = \mathcal{B}_n ,   \ \ n \in \mathbb{N}, \\
&& \mathbf{A}_0 := \int_{\mathbf{\Omega}_0} V(x) |\ln|x||\, dx , \ \ \ \mathbf{B}_0 = 
\|V\|^{\rm (av)}_{\mathcal{B}, \mathbf{\Omega}_0} ,
\ \ \  \mathbf{\Omega}_0 =  \{x \in \mathbb{R}^2 : \ |x| \le e\} , \\
&&\mathcal{D}_n :=  \|V\|_{\mathcal{L}_1\left(\mathcal{I}_n, L_{\mathcal{B}}(\mathbb{S})\right)} , \ \ \ 
\mathcal{G}_n := e^n 
\int_{\mathbb{S}} \|V(\cdot, \vartheta)\|^{\rm (av)}_{\mathcal{B}, \mathcal{I}_n}\, d\vartheta\, , \ \ \
\mathcal{I}_n := (e^n, e^{n + 1}) , \ \ \ n \in \mathbb{Z} , 
\end{eqnarray*}
$1 < p < \infty$, and $U_n$, $\Omega_n$ are defined in \eqref{ringsR}, \eqref{ringsN}. 
Estimate \eqref{KMWest} was conjectured by N.N. Khuri, A. Martin and T.T. Wu in \cite{KMW}, 
\eqref{SolMainEst} was proved by M. Solomyak in \cite{Sol}, \eqref{GrigTalkEst0} was obtained by 
A. Grigor'yan and N. Nadirashvili in \cite{Grig},  \eqref{LaptNetrSolEst4} was proved by
A. Laptev and M. Solomyak in \cite{LS}, and \eqref{LaptevEst} was proposed by A. Laptev
(see Remark \ref{Laptev}).}

$$
\begin{array}{ccccccl}
&& \eqref{LaptNetrSolEst}   & \Longrightarrow &  \framebox{\eqref{LaptNetrSolEst2}  
$\Longleftrightarrow$  
\eqref{LaptNetrSolEst5}} & \Longrightarrow &  \framebox{\eqref{LaptNetrSolEst3}  
$\Longleftrightarrow$   
\eqref{LaptNetrSolEst4}}\\ \\
&& \big\Downarrow    &   & \big\Downarrow \\ \\
\eqref{LaptevEst} & \Longrightarrow &
\eqref{GrigTalkEst}   & \Longrightarrow &  \eqref{SolMainEst} & \Longrightarrow & 
\framebox{\eqref{clCLRest}  $\Longleftrightarrow$  \eqref{myKMWest}  $\Longleftrightarrow$
\eqref{KMWest}}  \\ \\
&& \big\Downarrow \\ \\
&& \eqref{GrigTalkEst0}
\end{array}
$$

Our next task is to show that no other implication holds between the estimates in the above
diagram. Suppose $V(r, \vartheta) = V_1(r) V_2(\vartheta)$. If $V_2 \equiv 1$, \ 
$V_1(r) = \frac{\alpha}{r^2 (1 + \ln^2 r)}$ and $\alpha > 0$ is sufficiently small, then
the right-hand side of \eqref{GrigTalkEst0} equals 1 while 
$\left\|(A_n)_{n \in \mathbb{N}}\right\|_{1, \infty} = +\infty$ 
(see \cite{Grig}). Hence \eqref{LaptNetrSolEst2} does not imply \eqref{GrigTalkEst0}.

Suppose now $\mbox{\rm supp}\, V_1 \subseteq [1, e]$. If $V_2 \equiv 1$ and
$V_1 \in L_1([1, e])\setminus L_{\mathcal{B}} ([1, e])$, e.g. 
$V_1(r) = \frac{1}{(r - 1)(1 + \ln^2(r -1))}$, then the right-hand side of \eqref{LaptevEst} is
infinite, while the right-hand side of \eqref{LaptNetrSolEst4} is finite. Hence
\eqref{LaptevEst} does not imply \eqref{LaptNetrSolEst3}, \eqref{LaptNetrSolEst4}.
If $V_2 \equiv 1$ and
$V_1 \in L_{\mathcal{B}}([1, e])\setminus\cup_{p > 1} L_p([1, e])$, e.g. 
$V_1(r) = \frac{1}{(r - 1) (1 + |\ln(r -1)|^3)}$, then the right-hand side of \eqref{GrigTalkEst0} is
infinite, while the right-hand side of \eqref{clCLRest} is finite. Hence
\eqref{GrigTalkEst0} does not imply \eqref{clCLRest}. 
If $V_1 \equiv 1$ on 
$[1, e]$ and $V_2 \in L_1(\mathbb{S})\setminus L_{\mathcal{B}}(\mathbb{S})$, 
e.g. $V_2(\vartheta) = \frac{1}{|\vartheta| (1 + \ln^2|\vartheta|)}$, then 
then the right-hand side of \eqref{LaptNetrSolEst}  is
infinite, while the right-hand side of \eqref{LaptevEst} is finite. Hence
\eqref{LaptNetrSolEst}  does not imply \eqref{LaptevEst}.
Finally, if $V_1 \equiv 1$ on 
$[1, e]$ and $V_2 \in L_{\mathcal{B}}(\mathbb{S})\setminus\cup_{p > 1} L_p(\mathbb{S})$, 
e.g. $V_2(\vartheta) = \frac{1}{|\vartheta| (1 + |\ln|\vartheta||^3)}$, then the right-hand side of \eqref{LaptNetrSolEst3}  is
infinite, while the right-hand side of \eqref{clCLRest} is finite. Hence
\eqref{LaptNetrSolEst3}  does not imply \eqref{clCLRest}.

\section{Concluding remarks}\label{Concl}

Using estimate \eqref{LaptNetrSolEst5}, one can prove that Theorem 1.1 and Proposition 1.2 in 
\cite{LS} remain true if one substitutes the condition 
$V_{\mathcal{N}} \in \mathcal{L}_1\left(\mathbb{R}_+, L_p(\mathbb{S})\right)$, $p > 1$ with
\begin{equation}\label{nradLlogL}
V_{\mathcal{N}} \in \mathcal{L}_1\left(\mathbb{R}_+, L_{\mathcal{B}}(\mathbb{S})\right) .
\end{equation}
In particular, if \eqref{nradLlogL} is satisfied, then the Weyl-type asymptotic formula 
$$
\lim_{\alpha \to +\infty} \alpha^{-1} N_- (\mathcal{E}_{\alpha V}) = \frac1{4\pi} \int_{\mathbb{R}^2}
V(x)\, dx
$$
holds if and only if 
\begin{equation}\label{l10}
\lim_{s \to 0+} s\, \mbox{card}\, \{n \in \mathbb{Z} : \ A_n > s\} = 0
\end{equation}
(see \eqref{AnBnR}). One can also prove that, if \eqref{nradLlogL} is satisfied, then
$N_- (\mathcal{E}_{\alpha V}) = O(\alpha)$ as $\alpha \to +\infty$ if and only if
$\left\|(A_n)_{n \in \mathbb{Z}}\right\|_{1, \infty} < \infty$ (cf. \cite[Theorem 1.1]{LS0}; see also
Remark \ref{alphaq} below).
The last result and \eqref{LaptNetrSolEst} imply that if
$$
\sum_{\{n \in \mathbb{Z} : \ \alpha A_n > c\}}  \sqrt{\alpha A_n} = O(\alpha) \ \mbox{ as } \
\alpha \to +\infty ,
$$
then $\left\|(A_n)_{n \in \mathbb{Z}}\right\|_{1, \infty} < \infty$. 
This implication is a special case ($q = 1$, $\sigma = 1/2$) of the following result on sequences of numbers $a_n$,
$n \in \mathbb{Z}$: if $q > \sigma > 0$, then
\begin{eqnarray}
&& \sum_{\{n \in \mathbb{Z} : \ \alpha |a_n| > c\}}  \left(\alpha |a_n|\right)^\sigma = O(\alpha^q) \ 
\mbox{ as } \ \alpha \to +\infty \label{alph}  \\
&& \Longleftrightarrow \ \mbox{card}\, \{n \in \mathbb{Z} : \  |a_n| > s\} = O(s^{-q}) \ 
\mbox{ as } \ s \to 0+ . \label{s}  
\end{eqnarray}
Indeed, if \eqref{s} holds, then one gets similarly to \eqref{sqrtweak}
\begin{eqnarray*}
&& \sum_{\alpha |a_n| > c} (\alpha |a_n|)^\sigma = 
\alpha^\sigma \sum_{|a_n| > c/\alpha} \int_0^{|a_n|} \sigma s^{\sigma - 1} ds \\
&& = \sigma \alpha^\sigma \int_0^\infty s^{\sigma - 1}\, \mbox{card}\, \{n \in \mathbb{Z} : \ 
|a_n| > s \ \& \ |a_n| > c/\alpha\}\, ds  \\
&& = \sigma \alpha^\sigma \int_0^{c/\alpha} 
s^{\sigma - 1}\, \mbox{card}\, \{n \in \mathbb{Z} : \ |a_n| > c/\alpha\}\, ds  \\
&& + \sigma \alpha^\sigma \int_{c/\alpha}^\infty 
s^{\sigma - 1}\, \mbox{card}\, \{n \in \mathbb{Z} : \ |a_n| > s\}\, ds \\
&& =  \alpha^\sigma\, O(\alpha^q) \int_0^{c/\alpha} s^{\sigma - 1}\, ds + \alpha^\sigma\, 
O\left(\int_{c/\alpha}^\infty s^{\sigma -1 - q}\, ds\right) \\
&& = O(\alpha^q) \ \mbox{ as } \ \alpha \to +\infty , 
\end{eqnarray*}
where the last two integrals exist due to the condition $q > \sigma > 0$.
Suppose now \eqref{alph} holds. Then using the notation $\alpha = c/s$ one gets
\begin{eqnarray*}
&& \mbox{card}\, \{n \in \mathbb{Z} : \ |a_n| > s\}  = \mbox{card}\, \{n \in \mathbb{Z} : \ |a_n| > c/\alpha\} \\
&& \le 
\frac{1}{c^\sigma}\, \sum_{\alpha |a_n| > c} (\alpha |a_n|)^\sigma = O(\alpha^q) = O(s^{-q}) \ 
\mbox{ as } \ s \to 0+ .
\end{eqnarray*}

Let us return to the discussion of the necessity of  
\eqref{l10} and $\left\|(A_n)_{n \in \mathbb{Z}}\right\|_{1, \infty} < \infty$.
Neither of these conditions is necessary for $N_- (\mathcal{E}_V)$ to be
finite. Indeed, let $V \ge 0$ be a radial potential such that $(A_n)_{n \in \mathbb{Z}} \in
\ell_\infty(\mathbb{Z})\setminus \ell_{1, \infty}(\mathbb{Z})$. Then it follows from Theorem
\ref{LaptNetrSol} that $N_- (\mathcal{E}_{\beta V}) = 1$ for sufficiently small $\beta > 0$,
although the above conditions are not satisfied for $\beta V$.
It turns out that the condition $(A_n)_{n \in \mathbb{Z}} \in \ell_\infty(\mathbb{Z})$ on the other hand
is necessary for $N_- (\mathcal{E}_V)$ to be finite (Theorem \ref{logneces}), while 
$\left\|(A_n)_{n \in \mathbb{Z}}\right\|_{1, \infty} < \infty$
is necessary for $N_- (\alpha V) = O\left(\alpha\right)$  as  $\alpha \to +\infty$ to hold
(Theorem \ref{lognecesnew}). Note that neither of these results assumes that \eqref{nradLlogL}
is satisfied.

\begin{theorem}\label{logneces}
Let $V \ge 0$. Then
\begin{equation}\label{10pi}
N_- (\mathcal{E}_V) \ge \frac13\, \mbox{\rm card}\, \{n \in \mathbb{Z} : \ A_n \ge 10\pi\} .
\end{equation} 
\end{theorem}
\begin{proof}
Let $m \in \mathbb{Z}$ be such that $A_m \ge 10\pi$ and let $r_0 < r_1 < r_2 < r_3$ be the
radii of the boundary circles of the annuli $U_{m - 1}$, $U_m$, $U_{m + 1}$. Consider the
function
$$
w_m(x) := \left\{\begin{array}{cl}
  0 ,   & \  |x | \le r_0 \ \mbox{ or } \ |x| \ge r_3 , \\ \\
 1 -  \frac{\ln (r_1/|x|)}{\ln(r_1/r_0)} ,  & \ r_0 <  |x| < r_1  , \\ \\
 1 , & \ r_1 \le |x| \le r_2 , \\ \\
 \frac{\ln (r_3/|x|)}{\ln(r_3/r_2)} ,  & \ r_2 <  |x| < r_3 .  
\end{array}\right.
$$
It is easy to see that
\begin{eqnarray*}
&& \int_{\mathbb{R}^2} |\nabla w_m(x)|^2\, dx = 2\pi \int_{r_0}^{r_1} \frac1{(\ln(r_1/r_0))^2}\,
\frac1r\, dr + 2\pi \int_{r_2}^{r_3} \frac1{(\ln(r_3/r_2))^2}\, \frac1r\, dr \\
&& = 2\pi \left(\frac1{\ln(r_1/r_0)} + \frac1{\ln(r_3/r_2)}\right) \\
&&  = \left\{\begin{array}{cl}
  2\pi \left(2^{-|m|} + 2^{2 - |m|}\right) ,   & \  m \le -2 \ \mbox{ or } \ m \ge 2 , \\ \\
 2\pi ,  & \ m = \pm 1 , \\ \\
 4\pi , & \ m = 0
 \end{array}\right. \\ 
 && \le 2\pi \left(2^{-|m|} + 2^{2 - |m|}\right)  = 10\pi\, 2^{-|m|} ,
\end{eqnarray*}
and the inequality is strict if $m = 0, \pm 1$.
Since $w_m(x) = 1$ for $x \in U_m$, one gets
\begin{eqnarray*}
&& \int_{\mathbb{R}^2} V(x) |w_m(x)|^2\, dx \ge \int_{U_m} V(x) dx \\
&& \ge \left\{\begin{array}{cl}
  \frac1{2^{|m|}}
\int_{U_m} V(x) |\ln|x||\, dx,   & \  m \not= 0 , \\ \\
 A_0 ,  & \ m = 0
 \end{array}\right. \\
&& =  \frac{A_m}{2^{|m|}}  
 \ge 10\pi\, 2^{-|m|} \ge \int_{\mathbb{R}^2} |\nabla w_m(x)|^2\, dx ,
\end{eqnarray*}
and the second inequality is strict if $m \not= 0$.
Hence $\mathcal{E}_V[w_m] < 0$ if $A_m \ge 10\pi$, and \eqref{10pi} follows from the fact
that $w_m$ and $w_k$ have disjoint supports if $|m - k| \ge 3$. Indeed, if the set 
$\Sigma := \{n \in \mathbb{Z} : \ A_n \ge 10\pi\}$ is infinite, then it contains infinitely many elements
that lie at a distance at least 3 from each other, and both sides of \eqref{10pi} are infinite. If
the set $\Sigma$ is  nonempty and finite, take
$$
m_1 := \min\Sigma, \ \ \ m_{j +1} = \min\{n \in \Sigma : \ n \ge m_j + 3\}, \ j = 1, \dots
$$ 
and continue the process until the set on the right-hand side is empty. This produces 
at least $\frac13\,\mbox{\rm card}\, \Sigma$ numbers and concludes the proof. \end{proof}

\begin{theorem}\label{lognecesnew}
Let $V \ge 0$. If $N_- (\alpha V) = O\left(\alpha\right)$  as  $\alpha \to +\infty$, \ then
 $\left\|(A_n)_{n \in \mathbb{Z}}\right\|_{1, \infty} < \infty$.
\end{theorem}
\begin{proof}
This follows from the previous theorem. 
Indeed, if $N_- (\alpha V) \le C \alpha$ for large $\alpha$, then \eqref{10pi} implies
\begin{eqnarray*}
&& \frac13\, \mbox{\rm card}\, \{n \in \mathbb{Z} : \ \alpha A_n \ge 10\pi\} \le  C \alpha  \ \ \ 
\Longleftrightarrow \\ 
&& \mbox{\rm card}\, \{n \in \mathbb{Z} : \ A_n \ge 10\pi/\alpha\} \le 3 C \alpha .
\end{eqnarray*}
Hence we have (with $s = 10\pi/\alpha$)
$$
\mbox{\rm card}\, \{n \in \mathbb{Z} : \ A_n \ge s\} \le C_1 s^{-1}  \ \ \ \mbox{ for small } \ \ \ s > 0,
$$
where $C_1 =  30\pi C$.
\end{proof}

\begin{remark}\label{alphaq}
{\rm Suppose \eqref{nradLlogL} holds. Then it follows from the above that 
$N_- (\mathcal{E}_{\alpha V}) = O(\alpha^q)$, $q \ge 1$ as $\alpha \to +\infty$ if and only if
\begin{equation}\label{alphaqest}
\mbox{\rm card}\, \{n \in \mathbb{Z} : \ A_n > s\} = O(s^{-q}) \ \mbox{ as } \ s \to 0+ . 
\end{equation}
Indeed, if this equality holds, then it follows from \eqref{LaptNetrSolEst}, \eqref{NVequiv} and
from the equivalence of \eqref{alph} and \eqref{s} that $N_- (\mathcal{E}_{\alpha V}) = O(\alpha^q)$.
Conversely, if the latter holds, then \eqref{10pi} with $\alpha = 10\pi/s$ implies
\begin{eqnarray*}
\mbox{\rm card}\, \{n \in \mathbb{Z} : \ A_n > s\} = 
\mbox{\rm card}\, \{n \in \mathbb{Z} : \ \alpha A_n > 10\pi\} \\
\le 3 N_- (\mathcal{E}_{\alpha V}) =
O(\alpha^q) = O(s^{-q}) \ \mbox{ as } \ s \to 0+ . 
\end{eqnarray*}
}
\end{remark}

Note that none of the estimates in the paper is sharp in the sense that 
$N_- (\mathcal{E}_V)$ has to be infinite if the right-hand side is infinite. 
Indeed, the examples at the end of Section \ref{Compar} show that the right-hand side of
\eqref{LaptNetrSolEst}
may be infinite while $N_- (\mathcal{E}_V)$ is finite due to \eqref{LaptevEst}, and the other
way around: the right-hand side of \eqref{LaptevEst} may be infinite while 
$N_- (\mathcal{E}_V)$ is finite due to \eqref{LaptNetrSolEst}. The same argument shows that
the finiteness of the right hand sides in these estimates is not necessary even for 
$N_- (\alpha V) = O\left(\alpha\right)$  as  $\alpha \to +\infty$ to hold. 
On the other hand,  
no estimate of the type
\begin{equation}\label{imposs}
N_- (\mathcal{E}_V) \le \mbox{const} + \int_{\mathbb{R}^2} V(x)  W(x)\, dx +
\mbox{const}\,  \|V\|_{\Psi, \mathbb{R}^2} 
\end{equation}
can hold with an Orlicz norm $\|\cdot\|_\Psi$ weaker than $\|\cdot\|_\mathcal{B}$, 
provided the weight function $W$ is bounded in a neighborhood of at least one point
(cf. \cite[Section 4]{Sol}). Indeed, let $\Omega$ be a bounded open set where $W$ is bounded 
and suppose $\Psi(s)/\mathcal{B}(s) \to 0$ as $s \to \infty$. Let
$\Phi$ be the complementary function to $\Psi$. Then it follows from \cite{HMT} and
\cite[Lemma 13.1]{KR} that there exist $w_j \in \mathring{W}^1_2(\Omega)$, 
$j \in \mathbb{N}$ such that $\|w\|_{W^1_2(\Omega)} = 1$ and $\|w_j^2\|_{\Phi, \Omega} \to
\infty$ as $j \to \infty$. The Banach-Steinhaus theorem and \cite[Theorem 14.2]{KR} imply
the existence of $v \in L_\Psi(\Omega)$ such that the sequence
$$
\int_\Omega v(x) w^2_j(x)\, dx ,  \ \ \ j \in \mathbb{N}
$$
is unbounded. Define $V \in L_\Psi\left(\mathbb{R}^2\right)$ by  
$V(x) = |v(x)|$ for $x \in \Omega$ and $V(x) = 0$ for $x \not\in \Omega$.
Then
$$
\mathcal{Q}(V) := \sup_{\|w\|_{W^1_2\left(\mathbb{R}^2\right)} = 1} 
\int_{\mathbb{R}^2} V(x) |w(x)|^2\, dx = \infty .
$$
Since the quadratic form $\int_{\mathbb{R}^2} V(x) |w(x)|^2 dx \ge 0$ is closable in
$W^1_2\left(\mathbb{R}^2\right)$, it corresponds to a nonnegative self-adjoint operator.
Since $\mathcal{Q}(V) = \infty$, the operator is unbounded, and it follows from the spectral theorem
that there exists an infinite-dimensional subspace $\mathcal{F} \subset 
W^1_2\left(\mathbb{R}^2\right)$ such that 
$$
\int_{\mathbb{R}^2} V(x) |w(x)|^2 dx > \|w\|^2_{W^1_2\left(\mathbb{R}^2\right)} , \ \ \ 
\forall w \in \mathcal{F}\setminus\{0\} .
$$
Hence $N_- (\mathcal{E}_V) = \infty$, but the right-hand side of \eqref{imposs} is finite.
Below is a more constructive proof of the same result. 

\begin{theorem}\label{LlogLneces}
Let $W \ge 0$ be bounded in a neighborhood of at least one point and let $\Psi$ be
an $N$-function such that
$$
\lim_{s \to \infty} \frac{\Psi(s)}{\mathcal{B}(s)} = 0.
$$
Then there exists a compactly supported $V \ge 0$ such that
$$
\int_{\mathbb{R}^2} V(x)  W(x)\, dx + \|V\|_{\Psi, \mathbb{R}^2} < \infty
$$
and $N_- (\mathcal{E}_V) = \infty$.
\end{theorem}
\begin{proof}
Shifting the independent variable if necessary, we can assume that $W$ is bounded 
in a neighborhood of $0$. Let $r_0 > 0$ be such that $W$ is bounded in the open ball
$B(0, r_0)$.

Let 
$$
\gamma(s) := \sup_{t \ge s}\, \frac{\Psi(t)}{\mathcal{B}(t)}\, .
$$ 
Then $\gamma$ is a non-increasing function, $\gamma(s) \to 0$ as $s \to \infty$, and
$\Psi(s) \le \gamma(s) \mathcal{B}(s)$.
Since $\Psi$ is an $N$-function, $\Psi(s)/s \to \infty$ as $s \to \infty$
(see \cite[(1.16)]{KR}). Hence there exists $s_0 \ge e$ such that $\Psi(s) \ge s$ and
$\gamma(s) \le 1$ for $s \ge s_0$. 

Choose $r_k \in (0, 1/s_0)$, $k \in \mathbb{N}$ in such a way that $r_k < \frac13\, r_{k - 1}$ and
$$
\sum_{k = 1}^\infty \gamma\left(\frac1{r_k}\right) < \infty .
$$
It is easy to see that the open disks $B(2r_k, r_k)$, $k \in \mathbb{N}$ lie in $B(0, r_0)$
and are pairwise disjoint.  

Let
\begin{eqnarray*}
& t_k := \frac{3}{\ln\frac1{r_k}}\, r_k^{-4}\, ,  & \\ \\
& V(x) := \left\{\begin{array}{ll}
  t_k ,   &  x \in B\left(2r_k, r_k^2\right) , \  k \in \mathbb{N}, \\
    0 ,  &   \mbox{otherwise.} 
\end{array}\right. &
\end{eqnarray*}
Then
$$
t_k = \frac{3}{\ln\frac1{r_k}}\, r_k^{-4} > \frac{1}{r_k^3 \ln\frac1{r_k}}\, r_k^{-1} > \frac1{r_k}
$$
and
\begin{eqnarray*}
&& \int_{\mathbb{R}^2} \Psi(V(x))\, dx = \sum_{k = 1}^\infty \pi r_k^4\, \Psi(t_k) \le
 \pi \sum_{k = 1}^\infty r_k^4\, \gamma(t_k) \mathcal{B}(t_k) \\
&&  < \pi \sum_{k = 1}^\infty r_k^4\, \gamma(t_k) (1 + t_k) \ln(1 + t_k) <
4\pi \sum_{k = 1}^\infty r_k^4\, \gamma(t_k)\, t_k \ln t_k \\
&& = 4\pi \sum_{k = 1}^\infty \gamma(t_k)\,
\frac{3 r_k^4}{r_k^4\, \ln\frac1{r_k}}\,
\ln\frac{3}{r_k^4\, \ln\frac1{r_k}}< 4\pi \sum_{k = 1}^\infty \gamma(t_k)\,
\frac{3}{\ln\frac1{r_k}}\,
\ln\frac{3}{r_k^4} \\
&& < 72\pi \sum_{k = 1}^\infty \gamma(t_k) \le
72\pi \sum_{k = 1}^\infty \gamma\left(\frac1{r_k}\right) < \infty .
\end{eqnarray*}
Hence $\|V\|_{\Psi, \mathbb{R}^2} < \infty$. Since $t_k > 1/r_k \ge s_0$, we have 
$t_k \le \Psi(t_k)$ and 
$$
\int_{\mathbb{R}^2} V(x)\, dx \le \int_{\mathbb{R}^2} \Psi(V(x))\, dx < \infty .
$$
Taking into account that the support of $V$ lies in $B(0, r_0)$ and that $W$ is bounded in
$B(0, r_0)$, we get
$$
\int_{\mathbb{R}^2} V(x)W(x)\, dx < +\infty .
$$

Let
$$
w_k(x) := \left\{\begin{array}{cl}
  1 ,   & \  |x - 2r_k| \le r_k^2 , \\ \\
    \frac{\ln (r_k/|x - 2r_k|)}{\ln(1/r_k)} ,  & \ r_k^2 <  |x - 2r_k| \le r_k  , \\ \\
 0 , & \  |x - 2r_k| > r_k  
\end{array}\right.
$$
(cf. \cite{Grig}). Then
$$
\int_{\mathbb{R}^2} |\nabla w_k(x)|^2\, dx = \frac{2\pi}{\ln(1/r_k)}
$$
and
$$
\int_{\mathbb{R}^2} V(x) |w_k(x)|^2\, dx \ge \int_{B\left(2r_k, r_k^2\right)} V(x)\, dx =
\pi r_k^4\, t_k = 
\frac{3\pi}{\ln(1/r_k)}\, .
$$
Hence
$$
\mathcal{E}_V[w_k] < 0 , \ \ \ \forall k \in \mathbb{N}
$$
and $N_- (\mathcal{E}_V) = \infty$.
\end{proof}

It is probably difficult to obtain an estimate for $N_- (\mathcal{E}_V)$ that is sharp in the above
sense, i.e. is such that $N_- (\mathcal{E}_V)$ is infinite if the right-hand side is infinite. Indeed, there
are potentials $V \ge 0$ such that $N_- (\mathcal{E}_{\alpha V}) < \infty$ for $\alpha < 1$ and
$N_- (\mathcal{E}_{\alpha V}) = \infty$ for $\alpha > 1$. For such potentials, $N_- (\mathcal{E}_V)$
may be finite or infinite and the following theorem shows that, in the latter case, 
 $N_- (\mathcal{E}_{\alpha V})$ may grow arbitrarily fast or arbitrarily slow as $\alpha \to 1 - 0$.

 \begin{theorem}\label{alpha1}
i) For any $N \in \mathbb{N}$, there exists a $V \in L_1(\mathbb{R}^2)$, $V \ge 0$, such that
$N_- (\mathcal{E}_V) = N$ and
$N_- (\mathcal{E}_{\alpha V}) = \infty$ for $\alpha > 1$. 

ii) For any sequence $(\alpha_k)_{k \in \mathbb{N}}$ increasing to 1 and any $N_k \in \mathbb{N}$,
there exists a $V \in L_1(\mathbb{R}^2)$, $V \ge 0$, such that 
$N_k \le N_- (\mathcal{E}_{\alpha_{k + 1} V}) - N_- (\mathcal{E}_{\alpha_k V}) < \infty$, 
$k \in \mathbb{N}$.

iii) For any sequence $(\alpha_k)_{k \in \mathbb{N}}$ increasing to 1 and satisfying the condition
\begin{equation}\label{075}
\sum_{k = 1}^\infty \frac{1 - \alpha_{k + 1}}{\alpha_{k + 1} - \alpha_k} < \infty\, ,
\end{equation}
there exist a $V \in L_1(\mathbb{R}^2)$, $V \ge 0$ and a $k_0 \in \mathbb{N}$ such that 
$N_- (\mathcal{E}_{\alpha_{k + 1} V}) - N_- (\mathcal{E}_{\alpha_k V}) = 1$ for all $k \ge k_0$.
\end{theorem}
\begin{proof}
See Appendix B.
\end{proof}

\section{Appendix A: \ Sharp 1-dimensional Sobolev type inequalities}

Let $0 < a <b$. It follows from the embedding $W^1_2([a, b]) \hookrightarrow 
C([a, b])$ that there exist constants $\alpha, \beta > 0$ such that
$$
\frac{|u(x)|^2}{|x|} \le \alpha \int_a^b |u'(t)|^2\, dt +  \beta \int_a^b \frac{|u(t)|^2}{|t|^2}\, dt , \ \ \ 
\forall u \in W^1_2([a, b]) .
$$
This inequality is used in Section \ref{BSsect} and it is natural to ask what the best values of 
$\alpha, \beta > 0$ are. Since there are two constants involved here, it is convenient to rewrite the inequality in the following
form
\begin{equation}\label{1dHSob}
\frac{|u(x)|^2}{|x|} \le C(\kappa) \left(\int_a^b |u'(t)|^2\, dt +  \kappa \int_a^b \frac{|u(t)|^2}{|t|^2}\, dt\right) , 
\ \ \  \forall u \in W^1_2([a, b]) ,
\end{equation}
and to look for the best value of $C(\kappa)$ for a given $\kappa > 0$.

\begin{lemma}\label{1dHSobLem}
Let
\begin{equation}\label{gammas}
\gamma_1 = \frac{- 1 + \sqrt{1 + 4\kappa}}{2}\, , \ \ \gamma_2 = \frac{- 1 - \sqrt{1 + 4\kappa}}{2}\, .
\end{equation}
For any $\kappa > 0$, \eqref{1dHSob} holds with
$$
C(\kappa; x) := \frac{\gamma_1^2 x^{\sqrt{1 + 4\kappa}} +
\kappa \left(b^{\sqrt{1 + 4\kappa}} + a^{\sqrt{1 + 4\kappa}}\right) + 
\gamma_2^2 (ab)^{\sqrt{1 + 4\kappa}}x^{-\sqrt{1 + 4\kappa}}}{\kappa 
\sqrt{1 + 4\kappa}\, \left(b^{\sqrt{1 + 4\kappa}} - a^{\sqrt{1 + 4\kappa}}\right)}
$$
and becomes an equality for
$$
u(t) = \left\{\begin{array}{ll}
\left(\gamma_2 b^{\sqrt{1 + 4\kappa}} - \gamma_1 x^{\sqrt{1 + 4\kappa}}\right)
\left(\gamma_1 t^{\gamma_1} - \gamma_2 a^{\sqrt{1 + 4\kappa}}t^{\gamma_2}\right) t  ,   
 & \  a \le t < x , \\ \\
\left(\gamma_2 a^{\sqrt{1 + 4\kappa}} - \gamma_1 x^{\sqrt{1 + 4\kappa}}\right)
\left(\gamma_1 t^{\gamma_1} - \gamma_2 b^{\sqrt{1 + 4\kappa}}t^{\gamma_2}\right) t ,   
 & \  x < t \le b.
\end{array}\right.
$$
The maximum of $C(\kappa; x)$ is achieved at $x = a$  and is equal to
\begin{equation}\label{1dHOpt}
C(\kappa) := \frac{1}{2\kappa}\,\left(1 + \sqrt{1 + 4\kappa}\,
\frac{b^{\sqrt{1 + 4\kappa}} + a^{\sqrt{1 + 4\kappa}}}{b^{\sqrt{1 + 4\kappa}} - 
a^{\sqrt{1 + 4\kappa}}}\right)\, .
\end{equation}
\end{lemma}
\begin{proof}
It is easy to see that
$$
u(x) \ln\frac{b}{a} = \int_a^b G_x(t) u'(t)\, dt + \int_a^b \frac{u(t)}t\, dt ,
$$
where
$$
G_x(t) := \left\{\begin{array}{ll} 
\ln\frac{t}{a}\,  ,  & \  a \le t < x , \\ \\
\ln\frac{t}{b}\,  ,  & \  x < t \le b .
\end{array}\right.
$$
Since
$$
\int_a^b\varphi(t) u'(t)\, dt + \int_a^b (t \varphi'(t))\, \frac{u(t)}t\, dt = 0 , \ \ \ \forall \varphi \in 
\mathring{W}^1_2\big((a, b)\big) ,
$$
we also have
$$
u(x) \ln\frac{b}{a} = \int_a^b (G_x(t) - \varphi(t)) u'(t)\, dt + \int_a^b (1 - t\varphi'(t))\, \frac{u(t)}t\, dt ,
$$
and the Cauchy--Schwarz inequality implies
$$
\frac{|u(x)|^2}{|x|} \le \frac{\mathcal{J}(\varphi)}{|x| \ln^2\frac{b}{a}}\,
 \left(\int_a^b |u'(t)|^2\, dt +  \kappa \int_a^b \frac{|u(t)|^2}{|t|^2}\, dt\right) ,
$$
where
$$
\mathcal{J}(\varphi) := \int_a^b |G_x(t) - \varphi(t)|^2\, dt +  
\frac1\kappa\, \int_a^b |1 - t\varphi'(t)|^2\, dt .
$$
Hence we need to minimise $\mathcal{J}(\varphi)$ on $\mathring{W}^1_2\big((a, b)\big)$.
It is clear we only need to consider real-valued $\varphi$.
The Euler equation for this functional  takes the form
$$
t^2 \varphi'' + 2t \varphi' - \kappa\varphi = 1 - \kappa G_x , \ \ \ \varphi(a) = 0 = \varphi(b) .
$$
It is easy to solve the equation on $(a, x)$ and separately on $(x, b)$: the function $\varphi_0 = G_x$
is a solution on both intervals, and the change of the independent variable $s = \ln t$ reduces the
corresponding homogeneous equation to an ODE with constant coefficients. Choosing the 
constants in the general solutions on $(a, x)$ and $(x, b)$ in such a way that
$\varphi(a) = 0$,  $\varphi(b) = 0$, 
$\varphi(x - 0) = \varphi(x + 0)$ and $\varphi'(x - 0) = \varphi'(x + 0)$, one gets
\begin{eqnarray}\label{Hminimiser}
&& \varphi(t) = G_x(t) + \frac{x^{-\gamma_1} \ln\frac{b}{a}}{(\gamma_1 - \gamma_2)
\left(b^{\sqrt{1 + 4\kappa}} - a^{\sqrt{1 + 4\kappa}}\right)} \Phi_x(t) ,  \\
&& \Phi_x(t) := \left\{\begin{array}{ll}
\left(\gamma_2 b^{\sqrt{1 + 4\kappa}} - \gamma_1 x^{\sqrt{1 + 4\kappa}}\right)
\left(t^{\gamma_1} - a^{\sqrt{1 + 4\kappa}}t^{\gamma_2}\right)   ,   
 & \  a \le t < x , \\ \\
\left(\gamma_2 a^{\sqrt{1 + 4\kappa}} - \gamma_1 x^{\sqrt{1 + 4\kappa}}\right)
\left(t^{\gamma_1} - b^{\sqrt{1 + 4\kappa}}t^{\gamma_2}\right) ,   
 & \  x < t \le b.
\end{array}\right. \nonumber
\end{eqnarray}
Taking into account that $\gamma_1$ and $\gamma_2$ (see \eqref{gammas}) are the roots
of the quadratic equation $\gamma^2 + \gamma - \kappa = 0$, one can easily check that \eqref{Hminimiser} does indeed 
solve the equation on $(a, x)$ and $(x, b)$ and satisfy the above conditions at $t = a, x, b$. 

It follows from the above that \eqref{1dHSob} holds with $\mathcal{J}(\varphi)/(|x| \ln^2\frac{b}{a})$ 
in place of $C(\kappa)$ for any $\varphi \in \mathring{W}^1_2\big((a, b)\big)$, 
in particular for the one given by \eqref{Hminimiser}. 
Let us show that the equality in \eqref{1dHSob} is achieved for the latter.  
Indeed, let 
\begin{eqnarray*}
&& u(t) := \frac1\kappa\, (t - t^2\varphi'(t)) \\
&& = M(\kappa)
\left\{\begin{array}{ll}
\left(\gamma_2 b^{\sqrt{1 + 4\kappa}} - \gamma_1 x^{\sqrt{1 + 4\kappa}}\right)
\left(\gamma_1 t^{\gamma_1} - \gamma_2 a^{\sqrt{1 + 4\kappa}}t^{\gamma_2}\right) t  ,   
 & \  a \le t < x , \\ \\
\left(\gamma_2 a^{\sqrt{1 + 4\kappa}} - \gamma_1 x^{\sqrt{1 + 4\kappa}}\right)
\left(\gamma_1 t^{\gamma_1} - \gamma_2 b^{\sqrt{1 + 4\kappa}}t^{\gamma_2}\right) t ,   
 & \  x < t \le b, 
 \end{array}\right. \\
&&  M(\kappa) := -\frac{x^{-\gamma_1} \ln\frac{b}{a}}{\kappa 
(\gamma_1 - \gamma_2) \left(b^{\sqrt{1 + 4\kappa}} - a^{\sqrt{1 + 4\kappa}}\right)} .
\end{eqnarray*}
Then $u \in W^1_2([a, b])$, $u' = (1 - t^2\varphi''(t) - 2t\varphi'(t))/\kappa = G_x - \varphi$, 
and the Cauchy--Schwarz inequality used above is in fact an equality.

It is left to evaluate $\mathcal{J}(\varphi)/(|x| \ln^2\frac{b}{a})$ for \eqref{Hminimiser}. 
Using the equalities
\begin{eqnarray*}
&& 2\gamma_1 + 1 = \sqrt{1 + 4\kappa}\, , \ \ 2\gamma_2 + 1 = - \sqrt{1 + 4\kappa}\, , \\ 
&& 1 + \frac{\gamma_1^2}{\kappa} = \frac{\sqrt{1 + 4\kappa}}{\kappa}\, \gamma_1 , \ \ 
1 + \frac{\gamma_2^2}{\kappa} = -\frac{\sqrt{1 + 4\kappa}}{\kappa}\, \gamma_2 , \\
&& 2\kappa \gamma_1 - \gamma_1^2 = \sqrt{1 + 4\kappa}\, \gamma_1^2 , \ \ 
2\kappa \gamma_2 - \gamma_2^2 = - \sqrt{1 + 4\kappa}\, \gamma_2^2,
\end{eqnarray*}
one gets after a straightforward but not a particularly pleasant calculation
$$
\frac{\mathcal{J}(\varphi)}{|x| \ln^2\frac{b}{a}} = 
\frac{\gamma_1^2 x^{\sqrt{1 + 4\kappa}} +
\kappa \left(b^{\sqrt{1 + 4\kappa}} + a^{\sqrt{1 + 4\kappa}}\right) + 
\gamma_2^2 (ab)^{\sqrt{1 + 4\kappa}}x^{-\sqrt{1 + 4\kappa}}}{\kappa 
\sqrt{1 + 4\kappa}\, \left(b^{\sqrt{1 + 4\kappa}} - a^{\sqrt{1 + 4\kappa}}\right)}\, .
$$
Since the function $z \mapsto \gamma_1^2 z + \gamma_2^2 (ab)^{\sqrt{1 + 4\kappa}} z^{-1}$
does not have a local maximum for $z > 0$ ($z = x^{\sqrt{1 + 4\kappa}}$), the above fraction
achieves its maximum on $[a, b]$ at an endpoint. It is easy to see that the maximum is achieved
at $x = a$ and is equal to \eqref{1dHOpt}.
\end{proof}

\begin{remark}\label{HSob01}
{\rm Suppose $u \in W^1_2([0, 1])$ and $u(0) = 0$. Then using \eqref{1dHSob},
\eqref{1dHOpt} with $b = 1$ and $a \to 0+$ one gets 
\begin{eqnarray*}
\frac{|u(x)|^2}{|x|} \le \frac{1}{2\kappa}\,\left(1 + \sqrt{1 + 4\kappa}\right) 
\left(\int_0^1 |u'(t)|^2\, dt +  \kappa \int_0^1 \frac{|u(t)|^2}{|t|^2}\, dt\right) ,   \\  
\forall x \in (0, 1],  
\end{eqnarray*} 
and the right-hand side is finite due to Hardy's inequality. Note that
$$
\frac{1}{2\kappa}\,\left(1 + \sqrt{1 + 4\kappa}\right) < C(\kappa)
$$
for any $b > a > 0$ (see \eqref{1dHOpt}).}
\end{remark}

Starting with the representation
$$
u(x) = \int_0^1 H_x(t) u'(t)\, dt + \int_0^1 u(t)\, dt ,
$$
where
$$
H_x(t) := \left\{\begin{array}{ll} 
t ,  & \  0 \le t < x , \\ 
t - 1  ,  & \  x < t \le 1 ,
\end{array}\right.
$$
one can find, similarly to \eqref{1dHSob}, \eqref{1dHOpt}, the optimal constant $C_0(\kappa)$
in the estimate
$$
|u(x)|^2 \le C_0(\kappa) \left(\int_0^1 |u'(t)|^2\, dt +  \kappa \int_0^1 |u(t)|^2\, dt\right) , \ \ \ 
\forall u \in W^1_2([0, 1]) .
$$
The calculations are easier in this case, and one gets after an affine transformation of 
the independent variable that for any $\kappa > 0$,
\begin{eqnarray}\label{1dSob}
|u(x)|^2 \le C_0(\kappa) \left((b- a)\, \int_a^b |u'(t)|^2\, dt +  
\frac{\kappa}{b - a} \int_a^b |u(t)|^2\, dt\right) ,   \\  
\forall u \in W^1_2([a, b])  \nonumber
\end{eqnarray}
holds with
$$
C_0(\kappa; x) := \frac{\sinh(2\sqrt{\kappa}) + 
\sinh\left(2\sqrt{\kappa}\, \frac{x - a}{b - a}\right) + 
\sinh\left(2\sqrt{\kappa}\, \frac{b - x}{b - a}\right)}{4 \sqrt{\kappa}\, \sinh^2\sqrt{\kappa}}
$$
and becomes an equality for
$$
u(t) = \left\{\begin{array}{ll}
 \cosh\left(\sqrt{\kappa}\, \frac{b - x}{b - a}\right)\, \cosh\left(\sqrt{\kappa}\, \frac{t - a}{b - a}\right) ,   
 & \  a \le t < x , \\ \\
\cosh\left(\sqrt{\kappa}\, \frac{x - a}{b - a}\right)\, 
 \cosh\left(\sqrt{\kappa}\, \frac{b - t}{b - a}\right) ,   
 & \  x < t \le b.
\end{array}\right.
$$
The maximum of $C_0(\kappa; x)$ is achieved at $x = a$ and $x = b$, and is equal to
\begin{equation}\label{1dOpt}
C_0(\kappa) := 
\frac{\coth\sqrt{\kappa}}{\sqrt{\kappa}}\, .
\end{equation}
One can rewrite the inequality
\begin{eqnarray}\label{1dSobab}
|u(x)|^2 \le \frac{\coth\sqrt{\kappa}}{\sqrt{\kappa}}\, \left((b- a)\, \int_a^b |u'(t)|^2\, dt +  
\frac{\kappa}{b - a} \int_a^b |u(t)|^2\, dt\right) , \\ 
\forall x \in [a, b] ,  \ \  \forall u \in W^1_2([a, b])  \nonumber
\end{eqnarray}
in the following more symmetric form (with $\varrho = \sqrt{\kappa}$):
$$
|u(x)|^2 \le \coth\varrho \left(\frac{b - a}{\varrho}\, \int_a^b |u'(t)|^2\, dt +  
\frac{\varrho}{b - a} \int_a^b |u(t)|^2\, dt\right) ,  \ \ \ \forall \varrho > 0 ,
$$
and the equality here is achieved for 
$$
u(t) = \cosh\left(\varrho\, \frac{t - a}{b - a}\right) , \ x = b \ \  \mbox{ and } \ \ 
u(t) = \cosh\left(\varrho\, \frac{t - b}{b - a}\right) , \ x = a .
$$

\section{Appendix B: \ Proof of Theorem \ref{alpha1}}\label{AppB}

\begin{proof}
All potentials $V$ appearing in the proof are radial $V(x) = F(|x|)$ and satisfy the following
conditions: $V(x) = 0$ if $|x| \le a_0$, and 
\begin{equation}\label{Vbetta}
0 \le V(x) \le \frac{\gamma}{|x|^2 \ln^2 |x|}\, , \ \ \  |x| > a_0
\end{equation}
with certain $a_0, \gamma > 0$. We can take $\gamma = 1/3$ in i) and 
$\gamma = 1/(3\alpha_1)$ in ii), iii). According to \eqref{RadialBSest},
$$
N_- (\mathcal{E}_V) = N_- (\mathcal{E}_{\mathcal{R}, V}) + N_- (\mathcal{E}_{\mathcal{N}, V}) 
$$
for radial potentials. It follows from the proof of Theorem \ref{LaptNetrSol} or from Lemma 4.1 in
\cite{LS0} that  one can choose $a_0 > 0$ large enough to get 
$N_- (\mathcal{E}_{\mathcal{N}, V}) = 0$. We assume below that $a_0 > 1$ and 
that the last equality holds.
Then
$$
N_- (\mathcal{E}_V) = N_- (\mathcal{E}_{X, G}) ,
$$
where $G(t) = e^{2t} F(e^t)$
(see \eqref{V1d} and \eqref{Gform}). 
The usual approximation argument shows that $N_- (\mathcal{E}_{X, G}) =
N_- (\mathcal{E}_{\mathcal{H}_1, G})$, where $\mathcal{E}_{\mathcal{H}_1, G}$ denotes
the form \eqref{Gform} with the domain
$$
\mathcal{H}_1 := \left\{u \in W^1_{2, {\rm loc}}(\mathbb{R}) : \ 
\int_{\mathbb{R}} |u'|^2\, dt + \int_{\mathbb{R}} \frac{|u|^2}{1 + |t|^2}\, dt < \infty\right\} .
$$
Note that \eqref{Vbetta} is equivalent to 
$$
0 \le G(t) \le  \frac{\gamma}{|t|^2} , \ \ \ t > a_1 := \ln a_0 > 0 .
$$
We assume throughout the proof that $G(t) = 0$ for $t \le a_1$.

It follows from the above that it is sufficient to prove the theorem with 
$N_- (\mathcal{E}_{\mathcal{H}_1, \alpha G})$ in place of $N_- (\mathcal{E}_{\alpha V})$.

Let $0 < a < b < \infty$. We denote by 
$\mathcal{E}_{\mathcal{H}_1(a, b), G}$, $\mathcal{E}^{a,b}_G$
and $\mathcal{E}_{\mathcal{H}_1(b, \infty), G}$
the forms defined by \eqref{Gform} on the domains
\begin{eqnarray*}
&& \mathcal{H}_1(a, b) := \left\{u \in \mathcal{H}_1 : \ 
u(a) = 0 = u(b)\right\} , \ \ \ \mathring{W}^1_2\big((a, b)\big) , 
\ \ \mbox{ and } \\
&& \mathcal{H}_1(b, \infty) := \left\{u \in \mathcal{H}_1 : \ 
u(t) = 0 , \ t \le b\right\}
\end{eqnarray*}
respectively, and we also use the following notation 
\begin{equation}\label{Wab}
 \mathcal{E}^{a,b}_\beta[u] : = \int_a^b |u'(t)|^2 dt - 
\beta \int_a^b  \frac{|u(t)|^2}{t^2}\, dt ,  \ \ \   
 \mbox{Dom}\, (\mathcal{E}^{a,b}_\beta) =
\mathring{W}^1_2\big((a, b)\big). 
\end{equation}

i) Let $G(t) = 1/(3t^2)$ for $a_1 < t \le a_2$ and $G(t) = 1/(4t^2)$ for $t > a_2$, where
$a_2$ is chosen in such a way that 
$$
\ln\frac{a_2}{a_1} = \sqrt{3}\, \left(2N - \frac32\right)\pi , \ \ \mbox{ i.e. } \ \ 
\frac13 =
\frac14 + \left(\frac{(N - 3/4)\pi}{\ln\frac{a_2}{a_1}}\right)^2 .
$$
Let $\alpha \in (3/4, 1)$.
We will need two auxiliary functions $u_1$ and $u_2$ that solve the equation $u'' + \alpha G u = 0$
 on $(-\infty, a_1)$,  $(a_1, a_2)$, and $(a_2, +\infty)$. On each of these intervals, $G$ has the form
$\mbox{const}/t^2$, and the change of the independent variable $s = \ln t$ reduces the
equation to an ODE with constant coefficients. Let
\begin{eqnarray*}
&& \mu_1 := \sqrt{\frac{\alpha}{3} - \frac14} \, , \ \ \ 
\mu_2 := \frac{1- \sqrt{1 - \alpha}}{2}\, , \\ \\
&& u_1(t) := \left\{\begin{array}{ll} 
\rho a_1^{1/2} \sin(\mu_1 \ln a_1 + \delta)  ,  & \  t < a_1  , \\ 
\rho t^{1/2} \sin(\mu_1 \ln t + \delta)  ,  & \  a_1 < t < a_2 , \\
t^{\mu_2} , & \ t > a_2 , 
\end{array}\right. \\ \\
&& u_2(t) := \left\{\begin{array}{ll} 
0 ,  & \  t < a_1  , \\ 
t^{1/2} \sin(\mu_1 (\ln t - \ln a_1))  ,  & \  a_1 < t < a_2 , \\
a_2^{1/2 - \mu_2} \sin\left(\mu_1 \ln\frac{a_2}{a_1}\right)\,  t^{\mu_2} , & \ t > a_2 , 
\end{array}\right. 
\end{eqnarray*}
where the constants $\rho , \delta \in \mathbb{R}$ are chosen in such a way that
$u_1(a_2 - 0) = u_1(a_2 + 0)$, $u'_1(a_2 - 0) = u'_1(a_2 + 0)$. It is easy to see that
$u_1 \in \mathcal{H}_1$ and $u_2 \in \mathcal{H}_1(a_1, \infty)$ are indeed solutions
on the above intervals and that $u_1(a_1 - 0) = u_1(a_1 + 0)$, $u_2(a_2 - 0) = u_2(a_2 + 0)$.
Integration by parts gives the following 
\begin{eqnarray*}
&& \mathcal{E}_{\mathcal{H}_1, \alpha G}(u_1, u) = 0 ,  \ \ \forall u \in \mathcal{H}_1(a_1, \infty) , \ 
\mbox{ in particular, }  \ \ \mathcal{E}_{\mathcal{H}_1, \alpha G}(u_1, u_2) = 0 , \\
&& \mathcal{E}_{\mathcal{H}_1, \alpha G}(u_2, u) = 0 , \ \ \ \forall u \in \mathcal{H}_1(a_1, a_2) , \\
&& \mathcal{E}_{\mathcal{H}_1, \alpha G}[u_1] = u_1(a_1) (u'_1(a_1 - 0) - u'_1(a_1 + 0)) \\
&&  \ \ \ \ \   =
-\rho^2 \sin^2(\mu_1 \ln a_1 + \delta) \left(\frac12 + \mu_1 \cot(\mu_1 \ln a_1 + \delta)\right) , \\
&& \mathcal{E}_{\mathcal{H}_1, \alpha G}[u_2] = u_1(a_2) (u'_2(a_2 - 0) - u'_2(a_2 + 0)) \\
&&  \ \ \ \ \   = \sin^2\left(\mu_1 \ln\frac{a_2}{a_1}\right) 
\left(\frac12 - \mu_2 + \mu_1 \cot\left(\mu_1 \ln\frac{a_2}{a_1}\right)\right) .
\end{eqnarray*}
If $\alpha$ is close to 1, then $\mu_1 \ln\frac{a_2}{a_1}$ is close to
$$
\frac{1}{2\sqrt{3}}\, \ln\frac{a_2}{a_1} = \left(N - \frac34\right) \pi ,
$$
and hence $\mathcal{E}_{\mathcal{H}_1, \alpha G}[u_2] > 0$. The condition
$$
\frac{u'_1(a_2 - 0)}{u_1(a_2 - 0)} = \frac{u'_1(a_2 + 0)}{u_1(a_2 + 0)}
$$
is equivalent to $1/2 + \mu_1 \cot(\mu_1 \ln a_2 + \delta) = \mu_2$. Again, if $\alpha$ is close to 1, 
then
$\mu_2 < 1/2$ is close to $1/2$. Hence $\cot(\mu_1 \ln a_2 + \delta)$ is a small negative number,
i.e. $\mu_1 \ln a_2 + \delta = (m + 1/2) \pi + \epsilon$, where $\epsilon$ is a small positive number 
and $m \in \mathbb{Z}$. Consequently, 
$$
\mu_1 \ln a_1 + \delta = (m + 1/2) \pi + \epsilon - \mu_1 \ln\frac{a_2}{a_1}
$$
is close to $(m - N + 1 + 1/4) \pi + \epsilon$, and $\mathcal{E}_{\mathcal{H}_1, \alpha G}[u_1] < 0$.

It is easy to see that any $u \in \mathcal{H}_1$ admits a unique representation
$$
u = d_1 u_1 + d_2 u_2 + u_0 , \ \ \ d_1, d_2 \in \mathbb{C} , \ u_0 \in \mathcal{H}_1(a_1, a_2) ,
$$ 
and it follows from the above that
$$
\mathcal{E}_{\mathcal{H}_1, \alpha G}[u]  = |d_1|^2 \mathcal{E}_{\mathcal{H}_1, \alpha G} [u_1] +
|d_2|^2 \mathcal{E}_{\mathcal{H}_1, \alpha G}[u_2] +
\mathcal{E}_{\mathcal{H}_1, \alpha G}[u_0] .
$$
Since $\mathcal{E}_{\mathcal{H}_1, \alpha G}[u_1] < 0$ and 
$\mathcal{E}_{\mathcal{H}_1, \alpha G}[u_2] > 0$, we have
\begin{eqnarray*}
&& N_- (\mathcal{E}_{\mathcal{H}_1, \alpha G}) = 1 + 
N_- (\mathcal{E}_{\mathcal{H}_1(a_1, a_2), \alpha G}) \\
&& = 1 + N_- \left(\mathcal{E}^{a_1,a_2}_{\alpha/3}\right) + 
N_- \left(\mathcal{E}_{\mathcal{H}_1(a_2, \infty), \alpha G}\right) = 1 + (N - 1) + 0 = N ,
\end{eqnarray*}
where the penultimate equality follows from Lemma \ref{Dirichlet} (see below) and
Hardy's inequality (see, e.g., \cite[Theorem 327]{HLP}). Hence
$$
N_- (\mathcal{E}_{\mathcal{H}_1, G}) = 
\lim_{\alpha \to 1 - 0} N_- (\mathcal{E}_{\mathcal{H}_1, \alpha G}) = N .
$$
On the other hand,
$$
N_- (\mathcal{E}_{\mathcal{H}_1, \alpha G}) \ge 
N_- \left(\mathcal{E}_{\mathcal{H}_1(a_2, \infty), \alpha G}\right) = +\infty , \ \ \ \forall \alpha > 1,
$$
where the last equality holds because $\alpha G > 1/4$ on $(a_2, \infty)$ (see, e.g., 
Lemma \ref{Dirichlet}).  \\

ii) Take any $\alpha_0 \in (0, \alpha_1)$, $N_0 = 0$, and define $a_2, a_3, \dots$
successively by 
$$
\sqrt{\frac{\alpha_k}{2(\alpha_k + \alpha_{k - 1})} - \frac14}\, \ln\frac{a_{k + 1}}{a_k} 
= \left(N_{k - 1} + 2 + \frac12\right)\pi , \ \ \ k \in \mathbb{N} .
$$
Let
$$
G(t) := \frac{1}{2(\alpha_k + \alpha_{k - 1})\, t^2}\, ,  \ \ \ a_k < t < a_{k + 1} , \ k \in \mathbb{N} .
$$
Then 
\begin{eqnarray*}
&& N_- (\mathcal{E}_{\mathcal{H}_1, \alpha_k G}) \le 2 + 
N_- (\mathcal{E}_{\mathcal{H}_1(a_1, a_{k + 1}), \alpha G}) \\
&& = 2 + N_- \left(\mathcal{E}^{a_1, a_{k + 1}}_{\alpha_k G}\right) + 
N_- \left(\mathcal{E}_{\mathcal{H}_1(a_{k + 1}, \infty), \alpha_k G}\right) = 
2 + N_- \left(\mathcal{E}^{a_1, a_{k + 1}}_{\alpha_k G}\right)
\end{eqnarray*}
due to Hardy's inequality, and
\begin{eqnarray*}
N_- (\mathcal{E}_{\mathcal{H}_1, \alpha_{k + 1} G}) \ge 
N_- \left(\mathcal{E}^{a_1, a_{k + 1}}_{\alpha_{k + 1} G}\right) + 
N_- \left(\mathcal{E}^{a_{k + 1}, a_{k + 2}}_{\alpha_{k + 1} G}\right) \\
\ge N_- \left(\mathcal{E}^{a_1, a_{k + 1}}_{\alpha_k G}\right) + N_k + 2
\ge N_- (\mathcal{E}_{\mathcal{H}_1, \alpha_k G}) + N_k
\end{eqnarray*}
due to Lemma \ref{Dirichlet}. \\

iii) Take $k_0 \in \mathbb{N}$ such that
\begin{equation}\label{18}
\sum_{k \ge k_0 - 1} \frac{1 - \alpha_{k + 1}}{\alpha_{k + 1} - \alpha_k} \le \frac18
\end{equation}
and set $\beta_k := \alpha_{k_0 - 1 + k}$, $k = 0, 1, \dots$,
\begin{eqnarray}\label{epsilonk}
&& \mu_{1, k, j} := \sqrt{\frac{\beta_k}{2(\beta_j + \beta_{j - 1})} - \frac14}\, , \nonumber \\
&& \mu_{1, k} := \mu_{1, k, k} = 
\frac12\, \sqrt{\frac{\beta_k - \beta_{k - 1}}{\beta_k + \beta_{k - 1}}}\, ,  \nonumber \\
&& \mu_{2, k} := \frac{1 - \sqrt{1 - \frac{2\beta_k}{\beta_{k + 1} + \beta_k}}}{2} =
\frac12 - \frac12\, \sqrt{\frac{\beta_{k + 1} - \beta_k}{\beta_{k + 1} + \beta_k}}\, ,  \nonumber \\
&& \epsilon_k := \frac{8\pi}{9}\, \frac{1 - \beta_k}{\beta_k - \beta_{k - 1}} < \frac{\pi}{9}\, , 
\ \ \ k \in \mathbb{N} , \ j = 1, \dots,  k .
\end{eqnarray}
Define $a_2, a_3, \dots$
successively by 
$$
\mu_{1, k} \ln\frac{a_{k + 1}}{a_k} = \pi + \epsilon_k , \ \ \ k \in \mathbb{N} .
$$
Then
\begin{eqnarray}\label{pi8}
&& \pi + \epsilon_j \le \mu_{1, k, j} \ln\frac{a_{j + 1}}{a_j}
< \sqrt{\frac{1}{2(\beta_j + \beta_{j - 1})} - \frac14} \, \ln\frac{a_{j + 1}}{a_j} \nonumber \\
&& = \sqrt{\frac{2 - \beta_j - \beta_{j - 1}}{\beta_j - \beta_{j - 1}}}\, (\pi + \epsilon_j) =
\sqrt{1 + 2\frac{1 - \beta_j}{\beta_j - \beta_{j - 1}}}\, (\pi + \epsilon_j)  \nonumber \\
&& \le \left(1 + \frac{1 - \beta_j}{\beta_j - \beta_{j - 1}}\right) (\pi + \epsilon_j) =
\pi + \epsilon_j + \pi\, \frac{1 - \beta_j}{\beta_j - \beta_{j - 1}}\\
&& + \epsilon_j\, \frac{1 - \beta_j}{\beta_j - \beta_{j - 1}} 
< \pi + 2\pi\, \frac{1 - \beta_j}{\beta_j - \beta_{j - 1}}\, , \ \ \ j = 1, \dots, k . \nonumber 
\end{eqnarray}
Let
$$
G(t) := \frac{1}{2(\beta_k + \beta_{k - 1})\, t^2}\, ,  \ \ \ a_k < t < a_{k + 1} , \ k \in \mathbb{N} .
$$
Similarly to part i) of the proof, define
\begin{eqnarray*}
&& u_{1, k}(t) := \left\{\begin{array}{ll} 
\rho_{1, 1} a_1^{1/2} \sin(\mu_{1, k, 1} \ln a_1 + \delta_{1, 1})  ,  & \  t < a_1  , \\  \\
\rho_{1, j} t^{1/2} \sin(\mu_{1, k, j} \ln t + \delta_{1, j})  ,  & \  a_j < t < a_{j + 1} , \\ 
& \ j = 1, \dots,  k, \\ \\
t^{\mu_{2, k}} , & \ t > a_{k + 1} , 
\end{array}\right. \\ \\
&& u_{2, k}(t) := \left\{\begin{array}{ll} 
0 ,  & \  t < a_1  , \\ \\
t^{1/2} \sin(\mu_{1, k, 1} (\ln t - \ln a_1))  ,  & \  a_1 < t < a_2 , \\ \\
\rho_{2, j} t^{1/2} \sin(\mu_{1, k, j} \ln t + \delta_{2, j})  ,  & \  a_j < t < a_{j + 1} , \\ 
& \ j = 2, \dots,  k, \\ \\
\rho_{2, k} a_{k + 1}^{1/2 - \mu_{2, k}} \sin\left(\mu_{1, k} \ln a_{k + 1} + 
\delta_{2, k}\right)\,  t^{\mu_{2, k}} , & \  t > a_{k + 1} , 
\end{array}\right. 
\end{eqnarray*}
where the constants $\rho_{l, j} , \delta_{l, j} \in \mathbb{R}$, $l = 1, 2$ are chosen in such a way 
that 
\begin{eqnarray*}
&& u_{1, k}(a_j - 0) = u_{1, k}(a_j + 0),  \ u'_{1, k}(a_j - 0) = u'_{1, k}(a_j + 0) ,   \ j = 2, \dots, k +1 , \\
&& u_{2, k}(a_j - 0) = u_{2, k}(a_j + 0),  \ u'_{2, k}(a_j - 0) = u'_{2, k}(a_j + 0) ,   \ j = 2, \dots, k .
\end{eqnarray*}
If $k = 1$, one needs to define $u_{2, k}$ in a slightly different way, which is closer to the definition
in part i):
$$
u_{2, 1}(t) := \left\{\begin{array}{ll} 
0 ,  & \  t < a_1  , \\ 
t^{1/2} \sin(\mu_{1, 1} (\ln t - \ln a_1))  ,  & \  a_1 < t < a_2 , \\
a_2^{1/2 - \mu_{2, 1}} \sin\left(\mu_{1, 1} \ln\frac{a_2}{a_1}\right)\,  t^{\mu_{2, 1}} , & \ t > a_2 .
\end{array}\right. 
$$
It is easy to see that
$u_{1, k} \in \mathcal{H}_1$ and $u_{2, k} \in \mathcal{H}_1(a_1, \infty)$ are solutions of
the equation $u'' + \beta_k G u = 0$ on the intervals 
$(-\infty, a_1)$,  $(a_j, a_{j + 1})$, $j = 1, \dots, k$, and $(a_{k + 1}, +\infty)$,
and that $u_{1, k}(a_1 - 0) = u_{1, k}(a_1 + 0)$, $u_{2, k}(a_{k + 1} - 0) = u_{2, k}(a_{k + 1} + 0)$.

Exactly as in part i), any $u \in \mathcal{H}_1$ admits a unique representation
$$
u = d_1 u_{1, k} + d_2 u_{2, k} + u_0 , \ \ \ d_1, d_2 \in \mathbb{C} , \ 
u_0 \in \mathcal{H}_1(a_1, a_{k + 1}) ,
$$ 
and one has
\begin{equation}\label{d1d2}
\mathcal{E}_{\mathcal{H}_1, \beta_k G}[u]  = 
|d_1|^2 \mathcal{E}_{\mathcal{H}_1, \beta_k G} [u_{1, k}] 
+ |d_2|^2 \mathcal{E}_{\mathcal{H}_1, \beta_k G}[u_{2, k}] +
\mathcal{E}_{\mathcal{H}_1, \beta_k G}[u_0] , 
\end{equation}
\begin{eqnarray*}
&& \mathcal{E}_{\mathcal{H}_1, \beta_k G}[u_{1, k}] = \\
&& -\rho_{1, 1}^2 \sin^2(\mu_{1, k, 1} \ln a_1 + \delta_{1, 1}) 
\left(\frac12 + \mu_{1, k, 1} \cot(\mu_{1, k, 1} \ln a_1 + \delta_{1, 1})\right) ,  \\
&& \mathcal{E}_{\mathcal{H}_1, \beta_k G}[u_{2, k}] = 
\rho_{2, k}^2 \sin^2(\mu_{1, k} \ln a_{k + 1} + \delta_{2, k}) 
\Big(\frac12 -  \mu_{2, k}   \\
&&  \ \ \ \ \  + \mu_{1, k} \cot(\mu_{1, k} \ln a_{k + 1} + \delta_{2, k})\Big) , 
\ \ \ k > 1 ,  \\
&& \mathcal{E}_{\mathcal{H}_1, \beta_1 G}[u_{2, 1}] =   \\
&& \sin^2\left(\mu_{1, 1} \ln\frac{a_2}{a_1}\right) 
\left(\frac12 - \mu_{2, 1} + \mu_{1, 1} \cot\left(\mu_{1, 1} \ln\frac{a_2}{a_1}\right)\right) . 
\end{eqnarray*}
Similarly to part i), it follows from \eqref{epsilonk} that 
$\mathcal{E}_{\mathcal{H}_1, \beta_1 G}[u_{2, 1}] > 0$. Let us show that 
$\mathcal{E}_{\mathcal{H}_1, \beta_k G}[u_{2, k}] > 0$ for $k > 1$ as well.
It follows from \eqref{18} and \eqref{pi8} that $u_{2, k}$ has exactly one zero in
$(a_1, a_2)$ and that
$$
\pi + \epsilon_1 < \mu_{1, k, 1} \ln\frac{a_2}{a_1} < \pi + 2\pi\, \frac{1 - \beta_1}{\beta_j - \beta_0} <
\pi + \frac{\pi}{4}\, .
$$
The condition
$$
\frac{u'_{2, k}(a_2 - 0)}{u_{2, k}(a_2 - 0)} = \frac{u'_{2, k}(a_2 + 0)}{u_{2, k}(a_2 + 0)}
$$
is equivalent to 
$$
\mu_{1, k, 1} \cot\left(\mu_{1, k, 1} \ln\frac{a_2}{a_1}\right) = 
\mu_{1, k, 2} \cot(\mu_{1, k, 2} \ln a_2 + \delta_{2, 2}) .
$$
Since $\mu_{1, k, 1} > \mu_{1, k, 2}$, we get
$$
\pi < \mu_{1, k, 2} \ln a_2 + \delta_{2, 2} + m\pi <
\mu_{1, k, 1} \ln\frac{a_2}{a_1} < \pi + 2\pi\, \frac{1 - \beta_1}{\beta_1 - \beta_0} 
$$
for some $m \in \mathbb{Z}$. Using \eqref{18} and \eqref{pi8} again, we see that
$u_{2, k}$ has exactly one zero in $(a_2, a_3)$ and that
$$
2\pi < \mu_{1, k, 2} \ln a_3 + \delta_{2, 2} + m\pi < 2\pi + 
2\pi\, \left(\frac{1 - \beta_1}{\beta_1 - \beta_0} + \frac{1 - \beta_2}{\beta_2 - \beta_1}\right) <
2\pi + \frac{\pi}{4}\, .
$$
Continuing the above argument, we show that $u_{2, k}$ has exactly one zero in 
each interval $(a_j, a_{j + 1})$, $j = 1, \dots, k$, and that
$$
k \pi < \mu_{1, k} \ln a_{k + 1} + \delta_{2, k} + n\pi < k \pi + 
2\pi\, \sum_{j = 1}^k \frac{1 - \beta_j}{\beta_j - \beta_{j - 1}}  <
k \pi + \frac{\pi}{4}
$$
for some $n \in \mathbb{Z}$. This inequality implies that 
$\mathcal{E}_{\mathcal{H}_1, \beta_k G}[u_{2, k}] > 0$.

Our next task is to show that $\mathcal{E}_{\mathcal{H}_1, \beta_k G}[u_{1, k}] < 0$.
Suppose the contrary: $\mathcal{E}_{\mathcal{H}_1, \beta_k G}[u_{1, k}] \ge 0$.
Then there exists $\ell \in \mathbb{Z}$ such that
$$
\omega_k := \mbox{arccot} \left(-\frac1{2\mu_{1, k, 1}} \right) \le 
\mu_{1, k, 1} \ln a_1 + \delta_{1, 1} + \ell\pi \le \pi .
$$
It follows from \eqref{18} and \eqref{pi8} that
$$
\omega_k + \pi  < 
\mu_{1, k, 1} \ln a_2 + \delta_{1, 1} + \ell\pi <  2\pi + 2\pi\, \frac{1 - \beta_1}{\beta_j - \beta_0} \, ,
$$
and one shows as above that
$$
\omega_k + \pi  < 
\mu_{1, k, 2} \ln a_2 + \delta_{1, 2} + m \pi <  2\pi + 2\pi\, \frac{1 - \beta_1}{\beta_j - \beta_0}
$$
for some $m \in \mathbb{Z}$. Continuing as above, one gets
$$
\omega_k + k \pi < \mu_{1, k} \ln a_{k + 1} + \delta_{1, k} + n\pi < (k + 1) \pi + 
2\pi\, \sum_{j = 1}^k \frac{1 - \beta_j}{\beta_j - \beta_{j - 1}}  <
(k + 1) \pi + \frac{\pi}{4}
$$
for some $n \in \mathbb{Z}$. Then
\begin{eqnarray*}
&& \mbox{either } \ \ \frac12 + \mu_{1, k} \cot(\mu_{1, k} \ln a_{k + 1}  + \delta_{1, k}) < 0 \\
&& \mbox{or } \ \ \frac12 + \mu_{1, k} \cot(\mu_{1, k} \ln a_{k + 1}  + \delta_{1, k}) >  \frac12\, .
\end{eqnarray*}
On the other hand, the condition
$$
\frac{u'_{1, k}(a_{k + 1} - 0)}{u_{1, k}(a_{k + 1} - 0)} = 
\frac{u'_{1, k}(a_{k + 1} + 0)}{u_{1, k}(a_{k + 1} + 0)}
$$
is equivalent to 
$$
\frac12 + \mu_{1, k} \cot(\mu_{1, k} \ln a_{k + 1}  + \delta_{1, k}) = \mu_{2, k} \in (0, 1/2) .
$$
The obtained contradiction shows that $\mathcal{E}_{\mathcal{H}_1, \beta_k G}[u_{1, k}] < 0$.
Now it follows from \eqref{d1d2} that 
\begin{eqnarray*}
&& N_- (\mathcal{E}_{\mathcal{H}_1, \beta_k G}) = 1 + 
N_- (\mathcal{E}_{\mathcal{H}_1(a_1, a_{k + 1}), \beta_k G}) \\
&& = 1 + N_- \left(\mathcal{E}^{a_1, a_{k + 1}}_{\beta_k G}\right) + 
N_- \left(\mathcal{E}_{\mathcal{H}_1(a_{k + 1}, \infty), \beta_k G}\right) = 
1 + N_- \left(\mathcal{E}^{a_1, a_{k + 1}}_{\beta_k G}\right)  ,
\end{eqnarray*}
where the last equality follows from Hardy's inequality. It is easy to see that
$$
N_- \left(\mathcal{E}^{a_1, a_{k + 1}}_{\beta_k G}\right) \ge 
\sum_{j = 1}^k N_- \left(\mathcal{E}^{a_j, a_{j + 1}}_{\beta_k G}\right) = k
$$
(see \eqref{pi8} and Lemma \ref{Dirichlet}). If 
$N_- \left(\mathcal{E}^{a_1, a_{k + 1}}_{\beta_k G}\right) > k$, then there exists 
$\lambda < 0$ and a nontrivial solution $u \in \mathring{W}^1_2\big((a_1, a_{k + 1})\big)$
of $u'' + (\beta_k G + \lambda) u = 0$
that has $k$ zeros in $(a_1, a_{k + 1})$ (see, e.g., \cite[Ch. I, Theorem 3.3]{LeSa} or 
\cite[Theorem 13.2]{Wei}). Then $u_{2, k}$ has to have at least $k + 1$ zeros in 
$(a_1, a_{k + 1})$ (see \cite[Ch. I, Theorem 3.1]{LeSa} or \cite[Theorem 13.3]{Wei}). 
On the other hand, we have shown that $u_{2, k}$ has exactly $k$ zeros in 
$(a_1, a_{k + 1})$. Hence $N_- \left(\mathcal{E}^{a_1, a_{k + 1}}_{\beta_k G}\right)$
cannot be larger than $k$, i.e. $N_- \left(\mathcal{E}^{a_1, a_{k + 1}}_{\beta_k G}\right) = k$,
and $N_- (\mathcal{E}_{\mathcal{H}_1, \beta_k G}) = k + 1$. Finally,
\begin{eqnarray*}
N_- (\mathcal{E}_{\mathcal{H}_1, \alpha_{k + 1} G}) - 
N_- (\mathcal{E}_{\mathcal{H}_1, \alpha_k G}) = 
N_- (\mathcal{E}_{\mathcal{H}_1, \beta_{k - k_0 + 2} G}) - 
N_- (\mathcal{E}_{\mathcal{H}_1, \beta_{k - k_0 + 1} G}) \\
= 1, \ \ \  \forall k \ge k_0 .
\end{eqnarray*}

\end{proof}

\begin{lemma}\label{Dirichlet} 
The following equality holds for the form \eqref{Wab}
$$
N_- \left(\mathcal{E}^{a,b}_\beta\right) = 
\left\{\begin{array}{ll} 
0  ,  & \  \beta \le \ \frac14 + \left(\frac{\pi}{\ln\frac{b}{a}}\right)^2 , \\ \\
N  ,  & \  \frac14 + \left(\frac{N\pi}{\ln\frac{b}{a}}\right)^2 < \beta \le
\frac14 + \left(\frac{(N + 1)\pi}{\ln\frac{b}{a}}\right)^2 , \ \ \ N \in \mathbb{N} .
\end{array}\right.
$$
\end{lemma}
\begin{proof}
$ \mathcal{E}^{a,b}_\beta$ is the quadratic form of the self-adjoint operator 
$A_\beta := -\frac{d^2}{dt^2} - \frac{\beta}{t^2}$ on $L_2([a, b])$ with the domain
$W^2_2([a, b])\cap\mathring{W}^1_2\big((a, b)\big)$. The spectrum of $A_\beta$ is
discrete and consists of simple eigenvalues. It follows from Hardy's inequality 
that the eigenvalues are positive for $\beta \le 1/4$.   
As $\beta$ increases, the eigenvalues move continuously (see, e.g., \cite[Theorem V.4.10]{Kat})
to the left, and $N_- \left(\mathcal{E}^{a,b}_\beta\right)$ increases by one when an eigenvalue
crosses 0. This happens for those values of $\beta$ for which 0 is an eigenvalue of $A_\beta$, i.e.
when
$$
-u'' - \frac{\beta}{t^2}\, u = 0 , \ \ \ u(a) = 0 = u(b) , \ \ \ \left(\beta > \frac14\right)
$$ 
has a nontrivial solution. The change of the independent variable $s = \ln t$ reduces this
equation to an ODE with constant coefficients, and one finds that the solutions of the original
equation that satisfy the condition $u(a) = 0$ are multiples of
$$
t^{1/2} \sin\left(\sqrt{\beta - \frac14}\, \big(\ln t - \ln a\big)\right) .
$$ 
The latter satisfies the condition $u(b) = 0$ if and only if
$$
\sqrt{\beta - \frac14}\, \ln\frac{b}{a} = N\pi , \ \ \mbox{ i.e. } \ \
\beta = \frac14 + \left(\frac{N\pi}{\ln\frac{b}{a}}\right)^2 ,  \ \ \ N \in \mathbb{N} .
$$
\end{proof}

\textsc{Acknowledgement.}
This work has been strongly influenced by Michael Solomyak and I am very grateful for his advice.
I am also grateful to Alexander Grigor'yan, Ari Laptev, and Yuri Safarov 
for very helpful comments and 
suggestions. A part of the paper was written during my stay at Institut Mittag-Leffler. The unique working
conditions provided by its Research in Peace programme are gratefully acknowledged.

\end{document}